\documentclass{amsart}
\usepackage{fancyhdr}
\usepackage{geometry}
\usepackage{enumitem}
\usepackage{mathrsfs}
\usepackage{amsmath}
\usepackage{amssymb}
\usepackage{latexsym}
\usepackage{comment}
\usepackage{amscd}
\usepackage{wasysym}  
\usepackage{tikz}
\usetikzlibrary{matrix,arrows}
\usepackage{tikz-cd}
\usepackage{appendix}
\usepackage[nice]{nicefrac}
\usepackage{fix-cm} 
\usepackage{graphicx}
\usepackage{marvosym}
\usepackage[all,cmtip]{xy}
\usepackage{tabularx}
\usepackage{xcolor,colortbl}
\usepackage{cancel}
\usepackage{soul}
\usepackage{circledsteps}
\usepackage{rotating}
\usepackage{stmaryrd}

\usepackage[pagebackref,hyperindex,linktocpage=true]{hyperref}
\hypersetup{
    colorlinks,
    linkcolor={blue!40!black},
    citecolor={blue!40!black},
    urlcolor={blue!40!black}
}

\usepackage{color}

\usetikzlibrary{decorations.markings}

\makeatletter
\tikzcdset{
  open/.code     = {\tikzcdset{hook, circled};},
  closed/.code   = {\tikzcdset{hook, slashed};},
  open'/.code    = {\tikzcdset{hook', circled};},
  closed'/.code  = {\tikzcdset{hook', slashed};},
  circled/.code  = {\tikzcdset{markwith = {\draw (0,0) circle (.375ex);}};},
  slashed/.code  = {\tikzcdset{markwith = {\draw[-] (-.4ex,-.4ex) -- (.4ex,.4ex);}};},
  markwith/.code ={
    \pgfutil@ifundefined%
    {tikz@library@decorations.markings@loaded}%
    {\pgfutil@packageerror{tikz-cd}{You need to say %
      \string\usetikzlibrary{decorations.markings} to use arrows with markings}{}}{}%
    \pgfkeysalso{/tikz/postaction = {
      /tikz/decorate,
      /tikz/decoration={markings, mark = at position 0.5 with {#1}}}
    }
  },
}
\makeatother

\renewcommand{\geq}{\geqslant}
\renewcommand{\leq}{\leqslant}

\renewcommand{\le}{\leqslant}

\newtheorem{thm}{Theorem}[section]

\newtheorem{propo}[thm]{Proposition}

\newtheorem{lem}[thm]{Lemma}
\newtheorem{sublem}[thm]{Sublemma}
\newtheorem{lem-def}[thm]{Lemma-Definition}
\newtheorem{thm-def}[thm]{Theorem-Definition}
\newtheorem{cor}[thm]{Corollary}
\newtheorem{conject}[thm]{Conjecture}
\newtheorem{propert}[thm]{Properties}
\newtheorem{observ}[thm]{Observation}
\newtheorem{assum}[thm]{Assumption}

\newtheorem{fac}[thm]{Fact}
\newtheorem{ex}[thm]{Example}

\theoremstyle{definition}
\newtheorem*{ack}{Acknowledgement}

\newtheorem{rmk}[thm]{Remark}
\newtheorem{dfn}[thm]{Definition}
\newtheorem{quest}[thm]{Question}
\newtheorem{expec}[thm]{Expectation}
\newtheorem*{abs}{Abstract}

\numberwithin{equation}{section}
\def\ov{\overline}
\def\N{{\mathbb N}}
\DeclareMathOperator{\Mon}{Mon}
\DeclareMathOperator{\supp}{supp}

\newcommand{\nc}{\newcommand}

\nc{\theo}{\begin{thm}} \nc{\xtheo}{\end{thm}}
\nc{\prop}{\begin{propo}} \nc{\xprop}{\end{propo}}
\nc{\lemm}{\begin{lem}} \nc{\xlemm}{\end{lem}}
\nc{\sublemm}{\begin{sublem}} \nc{\xsublemm}{\end{sublem}}
\nc{\lemmdefi}{\begin{lem-def}} \nc{\xlemmdefi}{\end{lem-def}}
\nc{\coro}{\begin{cor}} \nc{\xcoro}{\end{cor}}
\nc{\conj}{\begin{conject}} \nc{\xconj}{\end{conject}}
\nc{\proper}{\begin{propert}} \nc{\xproper}{\end{propert}}
\nc{\obse}{\begin{observ}} \nc{\xobse}{\end{observ}}
\nc{\ques}{\begin{quest}} \nc{\xques}{\end{quest}}

\nc{\fact}{\begin{fac}} \nc{\xfact}{\end{fac}}
\nc{\expe}{\begin{expec}} \nc{\xexpe}{\end{expec}}

\nc{\ackn}{\begin{ack}} \nc{\xackn}{\end{ack}}
\nc{\exam}{\begin{ex}} \nc{\xexam}{\end{ex}}
\nc{\rema}{\begin{rmk}} \nc{\xrema}{\end{rmk}}
\nc{\defi}{\begin{dfn}} \nc{\xdefi}{\end{dfn}}
\nc{\abst}{\begin{abs}} \nc{\xabst}{\end{abs}}

\nc{\pf}{\begin{proof}} \nc{\xpf}{\end{proof}}

\nc{\on}{\operatorname}
\nc{\fraka}{{\mathfrak a}} \nc{\bba}{{\mathbf a}}
\nc{\frakb}{{\mathfrak b}}
\nc{\frakc}{{\mathfrak c}}
\nc{\frakd}{{\mathfrak d}}
\nc{\frake}{{\mathfrak e}}
\nc{\frakf}{{\mathfrak f}}
\nc{\frakg}{{\mathfrak g}}
\nc{\frakh}{{\mathfrak h}}
\nc{\fraki}{{\mathfrak i}}
\nc{\frakj}{{\mathfrak j}}
\nc{\frakk}{{\mathfrak k}}
\nc{\frakl}{{\mathfrak l}}
\nc{\frakm}{{\mathfrak m}}
\nc{\frakn}{{\mathfrak n}}
\nc{\frako}{{\mathfrak o}}
\nc{\frakp}{{\mathfrak p}}
\nc{\frakq}{{\mathfrak q}}
\nc{\frakr}{{\mathfrak r}}
\nc{\fraks}{{\mathfrak s}}
\nc{\frakt}{{\mathfrak t}}
\nc{\fraku}{{\mathfrak u}}
\nc{\frakv}{{\mathfrak v}}
\nc{\frakw}{{\mathfrak w}}
\nc{\frakx}{{\mathfrak x}}
\nc{\fraky}{{\mathfrak y}}
\nc{\frakz}{{\mathfrak z}}
\nc{\frakA}{{\mathfrak A}}
\nc{\frakB}{{\mathfrak B}}
\nc{\frakC}{{\mathfrak C}}
\nc{\frakD}{{\mathfrak D}}
\nc{\frakE}{{\mathfrak E}}
\nc{\frakF}{{\mathfrak F}}
\nc{\frakG}{{\mathfrak G}}
\nc{\frakH}{{\mathfrak H}}
\nc{\frakI}{{\mathfrak I}}
\nc{\frakJ}{{\mathfrak J}}
\nc{\frakK}{{\mathfrak K}}
\nc{\frakL}{{\mathfrak L}}
\nc{\frakM}{{\mathfrak M}}
\nc{\frakN}{{\mathfrak N}}
\nc{\frakO}{{\mathfrak O}}
\nc{\frakP}{{\mathfrak P}}
\nc{\frakQ}{{\mathfrak Q}}
\nc{\frakR}{{\mathfrak R}}
\nc{\frakS}{{\mathfrak S}}
\nc{\frakT}{{\mathfrak T}}
\nc{\frakU}{{\mathfrak U}}
\nc{\frakV}{{\mathfrak V}}
\nc{\frakW}{{\mathfrak W}}
\nc{\frakX}{{\mathfrak X}}
\nc{\frakY}{{\mathfrak Y}}
\nc{\frakZ}{{\mathfrak Z}}
\nc{\bbA}{{\mathbb A}}
\nc{\bbB}{{\mathbb B}}
\nc{\bbC}{{\mathbb C}}
\nc{\bbD}{{\mathbb D}}
\nc{\bbE}{{\mathbb E}}
\nc{\bbF}{{\mathbb F}} \nc{\bbf}{{\mathbf f}}
\nc{\bbG}{{\mathbb G}}
\nc{\bbH}{{\mathbb H}}
\nc{\bbI}{{\mathbb I}}
\nc{\bbJ}{{\mathbb J}}
\nc{\bbK}{{\mathbb K}}
\nc{\bbL}{{\mathbb L}}
\nc{\bbM}{{\mathbb M}}
\nc{\bbN}{{\mathbb N}}
\nc{\bbO}{{\mathbb O}}
\nc{\bbP}{{\mathbb P}}
\nc{\bbQ}{{\mathbb Q}}
\nc{\bbR}{{\mathbb R}}
\nc{\bbS}{{\mathbb S}}
\nc{\bbT}{{\mathbb T}}
\nc{\bbU}{{\mathbb U}}
\nc{\bbV}{{\mathbb V}}
\nc{\bbW}{{\mathbb W}}
\nc{\bbX}{{\mathbb X}}
\nc{\bbY}{{\mathbb Y}}
\nc{\bbZ}{{\mathbb Z}}
\nc{\calA}{{\mathcal A}}
\nc{\calB}{{\mathcal B}}
\nc{\calC}{{\mathcal C}}
\nc{\calD}{{\mathcal D}}
\nc{\calE}{{\mathcal E}}
\nc{\calF}{{\mathcal F}}
\nc{\calG}{{\mathcal G}}
\nc{\calH}{{\mathcal H}}
\nc{\calI}{{\mathcal I}}
\nc{\calJ}{{\mathcal J}}
\nc{\calK}{{\mathcal K}}
\nc{\calL}{{\mathcal L}}
\nc{\calM}{{\mathcal M}}
\nc{\calN}{{\mathcal N}}
\nc{\calO}{{\mathcal O}}
\nc{\calP}{{\mathcal P}}
\nc{\calQ}{{\mathcal Q}}
\nc{\calR}{{\mathcal R}}
\nc{\calS}{{\mathcal S}}
\nc{\calT}{{\mathcal T}}
\nc{\calU}{{\mathcal U}}
\nc{\calV}{{\mathcal V}}
\nc{\calW}{{\mathcal W}}
\nc{\calX}{{\mathcal X}}
\nc{\calY}{{\mathcal Y}}
\nc{\calZ}{{\mathcal Z}}

\nc{\scrA}{{\mathscr A}}
\nc{\scrE}{{\mathscr E}}
\nc{\scrR}{{\mathscr R}}

\nc{\Bmu}{\mbox{$\raisebox{-0.59ex}{$l$}\hspace{-0.18em}\mu\hspace{-0.88em}\raisebox{-0.98ex}{\scalebox{2}{$\color{white}.$}}\hspace{-0.416em}\raisebox{+0.88ex}{$\color{white}.$}\hspace{0.46em}$}{}}

\nc{\bnu}{{\bar{ \nu}}}

\nc{\olO}{\bar{\calO}}

\nc{\al}{{\alpha}} 
\nc{\be}{{\beta}}
\nc{\ga}{{\gamma}} \nc{\Ga}{{\Gamma}}
 \nc{\hGa}{\hat{\Gamma}}
\nc{\ve}{{\varepsilon}} 
\nc{\la}{{\lambda}} \nc{\La}{{\Lambda}}
\nc{\om}{\omega} \nc{\Om}{\Omega} 
\nc{\sig}{{\sigma}} \nc{\Sig}{{\Sigma}}

\nc{\tnb}{\psi_{\rm tame}}
\nc{\oM}{\overline{{M}}}
\nc{\op}{{\on{op}}}
\nc{\ad}{{\on{ad}}}
\nc{\alg}{{\on{alg}}}
\nc{\Ad}{{\on{Ad}}}
\nc{\Adm}{{\on{Adm}}} \nc{\aff}{{\on{aff}}}
\nc{\Aut}{{\on{Aut}}}
\nc{\Bun}{{\on{Bun}}}
\nc{\cha}{{\on{char}}}
\nc{\der}{{\on{der}}}
\nc{\Der}{{\on{Der}}}
\nc{\diag}{{\on{diag}}}
\nc{\End}{{\on{End}}}
\nc{\Fl}{{\calF\!\ell}}
\nc{\Tr}{{\on{Transp}}}
\nc{\TR}{{\calT\!\calR}}
\nc{\Gal}{{\on{Gal}}}
\nc{\Gr}{{\on{Gr}}}
\nc{\rH}{{\on{H}}}
\nc{\Hom}{{\on{Hom}}}
\nc{\IC}{{\on{IC}}}
\nc{\id}{{\on{id}}}
\nc{\Id}{{\on{Id}}}
\nc{\ind}{{\on{ind}}}
\nc{\Ind}{{\on{Ind}}}
\nc{\Lie}{{\on{Lie}}}
\nc{\Pic}{{\on{Pic}}}
\nc{\pr}{{\on{pr}}}
\nc{\Res}{{\on{Res}}}
\nc{\res}{{\on{res}}} \nc{\Sat}{{\on{Sat}}}
\nc{\s}{{\on{sc}}}
\nc{\drv}{{\on{der}}}
\nc{\sgn}{{\on{sgn}}}
\nc{\Spec}{{\on{Spec}}}\nc{\Spf}{\on{Spf}} 
\nc{\Sph}{\on{Sph}}
\nc{\St}{{\on{St}}}
\nc{\tr}{{\on{tr}}}
\nc{\Mod}{{\mathrm{-Mod}}}
\nc{\Hilb}{{\on{Hilb}}} 
\nc{\Ext}{{\on{Ext}}} 
\nc{\vs}{{\on{Vec}}}
\nc{\ev}{{\on{ev}}}
\nc{\nO}{{\breve{\calO}}}
\nc{\tS}{{\tilde{S}}}
\nc{\spe}{{\on{sp}}}
\nc{\loc}{{\on{loc}}}
\nc{\Sym}{{\on{Sym}}}
\nc{\Cone}{{\on{C}}}
\nc{\syn}{{\on{syn}}}
\nc{\reg}{{\on{reg}}}
\nc{\colim}{{\on{colim}}}
\nc{\Norm}{{\on{N}}}

\nc{\nscrR}{{\mathscr{R}^{\on{nr}}}}

\nc{\GL}{{\on{GL}}}
\nc{\U}{{\on{U}}}
\nc{\Gl}{\on{Gl}} 
\nc{\GSp}{{\on{GSp}}}
\nc{\gl}{{\frakg\frakl}}
\nc{\SL}{{\on{SL}}} 
\nc{\SU}{{\on{SU}}} 
\nc{\SO}{{\on{SO}}}
\nc{\PGL}{{\on{PGL}}}

\nc{\Conv}{{\on{Conv}}}
\nc{\Rep}{{\on{Rep}}}
\nc{\Dom}{{\on{Dom}}}
\nc{\act}{{\on{act}}}
\nc{\nr}{{\on{nr}}}
\nc{\ctf}{{\on{ctf}}}

\nc{\str}{{\on{-}}} 
\nc{\os}{{\bar{s}}}
\nc{\oeta}{{\bar{\eta}}}

\nc{\Def}{\mathrm{Def}}

\nc{\hookto}{\hookrightarrow}
\nc{\longto}{\longrightarrow}
\nc{\leftto}{\leftarrow}
\nc{\onto}{\twoheadrightarrow}
\nc{\lonto}{\twoheadleftarrow}
\newcommand*\isomto{%
  \renewcommand{\arraystretch}{0.1}
  \begin{array}[b]{c} {}_{\sim} \\ \longrightarrow \end{array}%
}

\nc{\uG}{{\underline{G}}}
\nc{\uA}{{\underline{A}}}
\nc{\uS}{{\underline{S}}}
\nc{\uT}{{\underline{T}}}
\nc{\uM}{{\underline{M}}}
\nc{\uP}{{\underline{P}}}
\nc{\uB}{{\underline{B}}}
\nc{\uN}{{\underline{N}}}

\nc{\ucG}{{\underline{\calG}}}
\nc{\ucA}{{\underline{\calA}}}
\nc{\ucS}{{\underline{\calS}}}
\nc{\ucT}{{\underline{\calT}}}
\nc{\ucalM}{{\underline{\calM}}}
\nc{\ucP}{{\underline{\calP}}}
\nc{\ucalN}{{\underline{\calN}}}

\nc{\bF}{{\breve{F}}}

\nc{\oFl}{{\overline{\Fl}}} 
\nc{\bU}{{\overline{U}}}
\nc{\tGr}{{\tilde{\Gr}}}
\nc{\cGr}{\calG\! r}
\nc{\oGr}{\overline{\on{Gr}}} 
\nc{\ocGr}{\overline{\calG\! r}}
\nc{\co}{{\colon}}
\nc{\sch}[1]{(Sch/{#1})}
\nc{\HypLoc}[1]{HypLoc({#1})}

\nc{\ohtimes}{\stackrel{!}{\otimes}}
\nc{\boxtilde}{\widetilde{\boxtimes}}
\nc{\vstar}{{\varhexstar}}

\nc{\Div}{\on{Div}}
\nc{\Sht}{\on{Sht}}
\nc{\Frob}{\on{Frob}}

\nc{\Pan}{\on{Pan}}

\nc{\x}{\times}
\nc{\bsl}{\backslash}
\nc{\algQl}{{\bar{\bbQ}_\ell}}
\nc{\sF}{{\bar{F}}}
\nc{\nF}{{\breve{F}}}
\nc{\nW}{{W^{\on{nr}}}}
\nc{\sk}{{\bar{k}}}
\nc{\cont}{\on{c}}
\nc{\Supp}{\on{Supp}}
\nc{\blt}{\bullet}  
\nc{\dom}{\on{dom}}
\nc{\scon}{{\on{sc}}} 
\nc{\Affine}{\on{Aff}} 
\nc{\nscrA}{\mathscr{A}^{\on{nr}}} 
\nc{\nfraka}{{\bbf^{\on{nr}}}}
\nc{\ran}{{\rangle}}
\nc{\lan}{{\langle}}
\nc{\bk}{{\bar{k}}}
\nc{\tF}{{\tilde{F}}}
\nc{\sS}{{\bar{S}}}
\nc{\LG}{{^\text{L}\hspace{-0.04cm}G}}
\nc{\LL}{{^\text{L}\hspace{-0.07cm}L}}
\nc{\et}{{\text{\rm \'et}}}
\nc{\inv}{{\on{inv}}}
\nc{\Hecke}{{\on{Hecke}}}
\nc{\Isom}{{\on{Isom}}}
\nc{\oSht}{{\overline{\on{Sht}}}}
\nc{\umu}{{\underline \mu}}
\nc{\AIJ}{{\calO_X[{\scriptstyle{\calI\over \calJ}}]}}
\nc{\Proj}{{\on{Proj}}}
\nc{\Bl}{{\on{Bl}}}
\nc{\Stab}{{\on{Stab}}}
\nc{\cl}{{\on{cl}}}

\nc{\Pos}{{\on{Pos}}}
\nc{\Sets}{{\on{Sets}}}
\nc{\AffSch}{{\on{AffSch}}}
\nc{\Groups}{{\on{Groups}}}
\nc{\Gpds}{{\on{Groupoids}}}
\nc{\Sch}{{\on{Sch}}}
\nc{\fl}{{\on{flat}}}

\nc{\pot}[1]{ [\hspace{-0,5mm}[ {#1} ]\hspace{-0,5mm}] }
\nc{\rpot}[1]{ (\hspace{-0,7mm}( {#1} )\hspace{-0,7mm}) }

\nc{\defined}{\hspace{0.1cm}\stackrel{\text{\tiny \rm def}}{=}\hspace{0.1cm}}

\usepackage{moresize}
\usepackage{graphicx}

\usepackage{xcolor}

\newenvironment{talign*}
 {\csname align*\endcsname}
 {\endalign}

\begin{document}

\title
[Multi-centered deformation spaces ] 
{A polyptych of multi-centered deformation spaces}


\author{Adrien Dubouloz}

\email{adrien.dubouloz@math.cnrs.fr}

\address{
CNRS, Université de Poitiers, LMA, Poitiers, France. \newline
Université Bourgogne Europe, CNRS, IMB UMR 5584, 21000 Dijon, France.
}

\author{Arnaud Mayeux}

\email{mayeux@wisc.edu}

\address{University of Wisconsin–Madison, Madison, Wisconsin 53706, United States of America.}

\maketitle 
\begin{center}
  {\it \large with three appendices, including one joint with} {\it \large Sylvain Brochard\footnote[1]{Institut Montpelliérain Alexander Grothendieck, Université de Montpellier, France}} 
\end{center}

\abs Extending Verdier's deformation space to the normal cone of a closed subscheme and Rost's double deformation space of a pair of nested closed subschemes, we introduce a notion of deformation spaces attached to chains of immersions of arbitrary lengths $n$. One main result, which builds on the formalism of multi-centered dilatations of schemes, is the existence of so-called panelization isomorphisms, which produce under suitable regularity conditions several canonical isomorphisms between a given deformation space of length $n$ and some deformation spaces of smaller lengths. Having these panelization isomorphisms also allows one to give geometric descriptions of the strata --certain restrictions of special interest-- of deformation spaces. 
\tableofcontents

\ackn
 We warmly thank Frédéric Déglise for his inspirational support, which greatly motivated us throughout this work. This project has received funding from the ISF A.M. grant 1577/23 and support from the ANR Grant ``HQ-Diag'' ANR-21-CE40-0015. We are grateful
to Sylvain Brochard and Niels Feld for several discussions and comments.
\xackn 
\section*{Introduction}
 The purpose of this article is to introduce and study a notion of deformation space for chains of immersions of closed subschemes of arbitrary lengths, which recovers as special cases the deformation space to the normal cone of a closed subscheme as defined by Verdier \cite{Ver76} and Rost's double deformation space of a pair of nested closed subschemes \cite{Ro96}. 
 
For a closed subscheme $Y$ of a scheme $X$ with defining ideal sheaf $\mathcal{I}\subset \mathcal{O}_X$, the deformation space $\mathrm{Def}(Y,X)$ is a scheme over $\mathbb{A}^1_X$
defined as the relative spectrum of the $\mathcal{O}_X[t]$-subalgebra 
$$ 
\bigoplus _{n\in \mathbb{Z}}\mathcal{I}^nt^{-n} \subset \mathcal{O}_X[t^{\pm 1}],
$$
where, by convention, $\mathcal{I}^n=\mathcal{O}_X$ for $n<0$. The structure morphism $\sigma:\mathrm{Def}(Y,X)\to \mathbb{A}^1_X$ expresses $\mathrm{Def}(Y,X)$ as the dilatation (also called affine blow-up or affine modification \cite{Du05, MRR20}) of $\mathbb{A}^1_X$ with center consisting of the closed subscheme $\mathbb{A}^1_Y\subset \mathbb{A}^1_X$ and the divisor $D$ equal to the image of the zero section $\{0\}:X\to \mathbb{A}^1_X$, see e.g. \cite{ADO}. The fact that $\mathrm{Def}(Y,X)$ fits into a pair of cartesian diagrams 
$$
\xymatrix{C_{Y/X} \ar[r] \ar[d] & \mathrm{Def}(Y,X) \ar[d] & \mathbb{G}_{m,X} \ar[l] \ar@{=}[d] \\ X \ar[r]^{\{0\}} & \mathbb{A}^1_X & \mathbb{G}_{m,X} \ar[l] }
$$ 
where $C_{Y/X}=\mathrm{Spec}(\bigoplus_{n\geq0} \mathcal{I}^n/\mathcal{I}^{n+1})$ is the normal cone of $Y$ in $X$, plays a central role in the construction of Gysin homomorphisms for regular closed immersions in several cohomology theories, and in particular in the construction of intersection theory within Chow groups. The double deformation space $\mathrm{Def}(Z,Y,X)$ of a pair of nested closed immersions $Z\subset Y\subset X$ was introduced later on by Rost \cite{Ro96} as a scheme over $\mathbb{A}^2_X$ whose purpose was primarily to serve as a tool to verify the associativity of certain  intersection operations on Chow groups.

It was observed in \cite{Ma24d} that in the same way as $\mathrm{Def}(Y,X)$ can be interpreted as a mono-centered dilatation of $\mathbb{A}^1_X$, Rost's double deformation space $\mathrm{Def}(Z,Y,X)$ can be interpreted in a natural way as a double-centered dilatation of $\mathbb{A}^2_X$. Building on this observation, it becomes natural to define a notion of deformation space attached to a chain of nested closed immersions of arbitrary length within the formalism of multi-centered dilatations of schemes of \cite{Ma24d}. As a result, the first part of this paper proposes a general notion of multi-centered deformation space of arbitrary length and studies its basic properties. Building on ideas and techniques inspired by iterated multi-centered dilatations of schemes we derive from universal properties of multi-centered dilatations \cite{MRR20, Ma24d} the existence of canonical morphisms, called \emph{panelization morphisms}, which allow, under appropriate regularity assumptions, to identify a given multi-centered deformation space with several other multi-centered deformation spaces of smaller length. These panelization morphisms and isomorphisms allow an inductive view on the construction and the study of the properties of multi-centered deformation spaces. We use in turn these different complementary presentations of multi-centered deformation spaces to study the geometry of their \emph{strata}, which are particular closed subschemes analogous in our setting to the restrictions of the deformation space and the double deformation space to the coordinate hyperplanes and their intersections.

The article is organized as follows. Section \ref{secpreli} 
provides a quick recollection on (multi-centered) dilatations of rings and schemes following \cite{Ma24d}. Section \ref{secD} presents the definition of our general multi-centered deformation spaces and collects basic results on these. Section \ref{secSliding} is devoted to the construction of the panelization morphisms of multi-centered deformation spaces  (Proposition \ref{theopanelization-mor}) and the identification of suitable regularity assumptions under which these morphisms are isomorphisms (Theorem \ref{theo-iso-panelization} and Theorem  \ref{coroset}).  
Section \ref{secstrata} discusses the notion of strata of multi-centered deformation spaces, the main results there being Proposition \ref{th:vb-strata}
and Proposition  \ref{higher_strata} which provide, again under appropriate regularity assumptions, complementary descriptions of these strata in the form of  multi-centred deformation spaces of smaller lengths. Section \ref{sec:Verdier-Rost} provides a summary and some applications of the results of the previous sections in the special situation of the ``Verdier-Rost deformation space'' $\mathrm{Def}(Z_n,\ldots, Z_0)$ of a collection of closed immersions $Z_n\subset Z_{n-1}\subset \cdots \subset Z_1\subset Z_0$ between schemes smooth over a locally Noetherian base scheme. 

The article is completed by three appendices: Appendix \ref{appendix:vbsym} examines from the viewpoint of multi-centered dilatations some  symmetric variants of multi-centered deformation spaces appearing in \cite{Ro96, Ivorra14, Le26}. Appendix \ref{appendix:vb} provides constructions of vector bundles by means of dilatations  which are in particular used to clarify the existence of non-canonical isomorphisms between certain strata of multi-centered deformation spaces. Appendix \ref{secap}, joint with Sylvain Brochard, is devoted to establishing several results of independent interest concerning products, intersections and sums of ideals generated by monomials in a regular sequence, which play an important role in the proof of Theorem \ref{coroset} but for which we were unable to find a convenient reference.

The definitions of multi-centered deformation spaces and their panels, the kernel formula, and the panelization theorems (Proposition \ref{theopanelization-mor}, Theorem \ref{theo-iso-panelization} and Corollary \ref{KsupS}) have been formalized in Lean 4 by the second author in  \cite{MayZ26} as part of a broader formalization of the framework of multi-centered dilatations of rings and schemes.

 \section{Preliminaries} \label{secpreli}

 \subsection{Closed subschemes, divisors and vector bundles} \label{subsecclo}
 The set of all closed subschemes of a scheme $X$ is denoted $Clo(X)$. If $Y_1 , Y_2$ belong to $Clo(X)$, we write $Y_1 \subset Y_2$ if $Y_1$ is a closed subscheme of $Y_2$. Let $\calM _1$ and $\calM_2$ be the quasi-coherent ideals of $\calO_X$ such that $Y_1= V (\calM _1)$, and $Y_2 = V (\calM _2)$. Then $Y_1+ Y_2$ is defined as $V (\calM _1 \calM_2)$, $Y_1 \cup Y_2$ is defined as $V ( \calM _1 \cap \calM _2)$ and  $Y_1 \cap Y_2$ is defined as $V( \calM_1 + \calM_2)$.
 
 An element $D $ of $Clo (X)$ is called locally principal if its ideal sheaf is locally generated by a single element and an effective Cartier divisor if its ideal sheaf is locally generated by a non-zero-divisor. 
 
 By a vector bundle over a scheme $X$, we mean the relative spectrum $p:\mathbb{V}(\mathcal{E})\to X$ of the symmetric algebra $\mathrm{Sym}^{\bullet}\mathcal{E}$ of a coherent locally free $\mathcal{O}_X$-module $\mathcal{E}$.

 \subsection{Multi-centered dilatations of schemes}\label{subsecdil} 
  We assume the formalism of multi-centered dilatations of schemes as defined and studied in \cite{Ma24d}. For the convenience of the reader, we briefly recall some essential definitions and results.

  A multi-center $\{[Y_i, D_i]\}_{i\in I}$
  on a scheme $X$  is a set of pairs of closed subschemes of $X$ such that each $D_i$, $i\in I$, is locally principal. Following \cite[Definition 3.4]{Ma24d} for a collection $D_I=(D_i)_{i\in I}$ of locally principal closed subschemes of $X$, we denote by $\Sch _{X}^{\{D_i\}_{i \in I} \text{-reg}}$  the full subcategory of schemes over $X$ whose objects are morphisms $f:T \to X$ such that $f^{-1} (D_i)$ is an effective Cartier divisor in $T$ for all $i$. Building on  \cite[Proposition 3.17]{Ma24d}, a multi-centered dilatation of a scheme $X$ can be shortly defined for our purpose through its universal property:

 \begin{thm-def} \label{Def-Multi-centered} The dilatation of a scheme $X$ with multi-center $\{[Y_i, D_i]\}_{i\in I}$ is the  $X$-scheme $\sigma:\Bl \left\{^{D_i}_{Y_i} \right\}_{i \in I} X\to X$ representing the functor $\Sch _{X}^{\{D_i\}_{i\in I} \text{-reg}} \to Sets$ defined by
\begin{equation}
\label{eq:functor}
(f: T\to X) \;\longmapsto\; \begin{cases}\left\{*\right\}, \; \text{if $f^{-1} (D_i)\subset f^{-1} \left( Y_i \right) $ in $Clo \left( T \right) $ for every $i \in I  $};\\ \varnothing,\;\text{else.}\end{cases}
\end{equation}
 \end{thm-def}
  
 If $\#I = 0$, then $\Bl \left\{^{D_i}_{Y_i} \right\}_{i \in I} X=X$. If $I= \{1\}$ is a singleton, then $[Y_1 , D_1]$ is called a mono-center and $\Bl \left\{^{D_1}_{Y_1} \right\} X$ is called a mono-centered dilatation.\footnote{Dilatations associated to mono-centers $[Z,D]$ with $Z \subset D$ were defined and studied in many references in specific settings, including \cite{Du05} and \cite{MRR20} for general schemes. We refer to \cite{Ma25} for a construction of dilatations in a categorical setting and to \cite{DMdS23}  for a survey on algebraic dilatations, including many other references.}

We briefly summarize  from \cite[Subsection 2.1]{Ma24d} the construction of the scheme representing  the functor \eqref{eq:functor}
in the local affine setting. Here, $X=\mathrm{Spec}(A)$ for some ring $A$, $Y_i=\mathrm{Spec}(A/M_i)$ for a collection of ideals $M_i\subset A$ and  $D_i=\mathrm{div}(a_i)$ for some elements $a_i\in A$, $i\in I$. Putting $L_i=(M_i+(a_i))$, we endow the set of pairs $(\ell,a^\nu)$, where $\ell \in \textstyle \prod_{i\in I}  L_i^{\nu_i}$ and $a^{\nu}=\prod_{i\in I} a_i^{\nu_i}$ for some multi-index $\nu=(\nu_i)_{i\in I}\in (\mathbb{Z}_{\geq 0})^I$ of natural numbers, with the equivalence relation defined by $(\ell,a^\nu)\sim (\ell',a^{\nu'})$ if  $(\ell a^{\nu'}-\ell' a^{\nu})a^\beta=0$ in $A$ for some multi-index $\beta\in (\mathbb{Z}_{\geq 0})^I$. We let $\tfrac{\ell}{a^{\nu}}$ be the equivalence class of $(\ell,a^\nu)$ and denote by $A\left[\left\{ \tfrac{M_i }{a_i}\right\}_{i \in I}\right]$ the set of all equivalence classes. The formulas $$ \frac{\ell}{a^{\nu}}+\frac{\ell'}{a^{\nu'}}=\frac{\ell a^{\nu'}+\ell' a^{\nu}}{a^{\nu+\nu'}} \quad \textrm{and} \quad \frac{\ell}{a^\nu} \times \frac{\ell'}{a^{\nu'}}=\frac{\ell\ell'}{a^{\nu+\nu'}}$$ 
define addition and multiplication laws on 
$A\left[\left\{ \tfrac{M_i }{a_i}\right\}_{i \in I}\right]$ which make it into a commutative unital ring. The natural map $\sigma^*:A\to A\left[\left\{ \tfrac{M_i }{a_i}\right\}_{i \in I}\right]$, $a\mapsto \tfrac{a}{1}$ is a ring homomorphism,  called the dilatation homomorphism. By \cite[Corollary 2.16]{Ma24d},
the so-defined $A$-algebra $A\left[\left\{\tfrac{M_i}{a_i}\right\}_{i \in I}\right]$ identifies with the sub-algebra of the localization $A[\{a_i\}_{i \in I}^{-1}]$ of $A$, obtained by adjoining formal inverses $a_i^{-1}$ for all $i \in I$, generated by   $\left\{\tfrac{M_i}{a_i}\right\}_{i \in I}$. 
The \emph{dilatation of $X$ with multi-center $\{[Y_i,D_i]_{i\in I}\}$} is then the spectrum $\Bl \left\{^{D_i}_{Y_i} \right\}_{i \in I} X$ of $A\left[\left\{ \frac{M_i }{a_i}\right\}_{i \in I}\right]$ and the morphism $\sigma:\Bl \left\{^{D_i}_{Y_i} \right\}_{i \in I} X\to X$ corresponding to $\sigma^*$ is called the dilatation morphism. 

\section{Multi-centered deformation spaces} \label{secD} 
Let $(I,\leq )$ be a finite totally ordered set. Every subset $J$ of $I$ is considered as endowed with the total ordering induced from $I$. Moreover, for every element $i\in I$, we let $$J_{\geq i} = \{ j \in J | i \leq j \} \textrm{ and } J_{>i} = \{ j \in J | i <j \}.$$

\subsection{ Definition and universal property}
\defi \label{defdatum}\label{defimultidef} A \emph{deformation datum} on a scheme $X$ indexed by a finite totally ordered set $(I,\leq )$ is a pair $\left({}^{D_I}_{X_I}\right)=\left({}^{D_i}_{X_i}\right)_{i\in I}$
consisting of an order reversing map $$X_I:(I,\leq)\to Clo(X,\subset): i\mapsto X_I(i)=X_i \quad \textrm{(if $i\geq i' $, then $X_i \subset X_{i'}$)}$$ and a collection $D_I=\{D_i\}_{i\in I}$ of locally principal closed subschemes of $X$. 
\xdefi

\defi 
\label{def:mc-def}The \emph{multi-centered deformation space} of a deformation datum $\left({}^{D_I}_{X_I}\right)$ on a scheme $X$ is the $X$-scheme
\begin{equation} 
\label{sigma_I} 
\sigma_I:\bbD_I =\bbD \left(\left({}^{D_I}_{X_I}\right) {}_{X} \right)=\bbD \left(\left({}^{D_i}_{X_i}\right)_{i\in I} {}_{X} \right) := \Bl \left\{{}^{\sum _{j\in I_{\geq i}} D_j}_{X_i}  \right\}_{i \in I} X\to X,
\end{equation}
that is, the dilatation of $X$ with multi-center $\{[X_i, \sum _{j \in I_{\geq i}} D_j ]\}_{i \in I}$. 
\xdefi

Having an abstract index set $I$ will be convenient to state and prove the main results, but in practice, there is no loss of generality to identify $I $ with the totally ordered set $ \{n \geq i \geq 1\}$ where $n= \# I$. A deformation datum is then a chain $X_n \subset \ldots \subset X_1 \subset X$ of closed subschemes together with locally principal closed subschemes $D_n , \ldots , D_1$ of $X$.
In this case, we will denote the associated multi-centered deformation space by 
\begin{equation}\label{eq:ordered-notation}
\bbD \big (^{D_n, \ldots, D_1} _{X_n,  \ldots, X_1  ~ X } \big)= \Bl \left\{ _{X_i}^{\sum _{j =i}^n D_j }\right\}_{n \geq i \geq 1} X.
\end{equation} 

\medskip

 \prop \label{Propunivdef}(Universal property of deformation spaces) 
 The multi-centered deformation space $\sigma_I:\bbD \left(\left({}^{D_I}_{X_I}\right) {}_{X} \right)\to X$ represents the contravariant functor 
 $\delta : \Sch _{X}^{\{D_i\}_{i\in I} \text{-reg}} \to Sets$ defined by
\begin{equation}\label{blow.up.iso.eq}
(f: T\to X) \;\longmapsto\; \begin{cases}\left\{*\right\}, \; \text{if $f^{-1} \left(\sum _{s\in I_{\geq i}} D_s \right) \subset f^{-1} \left( X_i \right) $ in $Clo \left( T \right) $ for every $i \in I  $};\\ \varnothing,\;\text{else.}\end{cases}
\end{equation}
 \xprop
 \pf
 The categories $\Sch _{X}^{\{D_i\}_{i \in I}  \text{-reg}}$ and $\Sch _{X}^{\left\{\sum _{j \in I _{\geq i}} D_j \right\}_{i \in I}\text{-reg}}$ being equal by \cite[Fact 3.5]{Ma24d}, the assertion is an immediate consequence of Definition \ref{def:mc-def} and \cite[Proposition 3.17]{Ma24d}.
 \xpf 
 
 \coro \label{coronumberofmor}
If $T \to X $ is an object in the category $\Sch _{X}^{\{D_i\}_{i\in I} \text{-reg}}$, we have \[\#\Hom _X \left( T, \bbD \left(\left({}^{D_I}_{X_I}\right) {}_{X} \right)\right) \in \left\{0,1\right\}.\]
 \xcoro
 \pf
 Immediate from Proposition \ref{Propunivdef}.
 \xpf

 \coro \label{cor:isocomplement} With the notation above, the open subset $\bbD_I\setminus\sigma_I^{-1}(\sum_{i\in I}D_i)$ of $\bbD_I$ is Zariski dense and the restriction $\sigma_I:\bbD_I\setminus\sigma_I^{-1}(\textstyle \sum_{i\in I}D_i)\to X\setminus(\textstyle \sum_{i\in I} D_i)$
 is an isomorphism. 
 \xcoro
 \pf Put $U_I=X\setminus(\textstyle \sum_{i\in I} D_i)$. The property that $\sigma_I^{-1}(U_I)=\bbD_I\setminus\sigma_I^{-1}(\textstyle \sum_{i\in I}D_i)$ is Zariski 
 dense follows from the fact that $\sigma_I^{-1}(\sum_{i\in I}D_i)$ is an effective Cartier divisor on $\bbD_I$. Since the inclusion morphism $f:U_I\to X$ belongs to  $\Sch _{X}^{D_I \text{-reg}}$  and since  $f^{-1}(\sum_{j\in I_{\geq i}} D_j)=\emptyset \subset f^{-1}(X_i)$ for every $i\in I$, Proposition \ref{Propunivdef} 
 implies the existence of a unique $X$-morphism $\hat{f}:U_I\to\bbD_I$, which is then necessarily an isomorphism onto its image $\sigma_I^{-1}(U_I)$.
 \xpf

 \subsection{Base change and restriction of index sets} 
 
 Let $\left({}^{D_I}_{X_I}\right)$ be a deformation datum on a scheme $X$ and let $f:X'\to X$ be an $X$-scheme. We then have an induced order reversing  map $$X'_I:(I, \leq) \to Clo (X', \subset ), i\mapsto X'_I(i):=X_i\times_X X'.$$ On the other hand, $D'_i:=f^{-1}D_i$ is a locally principal closed subscheme of $X'$ for every $i\in I$. 
 \defi \label{warning} The deformation datum $\left({}^{D'_I}_{X'_I}\right)$ is called the base change of $\left({}^{D_I}_{X_I}\right)$
 by $f:X'\to X$. We denote it by $f^{-1}\left({}^{D_I}_{X_I}\right)=\left({}^{D_I}_{X_I}\right)\times_X X'$ or, for short, simply by $\left({}^{D_I}_{X_I}\right)\!|_{{}_{X'}}$.
 \xdefi

In general, the formation of the multi-centered deformation space does not commute with base change. Nevertheless, we have the following sufficient criterion: 

 \lemm \label{lem:base-change} Let  $\left({}^{D_I}_{X_I}\right)$ be a deformation datum on a scheme $X$ and let $f:X'\to X$ be an $X$-scheme. Then there exists a unique morphism of $X$-schemes $$\bbD \left( \left({}^{D_I}_{X_I}\right)|_{X'}{~}_{X'}\right)\to \bbD \left( \left({}^{D_I}_{X_I}\right){~}_{X}\right).$$
 Moreover, if $\bbD \left( \left({}^{D_I}_{X_I}\right){~}_{X}\right)\times _X X'$ belongs to the category $\mathrm{Sch}_{X'}^{\{D_i|_{X'}\}_{i\in I}\textrm{-reg}}$ then the induced morphism of $X'$-schemes 
 $$\bbD \left( \left({}^{D_I}_{X_I}\right)|_{X'}{~}_{X'}\right)\to \bbD \left( \left({}^{D_I}_{X_I}\right){~}_{X}\right)\times_X X'$$
 is an isomorphism.
 \xlemm
\pf 
The composition  $h:=f\circ \tau:  \bbD \left( \left({}^{D_I}_{X_I}\right)|_{X'}{~}_{X'}\right)\to X'\to X$,  where $\tau$ is the dilatation morphism, belongs to the category $\mathrm{Sch}_X^{\{D_i\}_{i\in I}\textrm{-reg}}$. Since by definition of $\tau$, we have, for all $i\in I$,  $$\textstyle  h^{-1}(\sum_{s\in I_{\geq i}} D_s)=\tau^{-1}(f^{-1} (\sum_{s\in I_{\geq i}} D_s))=\tau^{-1}(\sum_{s\in I_{\geq i}} f^{-1}D_s) \subset \tau^{-1}(f^{-1}(X_i))=h^{-1}(X_i),$$
the existence and uniqueness 
of the claimed morphism follow from Proposition \ref{Propunivdef}. The second assertion follows from  \cite[Lemma 3.36]{Ma24d}.
\xpf
\exam With the notation above, assume that $f:X'\to X$ is flat and that $D_i$ is Cartier for every $i\in I$. Then $D_i|_{X'}$ is Cartier as well by  \cite[\href{https://stacks.math.columbia.edu/tag/02OO}{Tag 02OO}]{stacks-project} and the induced morphism  $\bbD \left( \left({}^{D_I}_{X_I}\right)|_{X'}{~}_{X'}\right)\to \bbD \left( \left({}^{D_I}_{X_I}\right){~}_{X}\right)\times_X X'$ is an isomorphism.
\xexam

 Let $\left({}^{D_I}_{X_I}\right)$ be a deformation datum on a scheme $X$.
For every subset $J\subset I$, restricting the index set to $J$ determines a deformation datum $\left({}^{D_J}_{X_J}\right)=\left({}^{D_j}_{X_j}\right)_{j\in J}$ on $X$. Letting  
\begin{equation}
    \sigma_I:\bbD_I:=\bbD \left( \left({}^{D_I}_{X_I}\right){~}_{X}\right)\to X \; \textrm{and} \;   \sigma_J:\bbD_J:=\bbD \left( \left({}^{D_J}_{X_J}\right){~}_{X}\right)\to X
\end{equation}
be the dilatation morphisms, we have the following:

\lemm \label{lem:nu_IJ} There exists a unique morphism of $X$-schemes 
\begin{equation} \label{eq:upsilon_ij}
\upsilon_{I,J}: \bbD_I \to \bbD_J. 
\end{equation}
Moreover,  $\upsilon_{I,J}:\bbD_I\to \bbD_J$ belongs to the category 
$\Sch_{\bbD_J}^{\{\sigma_J^{-1}D_{i}\}_{i\in I\setminus J} \text{-reg}}$.
\xlemm
\pf Since $\bbD_I\in \Sch_X^{\{D_j\}_{j\in J}\textrm{-reg}}$ and since, by definition of $\sigma_I$, we have, for every $j\in J$, $$\textstyle \sigma_I^{-1}(\sum_{s\in J_{\geq j}} D_s)\subset \sigma_I^{-1}(\sum_{s\in I_{\geq j}} D_s) \subset \sigma_I^{-1}(X_j),$$ the existence and uniqueness follow from 
Proposition \ref{Propunivdef}
. The second assertion is immediate.
\xpf

\subsection{Induced multi-centered deformation spaces and panels} \label{subsec:panels}
Let  $\left({}^{D_I}_{X_I}\right)$ be a deformation datum on a scheme $X$ and let $J$ be a subset of $I$. For every element $i\in I\setminus J$, the base change  $ \left({}^{D_J}_{X_J}\right)\!|_{_{X_i}}$ is a deformation datum on the closed subscheme $X_i$ of $X$, with associated multi-centered deformation space 
\begin{equation} \label{induced_1}
\sigma_{J,i}:(\bbD_{J})_i:= \bbD \left( \left({}^{D_J}_{X_J}\right)\!|_{_{X_i}} {~}_{X_i}\right)\to X_i.
\end{equation}
We let $(\sigma_{J})_i:(\bbD_J)_i\to X$  be the composition of $\sigma_{J,i}$ with the inclusion $X_i\subset X$. The inclusion $X_{i'}\subset X_i$ as closed subschemes of $X$ for every pair 
of elements $i< i'$ in $I\setminus J$ implies that $ \left(\left({}^{D_J}_{X_J}\right)\!|_{_{X_i}}\right)\!|_{X_{i'}}=\left({}^{D_J}_{X_J}\right)\!|_{_{X_{i'}}}$ and that  $(\sigma_{J})_{i'}$ factors through the inclusion $X_{i'}\subset X_i$. 

\prop \label{factfks} With the notation above, the following hold: 
\begin{enumerate}[label=(\roman*)]
 \item For all $i \in I\setminus J$, there exists a unique $X$-morphism 
\begin{equation} \label{eqfks}
F_J(i) : (\bbD_J)_i \to \bbD_J.
\end{equation}
Moreover, $F_J(i)$ is a closed immersion.
\item For all $i< i'$ in $I\setminus J$, there exists a unique $X$-morphism 
\begin{equation} \label{eqfks2}
F_J(i',i) : (\bbD_J)_{i'} \to (\bbD_J)_i. 
\end{equation}
Moreover $F_J(i',i)$ is a closed immersion and $F_J(i')=F_J(i)\circ F_J(i',i)$.
\end{enumerate}
\xprop
 
 \pf Since, by definition of $(\bbD_J)_i$, we have $$(\sigma_J)_i^{-1}(\sum_{s\in J_{\geq j}}D_s)=\sigma_{J,i}^{-1}(\sum_{s\in J_{\geq j}}D_s|_{X_i})\subset \sigma_{J,i}^{-1}(X_j|_{X_i})=(\sigma_J)_i^{-1}(X_j),$$ 
 the existence and uniqueness of the desired morphism $F_J(i):(\bbD_J)_i\to \bbD_J$ follow from  Proposition \ref{Propunivdef}. Since the projection $\mathrm{pr}_1:  \bbD_J\times_X X_i\to \bbD_J$ is a closed immersion, $F_J(i)$ is a closed immersion provided that the canonically induced $X_i$-morphism $\tilde{F}_J(i):(\bbD_J)_i\to \bbD_J\times_X X_i$ is a closed immersion. The latter properties being local on $X$ in the Zariski topology, the fact that $\tilde{F}_J(i)$ is a closed immersion follows from \cite[Proposition 2.41]{Ma24d}. 
 The existence of $F_J(i',i)$ for all $i<i'$ follows from Proposition \ref{Propunivdef} and the equality $F_J(i')=F_J(i)\circ F_J(i',i)$ from Corollary \ref{coronumberofmor}. Since $F_J(i')$ and $F_J(i)$ are both closed immersions, so is $F_J(i',i)$.
 \xpf
 
 By Proposition \ref{factfks}, for every $i<i'$ in $I\setminus J$, we have a commutative diagram 
 $$ 
 \begin{tikzcd}
  (\bbD_J)_{i'} \ar[r, "{F_J(i',i)}"] \ar[d,swap, "\sigma_{J,i'}"] & (\bbD_J)_i \ar[r,"F_J(i)"] \ar[d,"\sigma_{J,i}"] & \bbD_J \ar[d,swap, "\sigma_J"] \\ X_{i'} \ar[r] & X_i \ar[r] & X.
 \end{tikzcd}
 $$
 in which the top and bottom rows consist of closed immersions. 
Identifying each $(\bbD_J)_i$, $i\in I\setminus J$, with its image in $\bbD_J$, it follows that the association 
$$(\bbD_J)_{I\setminus J}:(I\setminus J,\leq)\to Clo(\bbD_J, \subset), \; i\mapsto (\bbD_J)_{I\setminus J}(i):=(\bbD_J)_i$$
is an order reversing map. 
For every $i\in I\setminus J$, the inverse image of $D_i$ by the dilatation morphism  $\sigma_J:\bbD_J \to X$ 
is a locally principal divisor $\sigma_J^{-1}D_i$ on $\bbD_J$. We put $\sigma_J^{-1}D_{I\setminus J}=(\sigma_J^{-1}D_i)_{i\in I\setminus J}$.

\defi \label{def:panel-datum} 
For every subset $J\subset I$, the multi-centered deformation space
\begin{equation}\label{tau-K-J}
\tau_{J}: \bbD\bbD_J:=
\bbD \left( \left({}^{\sigma_J^{-1}D_{I\setminus J}}_{(\bbD_J)_{I\setminus J}}\right) {~}_{\bbD_J}\right)=\bbD \left( \left({}^{\sigma_J^{-1}D_{i}}_{(\bbD_J)_{i}}\right)_{i\in I\setminus J} {~}_{\bbD_J}\right) \to \bbD_J 
\end{equation}
of  $\bbD_J= \bbD  \left(\left({}^{D_J}_{X_J}\right) {~}_X\right)$
is called the \emph{$J$-th panel} of the multi-centered deformation space $\bbD_I$.
\xdefi

\subsection{Local affine models}\label{subsec:affine}
We collect descriptions and basic properties of local affine versions of the global notions introduced in the previous subsections, which we reformulate in the equivalent setting of multi-centered dilatations of rings of \cite[Section  2]{Ma24d}. In what follows, $X=\mathrm{Spec}(A)$ is an affine scheme and the deformation  datum  $\left({}^{D_I}_{X_I}\right)$ on $X$ consists of closed subschemes $X_i=\mathrm{Spec}(A/M_i)$ for some ideals $M_i\subset A$, $i\in I$,  such that $M_{i}\subset M_{i'}$ for all $i'\geq i$ in $I$, and principal divisors $D_i=\mathrm{div}(d_i)$ for some elements $d_i\in A$, $i\in I$.  

For every subset $J\subset I$, the multi-centered deformation space $\sigma_J:\bbD_J\to X$
is then the spectrum of the $A$-algebra 
\begin{equation}
    \sigma_J^*:A\to A\left[ \left\{ \frac{M_{j}}{\prod_{s \in J_{\geq j}} d_s}\right\}_{j \in J} \right]:=R_J,   
\end{equation} 
where $\sigma_J^*$ is the dilatation homomorphism mapping an element $a$ to the element  $\tfrac{a}{1}$ of $R_J$. 

The universal property of $\sigma_J:\bbD_J\to X$ in Proposition \ref{Propunivdef} then translates in particular into the property that for an $A$-algebra  $f^*:A\to B$ such that $f^*(d_s)$ is  not a zero divisor in $B$ and such that for every $j\in J$, $f^*(M_j)$ is contained in the ideal $(f^*(\prod_{s\in J_{\geq j}} d_s))$ of $B$, $f^*$ factors as $\hat{f}^*\circ \sigma_J^*$ for a unique $A$-algebra homomorphism $\hat{f}^*:R_J\to B$.

The morphism $\upsilon_{I,J}:\bbD_I\to \bbD_J$ in Lemma \ref{lem:nu_IJ} corresponds to the $A$-algebra homomorphism
\begin{equation}
    \upsilon_{I,J}^*:R_J=     A\left[ \left\{ \frac{M_{j}}{\prod_{s \in J_{\geq j}} d_s}\right\}_{j \in J} \right]\to  A\left[ \left\{ \frac{M_{i}}{\prod_{s \in I_{\geq i}} d_s}\right\}_{i \in I}\right]=R_I
\end{equation}
which maps an element  
\begin{equation}\label{eq:r_j}
r=\frac{\prod_{j\in J} \ell_j}{\prod_{j\in J}(\prod_{s\in J_{\geq j}}d_s)^{\nu_j}}\in R_J,
\end{equation}
where $\ell_j\in (M_j+(\prod_{s\in J_{\geq j}} d_s))^{\nu_j}$, to the element 
$$ \upsilon_{I,J}^*(r):=\frac{\prod_{j\in J}\ell_j(\prod_{s\in (I\setminus J)_{\geq j}}d_s)^{\nu_j}}{\prod_{j\in J}(\prod_{s\in I_{\geq j}}d_s)^{\nu_j}}\in R_I.$$

The closed immersion  $F_J(i):(\bbD_J)_i\to \bbD_J$ in Proposition \ref{factfks} corresponds to the surjective homomorphism 
\begin{equation}
\label{eq:f_ij-def} 
\textstyle 
F_J(i)^*:   R_J=
A\left[ \left\{ \frac{M_{j}}{\prod_{s \in J_{\geq j}} d_s}\right\}_{j \in J} \right] \to A/M_{i} \left[ \left\{ \frac{M_{j} +M_i}{\prod_{s \in J_{\geq j}} \pi_i(d_s)}\right\}_{j \in J} \right],
  \end{equation}     
where $\pi_i:A\to A/M_i$ is the quotient morphism, which maps an element  $r\in R_J$ as in \eqref{eq:r_j} above to the element 
$$F_J(i)^*(r)=\frac{\prod_{j\in J}\pi_i(\ell_j)}{\prod_{j\in J}\left(\prod_{s\in J_{\geq j}} \pi_i(d_s)\right)^{\nu_j}}.$$

Finally, the morphism  $\tau_J:\bbD\bbD_J\to \bbD_J$ of Definition \ref{def:panel-datum} corresponds to the dilatation homomorphism of $R_J$-algebras 
\begin{equation}
    \tau_J^*: R_J \to R_J \left[ \left\{ \frac{\ker \left(F_J(i)^* \right) }{\sigma_J^*(\prod _{s \in (I\setminus J)_{\geq i}}d_s)}\right\}_{i\in I\setminus J}\right]:=B_J.
\end{equation}

\medskip 

As a preparation for the next section, we now establish basic results concerning the kernel  $\ker(F_J(i)^*)$ of the homomorphism $F_J(i)^*$. A first observation is the following: 
\lemm \label{lem:KerfJi-inclusion} For every $i\in I\setminus J$, the following inclusion of ideals holds in $R_J$:
 \begin{equation} \label{eq:Kerf_Ji-inclusion}
\left(\frac{M_i}{\prod_{s\in J_{\geq i}}d_s}+\sum_{j\in J} \frac{M_j\cap M_i}{\prod_{s\in J_{\geq j}} d_s}\right)\subset \ker(F_J(i)^*).
 \end{equation}    
\xlemm
\pf
First observe that the notation $\frac{M_i}{\prod_{s\in J_{\geq i}}d_s}$ makes sense as a subset of $R_J$. Indeed, if $J_{\geq i}=\emptyset$ then $\prod_{s\in J_{\geq i}}d_s=1$ and $\frac{M_i}{1}=\sigma_J^*(M_i)$. Otherwise, if $J_{\geq i}\neq \emptyset$ then letting $j_0=\min\{J_{\geq i}\}$ we have $M_i=M_{j_0}\cap M_i$ and $J_{\geq i}=J_{\geq j_0}$ so that $\frac{M_i}{\prod_{s\in J_{\geq i}}d_s}=\frac{M_{j_0}\cap M_i}{\prod_{s\in J_{\geq j_0}}d_s}$ which is a well-defined subset of $R_J$. Now the fact that the subsets 
$\frac{M_i}{\prod_{s\in J_{\geq i}}d_s}$
and $\frac{M_j\cap M_i}{\prod_{s\in J_{\geq j}} d_s}$ for all $j\in J$ are contained in $\ker(F_J(i)^*)$ is clear from the definition of $F_J(i)^*$.      
\xpf

\lemm \label{lem:Kerf_Ji-Cartier} With the notation above, assume  that for every index $j\in J$ the image $\pi_i(d_j)$ of $d_j$ in $A/M_i$ is not a zero divisor. Then 
\begin{equation}\label{YHJ85-1} \ker (F_J(i)^*) = \sum _{\nu=(\nu_j) \in (\mathbb{Z}_{\geq 0})^J} \frac{\left(\prod_{j \in J}\left(M_j + (\prod_{s \in J_{\geq j}} d_s)\right)^{\nu _j}\right)\cap M_i}{\prod _{j \in J}\left(\prod_{s \in J_{\geq j}} d_s\right) ^{\nu _j}}.  
 \end{equation}
\xlemm
\pf This is immediate from the definition of $R_J$ and Proposition  \ref{propnoyauappendix} below. 
\xpf

In the proof of Lemma \ref{lem:Kerf_Ji-Cartier}, we used the following variant of \cite[Proposition 2.39]{Ma24d}.
\prop \label{propnoyauappendix} Let $A$ be a ring and $I$ be a set. Let $\{ M_i \}_{i \in I}$ be ideals of $A$ and $\{a_i\}_{i \in I}$ be elements of $A$. Put $L_i=(M_i+(a_i))$ 
and for a multi-index $\nu=(\nu_i)\in (\mathbb{Z}_{\geq 0})^I$, put $L^\nu=\prod_{i\in I}L_i^{\nu_i}$ and $a^\nu=\prod_{i\in I}a_i^{\nu_i}$. Let $T$ be an ideal of $A$ 
such that  we have a commutative diagram of $A$-algebras 
\begin{equation*}
\begin{tikzcd}  A/T \left[\left\{ \frac{M_i +T}{\pi(a_i)}\right\}_{i \in I}\right] &  A/T  \ar[l,swap, "\tau^*"] \\  A\left[\left\{ \frac{M_i}{a_i}\right\}_{i \in I}\right] \ar[u, "\varphi"] &  A \ar[l,swap,  "\sigma^*"]\ar[u,swap,  "\pi"],
\end{tikzcd} 
\end{equation*} 
where $\sigma^*$ and $\tau^*$ are the dilatation  homomorphisms and $\pi$ is the quotient homomorphism. 
Assume $\pi(a_i )$ is a non-zero-divisor in $A/T$ for all $i \in I$.
Then 
$$ \ker (\varphi) = \sum _{\nu \in (\bbZ_{\geq 0})^I} \frac{L^\nu \cap T }{a^\nu}  \subset  A\left[\left\{ \frac{M_i}{a_i}\right\}_{i \in I}\right].$$
\xprop 
\pf Put $\psi=\tau^*\circ \pi$.
The assumption that the elements $\pi (a_i)$, $i \in I$, are not zero-divisors in $A/T$, implies that $\tau^*$ is injective by \cite[Proposition 2.20]{Ma24d}. 
For an element $\frac{\ell}{a^\nu}$ of $ A[\big\{ \frac{M_i}{a_i}\big\}_{i \in I}]$, where $\ell\in L^\nu$, we have 
\[ \psi ({a^\nu} ) \varphi (\frac{\ell}{a^\nu})= \varphi (\sigma^*(a^\nu)) \varphi (\frac{\ell}{a^\nu}  ) = \varphi (\sigma^*(\ell))=  \psi (\ell).\]
If  $\ell \in L^\nu \cap T$, then $\psi (\ell)=0$, and since $\psi(a^{\nu})=\tau^*(\pi(a^\nu))=\varphi(\sigma^*(a^\nu))$ is not a zero-divisor by \cite[Fact 2.10]{Ma24d}, it follows that $\varphi(\frac{\ell}{a^{\nu}})=0$. Thus, $ \sum _{\nu \in (\bbZ_{\geq 0})^I} \frac{L^\nu \cap T }{a^\nu}\subset \ker (\varphi)$. Conversely, if $ \varphi (\frac{\ell}{a^\nu}  )=0$,  then $\psi (\ell)=\tau^*(\pi(\ell))=0$, and since $\tau^*$ is injective, it follows that $\pi(\ell)=0$, whence that  $\ell\in L^{\nu}\cap T$. 
\xpf 
 
The next proposition provides sufficient conditions for the inclusion \eqref{eq:Kerf_Ji-inclusion} to be an equality. 

\prop \label{prop:Kernel-formula} With the notation above, assume that the following hold:
\begin{enumerate}[label=(\roman*)] 
    \item For every $j\in J$, the image $\pi_i(d_j)\in A/M_i$ of $d_j$ is not a zero divisor. 
    \item If $J_{\geq i}\neq \emptyset$ then the following two conditions are satisfied:
    \begin{enumerate}
    \item \label{assut3item2-1} For every $(\alpha_j) \in (\bbZ_{\geq 0})^{J_{\geq i}}\setminus\{0\}$, the following equality of ideals holds in $A$ 
     $$  \textstyle \left(\prod_{j\in J_{\geq i}} M_j^{\alpha_j}\right)\cap M_i=\left(\prod_{j\in J_{>j_0}} M_j^{\alpha_j}\right)M_{j_0}^{\alpha_{j_0}-1} M_i,$$   
  where $j_0 = \min \{ j \in J_{\geq i},\,  \alpha_{j} \neq 0 \}.$     
     \item  \label{assut3item1-1}  For every collection of pairs $((\alpha_{j'}^{(e)}),(\beta_j^{(e)}))\in (\mathbb{Z}_{\geq 0})^{J_{>i}}\times (\mathbb{Z}_{\geq 0})^{J}$ indexed by the elements $e$ of a finite set $E$,  the following equality of ideals holds in $A$: 
     {\small 
     $$ \textstyle \left(\sum_{e\in E} \left(\prod_{j'\in J_{\geq i}}  M_{j'}^{\alpha_{j'}^{(e)}}\right)\left(\prod_{j\in J} d_j^{\beta_j^{(e)}}\right)\right)\cap M_i=\sum_{e\in E}  \left(\left(\prod_{j'\in J_{\geq i}}  M_{j'}^{\alpha_{j'}^{(e)}}\right)\cap M_i\right)\left(\prod_{j\in J} d_j^{\beta_j^{(e)}}\right).$$   
     }
    \end{enumerate}
    \end{enumerate}
Then the inclusion \eqref{eq:Kerf_Ji-inclusion} in Lemma \ref{lem:KerfJi-inclusion} is an equality, that is,
$$\ker(F_J(i)^*)= \left(\frac{M_i}{\prod_{s\in J_{\geq i}}d_s}+\sum_{j\in J} \frac{M_j\cap M_i}{\prod_{s\in J_{\geq j}} d_s}\right).$$
\xprop

\pf
Applying Lemma \ref{lem:Kerf_Ji-Cartier} using Assumption (i), we are reduced to check that for every multi-index $\nu=(\nu_j)\in (\mathbb{Z}_{\geq 0})^{J}$, the inclusion 
\begin{equation}\label{eq:Goal-inclusion-prop}
    \frac{\left(\prod_{j \in J}\left(M_j + (\prod_{s \in J_{\geq j}} d_s)\right)^{\nu _j}\right)\cap M_i}{\prod _{j \in J}\left(\prod_{s \in J_{\geq j}} d_s\right) ^{\nu _j}} \subset \left(\frac{M_i}{\prod_{s\in J_{\geq i}}d_s}+\sum_{j\in J} \frac{M_j\cap M_i}{\prod_{s\in J_{\geq j}} d_s}\right).
\end{equation}
holds in $R_J$. 
Since $i\in I\setminus J$, we have $J=J_{\geq i}\cup J_{\leq i}=J_{>i}\cup J_{<i}$ and we can write 
\begin{equation*} 
\begin{split}
\textstyle   \prod_{j \in J}\left(M_j + (\prod_{s \in J_{\geq j}} d_s)\right)^{\nu _j} 
  & =\textstyle \left( \prod_{j \in J_{<i}}\left(M_j + (\prod_{s \in J_{\geq j}} d_s)\right)^{\nu _j}\right)  \textstyle  \left( \prod_{j \in J_{>i}}\left(M_j + (\prod_{s \in J_{\geq j}} d_s)\right)^{\nu _j}
    \right)\\
    & =  \left(S_{J_{<i}}+(Q_{J_{<i}})\right)\left(S_{J_{>i}}+(Q_{J_{>i}}) \right) 
    \\ 
    &= S_{J_{<i}}\left(S_{J_{>i}}+(Q_{J_{>i}})\right)+\left(S_{J_{>i}}(Q_{J_{<i}})+(Q_{J_{<i}})(Q_{J_{>i}})\right)
\end{split}    
\end{equation*}
where, for $\Gamma=J_{<i}$ and $J_{>i}$, we have put
$$S_{\Gamma}=\sum_{\underset {(\alpha_\gamma)\neq (0)}{(\alpha_\gamma)+(\beta_\gamma)=(\nu_\gamma)}}\textstyle \prod_{\gamma\in \Gamma}\left(M_\gamma^{\alpha_\gamma}(\prod_{s \in J_{\geq \gamma}} d_s)^{\beta_\gamma}\right) \textrm{ and } 
Q_\Gamma=\textstyle \prod_{j\in \Gamma} (\prod_{s \in J_{\geq j}} d_s)^{\nu_j},
$$
with the convention that $S_\Gamma=0$ and $Q_\Gamma=1$ if $\Gamma=\emptyset$. Putting $Q_J=Q_{J_{<i}}Q_{J_{>i}}$, it follows from Assumption (i) that in $A$, we have $(Q_J)\cap M_i=(Q_J)M_i$, which implies in turn that 
\begin{equation} \label{eq:Case-Q}
  \frac{(Q_J)\cap M_i}{\prod _{j \in J}\left(\prod_{s \in J_{\geq j}} d_s\right)^{\nu _j}}=\frac{(Q_J)M_i}{\prod _{j \in J}\left(\prod_{s \in J_{\geq j}} d_s\right)^{\nu _j}}\subset \left(\frac{M_i}{1}\right)\subset \left(\frac{M_i}{\prod_{s\in J_{\geq i}} d_s}\right).    
\end{equation}
Since $M_j\subset M_i$ for all $j\in J_{<i}$, it follows that $S_{J_{<i}}\subset M_i$ and hence that 
{\small 
\begin{equation}
\textstyle   \left(\prod_{j \in J}\left(M_j + (\prod_{s \in J_{\geq j}} d_s)\right)^{\nu _j} \right)\cap M_i=
\begin{cases} 
   S_{J_{<i}} + (Q_J)M_i & \textrm{ if } J_{>i}=\emptyset \\
  \left(S_{J_{>i}}(Q_{J_{<i}})+(Q_J)\right)\cap M_i & \textrm{ if } J_{<i}=\emptyset \\
   S_{J_{<i}}\left(S_{J_{>i}}+(Q_{J_{>i}})\right)+\left(S_{J_{>i}}(Q_{J_{<i}})+(Q_J)\right)\cap M_i & \textrm{otherwise.}
   \end{cases}
\end{equation}
}

\textbullet If $J_{<i}\neq \emptyset$ then for any term $\prod_{j\in J_{<i}} \left(M_j^{\alpha_j}(\prod_{s\in J_{\geq j}}d_s)^{\beta_j}\right)$ of the sum $S_{J_{<i}}$ and any term $\prod_{j\in J_{>i}} \left(M_j^{\alpha_j}(\prod_{s\in J_{\geq j}}d_s)^{\beta_j}\right)$ of the sum $(S_{J_{>i}}+(Q_{J_{>i}}))$, we have {\small 
\begin{equation} \label{eq:J<i-main-term}
    \begin{split}
        \frac{\left(\prod_{j\in J_{<i}} \left(M_j^{\alpha_j}(\prod_{s\in J_{\geq j}}d_s)^{\beta_j}\right)\right)\left(\prod_{j\in J_{>i}} \left(M_j^{\alpha_j}(\prod_{s\in J_{\geq j}}d_s)^{\beta_j}\right)\right)}{\prod _{j \in J}\left(\prod_{s \in J_{\geq j}} d_s\right) ^{\nu _j}}
        & \subset \prod_{j\in J_{<i}}\left(\frac{ M_j}{\prod_{s \in J_{\geq j}} d_s}\right)^{\alpha_j}\prod_{j\in J_{>i}}\left(\frac{ M_j}{\prod_{s \in J_{\geq j}} d_s}\right)^{\alpha_j} \\
          & \subset \left(\frac{M_{j_0}}{\prod_{s \in J_{\geq j_0}} d_s}\right)=\left(\frac{M_{j_0}\cap M_i}{\prod_{s \in J_{\geq j_0}} d_s}\right),
    \end{split} 
\end{equation}
}
where $j_0=\min\{j\in J_{<i}, \alpha_j\neq 0\}$.

\textbullet If $J_{>i}\neq \emptyset$ then, by Assumption (ii) (b), we have 
{\small 
\begin{equation}\label{eq:J>i-split}
\begin{split}    
\left(S_{J_{>i}}(Q_{J_{<i}})+(Q_J)\right)\cap M_i & = \sum_{\underset {(\alpha_j)\neq (0)}{(\alpha_j)+(\beta_j)
=(\nu_j)}}\textstyle \left(\left(\prod_{j\in J_{>i}} M_j^{\alpha_j}\right)\cap M_i\right)\left(Q_{J_{<i}}
\prod_{j\in J_{>i}}(\prod_{s \in J_{\geq j}} d_s)^{\beta_j}\right) +(Q_J)M_i. 
\end{split}
\end{equation}
}
Moreover, by Assumption (ii) (a), we have  $(\prod_{j\in J_{>i}}M_j^{\alpha_j})\cap M_i=(\prod_{j\in J_{>j_0}} M_j^{\alpha_j})M_{j_0}^{\alpha_{j_0}-1}M_i$, 
where $j_0=\min\{j\in J_{>i}, \alpha_j\neq 0\}$. Since $M_i=M_{j_0}\cap M_i$, it follows that we have 
\begin{equation}\label{eq:J>i-main-term}
\begin{split}
    \frac{\left(\left(\prod_{j\in J_{>i}} M_j^{\alpha_j}\right)\cap M_i\right)\left(Q_{J_{<i}}
\prod_{j\in J_{>i}}(\prod_{s \in J_{\geq j}} d_s)^{\beta_j}\right)}{\prod _{j \in J}\left(\prod_{s \in J_{\geq j}} d_s\right)^{\nu _j}} & = \frac{(\prod_{j\in J_{>j_0}} M_j^{\alpha_j})M_{j_0}^{\alpha_{j_0}-1}M_i}{\left(\prod _{j \in J_{>j_0}}\left(\prod_{s \in J_{\geq j}} d_s\right)^{\alpha_j}\right)\left(\prod_{s \in J_{\geq j_0}} d_s\right)^{\alpha_{j_0}}} \\
& \subset 
\left(\frac{M_{j_0}\cap M_i}{\prod_{s\in J_{\geq j_0}}d_s} \right).  
\end{split}
\end{equation}
The desired inclusion \eqref{eq:Goal-inclusion-prop} now follows by combining \eqref{eq:Case-Q} and \eqref{eq:J<i-main-term} in the case where $J_{>i}=\emptyset$,  \eqref{eq:Case-Q}, \eqref{eq:J>i-split} and \eqref{eq:J>i-main-term} in the case where $J_{<i}=\emptyset$ and all these equations in the general case. 
\xpf

\section{Panelizations of multi-centered deformation spaces} \label{secSliding} 

We proceed with the notation and assumptions of Section \ref{secD}. We fix a scheme $X$ and a deformation datum  $\left({}^{D_I}_{X_I}\right)$ on $X$ for some finite totally ordered set $(I,\leq)$ which  satisfies the following: 

\begin{assum} \label{assumslid} For all $(i,i') \in I^2$, the scheme $X_i \cap D_{i'}=X_i\times_X D_{i'}$ is a Cartier divisor in $X_i$.
\end{assum}

\subsection{The panelization morphism} \label{sec:panel-mor}
Recall from subsection \ref{subsec:panels} that given a subset $J\subset I$, we have for every $i\in I\setminus J$ an induced multi-centered deformation space $\sigma_{J,i}:(\bbD_{J})_i=\bbD \left( \left({}^{D_J}_{X_J}\right)\!|_{_{X_i}} {~}_{X_i}\right)\to X_i$. The collection of these multi-centered deformation spaces $(\bbD_J)_i$, $i\in I\setminus J$, determines a multi-centered deformation space 
$$\tau_{J}: \bbD\bbD_J=
\bbD \left( \left({}^{\sigma_J^{-1}D_{I\setminus J}}_{(\bbD_J)_{I\setminus J}}\right) {~}_{\bbD_J}\right) \to \bbD_J=\bbD  \left(\left({}^{D_J}_{X_J}\right) {~}_X\right)$$
of the multi-centered deformation space $\sigma_J:\bbD_J\to X$ which we called, in Definition \ref{def:panel-datum}, the $J$-th panel of the multi-centered deformation space $\sigma_I:\bbD_I\to X$. This  terminology is motivated by the following result, which asserts the existence of a canonical morphism $\Theta_J: \bbD\bbD_J \to \bbD_I$, which we call the \emph{panelization morphism} of $\bbD_I$ associated to the subset $J\subset I$. 
\prop 
 \label{theopanelization-mor} 
 With the assumption and notation above, for every subset $J\subset I$ there exists a unique $X$-morphism 
 \begin{equation} \label{paneisoeq}\Theta_J: \bbD\bbD_J\to \bbD_I 
  \end{equation}
 Moreover, the following diagram is commutative   \[ 
 \xymatrix@C=10ex{ & \bbD_I \ar[d]^{\upsilon_{I,J}} \ar[dr]^{\sigma_I} \\ \bbD\bbD_J \ar[ur]^{\Theta_J} \ar[r]^{\tau_J}  & \bbD_J \ar[r]^{\sigma_J}  & X.
 }
 \]
and $\Theta_J$ restricts to an isomorphism of schemes over $X\setminus(\textstyle\sum_{i\in I} D_i)$
$$(\sigma_J\circ \tau_J)^{-1}(X\setminus (\textstyle \sum_{i\in I} D_i)) \stackrel{\Theta_J}{\longrightarrow} \sigma_I^{-1}(X\setminus(\textstyle\sum_{i\in I} D_i)).$$  
\xprop

\pf 
It is straightforward to check that the scheme $\sigma_J\circ \tau_J:\bbD\bbD_J\to X$ belongs to the category $\Sch _{X}^{\{D_i\}_{i\in I}  \text{-reg}}$.  To prove the existence and uniqueness of $\Theta_J$, it suffices, by Proposition \ref{Propunivdef}, to verify that the inclusion 
$$\textstyle \left(\sigma_J \circ \tau_J \right) ^{-1} \left( \sum _{s \in I_{\geq i} } D_s \right) \subset  \left(\sigma_J \circ \tau_J \right) ^{-1} \left(  X_i \right)$$ 
holds in $Clo(\bbD\bbD_J)$ for every $i\in I$. The verification of these inclusions being of local nature with respect to the Zariski topology on $X$, we can reduce without loss of generality to the local setting and notation of subsection \ref{subsec:affine}. Putting $\rho_J^*=\tau_J^*\circ\sigma_J^*$, the inclusion to be verified is then equivalent 
to the property that in the coordinate ring $B_J$ of $\bbD\bbD_J$, we have 
 \begin{equation}\label{eq:inclusion_ideals-3.5(ii)}
     \rho_J^*(M_i) \subset ( \textstyle \rho_J^*(\prod _{s\in I_{\geq i}}d_s))
 \quad \textrm{for all } i\in I.
 \end{equation}
Assumption \ref{assumslid} implies that for all $s\in I$, the image $\pi_i(d_s)\in A/M_i$ of $d_s$ by the quotient homomorphism $\pi_i:A\to A/M_i$ is not a zero-divisor, so that  by Lemma \ref{lem:Kerf_Ji-Cartier} we have 
\begin{equation}\label{YHJ85} \ker (F_J(i)^*) = \sum _{\nu=(\nu_j) \in (\mathbb{Z}_{\geq 0})^J} \frac{\left(\prod_{j \in J}\left(M_j + (\prod_{s \in J_{\geq j}} d_s)\right)^{\nu _j}\right)\cap M_i}{\prod _{j \in J}\left(\prod_{s \in J_{\geq j}} d_s\right) ^{\nu _j}}.  
 \end{equation}
The universal properties of $\sigma_J:\bbD_J\to X$ and $\tau_J:\bbD\bbD_J\to \bbD_J$ translate into the inclusions 
\begin{equation} \label{univprop-ddj}
 \begin{cases}
\textstyle \sigma_J^*(M_j)  \subset \left(\sigma_J^*(\prod_{s\in J_{\geq j}} d_s)\right) &   \textrm{ for all } j\in J \\
 \textstyle \tau_J^*(\ker \left(F_J(i)^* \right)) \subset \left( \tau_J^*(\sigma_J^*(\prod _{s\in (I\setminus J)_{\geq i}} d_s ))\right)=\left( \rho_J^*(\prod _{s\in (I\setminus J)_{\geq i}} d_s )\right) & \textrm{for all } i\in I\setminus J
 \end{cases}
\end{equation}
in $R_J$ and $B_J$ respectively. 

\textbullet First assume that $i\in I\setminus J$. If $J_{\geq i}=\emptyset$ then taking $\nu=(0)$ in \eqref{YHJ85} gives that $\sigma_J^*(M_i)\subset \ker \left(F_J(i)^* \right)$ and hence that 
\begin{equation}
\label{eq:Inclus-I-1}
\rho_J^*(M_i) \subset\tau_J^*(\ker \left(F_J(i)^* \right))\subset \textstyle \left(\rho_J^* (\prod _{s\in (I\setminus J)_{\geq i}} d_s )\right)=\left(\rho_J^*(\prod _{s\in I_{\geq i}} d_s) \right).
\end{equation}
Otherwise, if $J_{\geq i}\neq \emptyset$, then for $j_0=\min\{J_{>i}\}$, we have $J_{\geq j_0}=J_{\geq i}$ and  $M_i=M_{j_0}\cap M_i$. Taking $\nu=(0,\ldots,\underset{j_0}{1},\ldots 0)$ in \eqref{YHJ85} gives that $\sigma_J^*(M_i)\subset  
\left(\textstyle \sigma_J^*(\prod_{s\in J_{\geq i}} d_s)\right) \ker \left(F_J(i)^* \right)$ 
and hence that
\begin{equation}
\label{eq:inclus-I-2}
\begin{array}{lcl}
\rho_J^*(M_i)
& \subset &  
\left(\textstyle \rho_J^*(\prod_{s\in J_{\geq i}} d_s)\right) (\tau_J^*(\ker \left(F_J(i)^* )\right)) \\ & \subset &  \textstyle \left(\rho_J^*(\prod_{s\in J_{\geq i}} d_s)\right)  \left( \rho_J^*(\prod _{s\in (I\setminus J)_{\geq i}} d_s) \right) = \textstyle \left( \rho_J^*(\prod _{s\in I_{\geq i}} d_s) \right).
\end{array}
\end{equation}

\textbullet Now assume that $j\in J$. If $(I\setminus J)_{\geq j}=\emptyset$  
then $\left( \prod _{s\in J _{\geq j}} d_s \right)=\left( \prod _{s\in I _{\geq j}} d_s \right)$, and hence 
\begin{equation}
\label{eq:inclus-J-1}
\rho_J^*(M_j)\subset \textstyle \left(\rho_J^* (\prod _{s\in J _{\geq j}} d_s) \right)= \left( \rho_J^*(\prod _{s\in I _{\geq j}} d_s) \right).
\end{equation}
 Otherwise, if $(I\setminus J)_{\geq j}\neq\emptyset$,  
then for  $i_0=\min\{(I\setminus J)_{\geq j}\}$, we have  $(I\setminus J)_{\geq i_0}=(I\setminus J)_{\geq j}$ and $M_j=M_j\cap M_{i_0}$. Taking  $\nu=(0,\ldots,\underset{j}{1},\ldots 0)$
 
in (\ref{YHJ85}) for $i=i_0$ then gives that  $\sigma_J^*(M_j)\subset \left(\sigma_J^*(\prod_{s\in J_{\geq j} }d_s)\right)\ker(F_J(i_0)^*)$ and hence that
\begin{equation}
    \label{eq:inclus-J-2}
\begin{array}{lcl}
\rho_J^*(M_j) & \subset & \textstyle \left(\textstyle \rho_J^*(\prod_{s\in J_{\geq j}} d_s)\right) \left(\tau_J^*(\ker (F_J(i_0)^*))\right) \\ 
& \subset & \textstyle \left( \rho_J^*(\prod_{s\in J_{\geq j}} d_s)\right)
\left( \rho_J^*(\prod_{s\in (I\setminus J)_{\geq j}} d_s)\right)=\left(\rho_J^* (\prod_{s\in I_{\geq j}} d_s)\right).
\end{array}
\end{equation}
Equations \eqref{eq:Inclus-I-1}, \eqref{eq:inclus-I-2}, \eqref{eq:inclus-J-1} and \eqref{eq:inclus-J-2} imply that the inclusion  \eqref{eq:inclusion_ideals-3.5(ii)} holds for all $i\in I$, which  completes the proof of the existence and uniqueness of the morphism $\Theta_J:\bbD\bbD_J\to \bbD_I$. 

Finally, the equality $ \upsilon_{I,J}\circ \Theta_J=\tau_J$ follows from Corollary \ref{coronumberofmor} and the fact that $\Theta_J$ restricts to an isomorphism of $X\setminus(\sum_{i\in I} D_i)$-schemes follows from Corollary \ref{cor:isocomplement}.
\xpf

\rema \label{fact:restricted-subset} \label{reformulcoro} Under Assumption \ref{assumslid}, the unique morphism of $X_i$-schemes 
\begin{equation*} \label{eq:iso-restricted}
 (\bbD_J)_i=\bbD\left( \left({}^{D_J}_{X_J}\right)\!|_{_{X_i}} {~}_{X_i}\right) \isomto  \bbD \left( \left({}^{D_{J_{>i}}}_{X_{J_{>i}}}\right)\!|_{_{X_i}} {~}_{X_i}\right)=(\bbD_{J_{>i}})_i.
\end{equation*}
given by Proposition \ref{Propunivdef} is an isomorphism.  This implies in turn the existence of a unique isomorphism of $\bbD_J$-schemes
\begin{equation*}
   \bbD\bbD_J=\bbD \left( \left( {}^{\sigma_J^{-1}D_{I\setminus J}}_{(\bbD_J)_{I\setminus J}} \right) {~}_{\bbD_J}\right) \to 
   \bbD \left( \left( {}^{\sigma_J^{-1}D_{I\setminus J}}_{(\bbD_{J_{>i}})_{I\setminus J}}\right) {~}_{\bbD_J}\right) 
\end{equation*} 
through which  the morphism $\Theta_J$ of Proposition \ref{theopanelization-mor} can be re-written in the compact full form
 \begin{equation*} \Theta_J: 
\bbD \left( \left( {}^{\sigma_J^{-1}D_i}_{ \bbD \left( \left( ^{D_j}_{X_j} \right)_{j \in J_{>i }}  {}_{X_i} \right)}  \right)_{i \in I\setminus J}   {~}_{\bbD  \left(\left({}^{D_j}_{X_j}\right)_{j\in J} {~}_X\right)} \right) \to \bbD \left( \left({}^{D_i}_{X_i} \right)_{i \in I} {~~}_{X} \right).
   \end{equation*}
\xrema

 \exam \label{rem:not-iso} 
  Let $I=\{1\leq 2\}$,  $X= \Spec \left(\bbZ [t_1,t_2,x] \right)$, 
 $X_2= V (x)\subset X_1 = V (x^2)$, $D_1= V(t_1)$ and $D_2= V (t_2)$. We then have 
 \begin{align*}
    \bbD_{\{1,2\}}= \bbD \left( ^{D_2,D_1}_{X_2,X_1 ~X} \right)= & \Spec( \bbZ [t_1,t_2, t_2^{-1}x, (t_1 t_2)^{-1}x^2]).
 \end{align*}
 
 We have $\bbD_{\{1\}}=\bbD \left( ^{D_1}_{X_1 ~X}  \right) = \Spec( \bbZ [t_1,t_2, x,  t_1^{-1}x^2])$ 
 and 
 $\bbD\left( \left({}^{D_1}_{X_1}\right)\!|_{_{X_2}} {~}_{X_2}\right)=\Spec( \bbZ [t_1,t_2])$. The morphism $F_{\{1\}}(2) : \bbD\left( \left({}^{D_1}_{X_1}\right)\!|_{_{X_2}} {~}_{X_2}\right) \to \bbD_{\{1\}}$ is the immersion as the closed subscheme with ideal $(t_1^{-1}x^2,x)$. It follows that 
 \begin{align*}
     \bbD\bbD_{\{1\}}= \bbD\left( {}^{\sigma_1^{-1}D_2}_{{\bbD\left( \left({}^{D_1}_{X_1}\right)\!|_{_{X_2}} {~}_{X_2}\right)}}  {~}_{\bbD_{\{1\}}} \right) = & \Spec (\bbZ [t_1,t_2, t_2^{-1}(t_1^{-1}x^2), t_2^{-1} x ])
 \end{align*}
and that the morphism $\Theta_{ \{1\}}:\bbD\bbD_{\{1\}}\to  \bbD_{\{1,2\}}$ is an isomorphism. 
 
On the other hand, we have $$\bbD_{\{2\}}=\bbD \left( ^{D_2}_{X_2 ~X}  \right) = \Spec( \bbZ [t_1,t_2, t_2^{-1}x]) \;\textrm{ and }\;  
\bbD\left( \left({}^{D_2}_{X_2}\right)\!|_{_{X_1}} {~}_{X_1}\right)= \Spec (\bbZ [t_1,t_2,t_2^{-1}x]/((t_2^{-1}x)^2)).$$ 
The morphism
 $F_{\{2\}}(1) : \bbD\left( \left({}^{D_2}_{X_2}\right)\!|_{_{X_1}} {~}_{X_1}\right) \to \bbD_{\{2\}}$
 is the immersion as the closed subscheme with ideal $((t_2^{-1}x)^2)$. It follows that 
\begin{align*}
     \bbD\bbD_{\{2\}}= \bbD \left( ^{\sigma_2^{-1}D_1}_{ \bbD \left(^{D_2}_{X_2 ~ X_1} \right) ~ \bbD_{\{2\}} }\right)  = & \Spec (\bbZ [t_1,t_2, t_2^{-1} x, t_1^{-1}(t_2^{-1}x)^2 ])
\end{align*} 
 and that $\Theta_{ \{2\}}:\bbD\bbD_{\{2\}}\to  \bbD_{\{1,2\}}$ 
 is the dilatation of $\bbD_{\{1,2\}}$ with center  $[V((t_1t_2)^{-1} x^2), V(t_2)]$, which is not an isomorphism. 
\xexam

 \subsection{Panelization isomorphisms}
We continue with the notation and assumptions of the previous subsections. Example \ref{rem:not-iso} illustrates  that the panelization morphisms $\Theta_J:\bbD\bbD_J\to \bbD_I$ of Proposition \ref{theopanelization-mor} need not be isomorphisms in general. Our aim in this subsection is to provide sufficient conditions on a subset of indices $J\subset I$ under which $\Theta_J:\bbD\bbD_J\to \bbD_I$ will be an isomorphism. 

As it will appear more clearly below, such conditions can be interpreted as certain regularity properties on the collections of closed subschemes and locally principal divisors in a deformation datum on the scheme $X$. 
We proceed in two steps: first, in subsection \ref{sec:local-panelization} we identify   conditions under which for a given subset $J\subset I$ the morphism $\Theta_J:\bbD\bbD_J\to \bbD_{I}$ of Proposition \ref{theopanelization-mor} is an isomorphism. Then, in subsection \ref{secset}, we provide examples of global geometric hypotheses on a deformation datum on a scheme $X$ which guarantee that all the panelization morphisms $\Theta_J:\bbD\bbD_J\to \bbD_{I}$ are isomorphisms, independently of the subset $J$.

\subsubsection{Panelization with respect to a subset of indices} \label{sec:local-panelization}
To formulate the main result of this subsection, we introduce the following specific technical condition on a deformation datum on a scheme $X$, which is a globalization of the two properties (a) and (b) in Proposition  \ref{prop:Kernel-formula} (ii).

\defi \label{assutheopan3} Let $\left({}^{D_I}_{X_I}\right)$ be a deformation datum on a scheme $X$ satisfying Assumption \ref{assumslid} and let $\mathcal{M}_i\subset \mathcal{O}_X$ and $\mathcal{I}_i\subset \mathcal{O}_X$ be the ideal sheaves of the subschemes $X_i$ and $D_i$ of $X$, $i\in I$. We say that  $\left({}^{D_I}_{X_I}\right)$ is \emph{dilatation-regular} with respect to a subset $J\subset I$ if for all $i \in I\setminus J$ such that $J_{\geq i} \ne \emptyset$, the following hold:
\begin{itemize}
    \item[$\quad (R_1)$] \label{assut3item2} For every $\alpha=(\alpha_j) \in (\bbZ_{\geq 0})^{J_{>i}}\setminus\{0\}$ , the following equality of ideal sheaves holds in $\mathcal{O}_X$: 
     $$  \textstyle \left(\prod_{j\in J_{>i}} \mathcal{M}_j^{\alpha_j}\right)\cap \mathcal{M}_i=\left(\prod_{j\in J_{>j_0}} \mathcal{M}_j^{\alpha_j}\right)\mathcal{M}_{j_0}^{\alpha_{j_0}-1} \mathcal{M}_i,$$   
  where $j_0 = \min \{ j \in J_{>i} \textrm{ such that } \alpha_{j} \neq 0 \}.$     
     \item[$(R_2)$]  \label{assut3item1}  For every collection of pairs $(\alpha^{(e)},\beta^{(e)})=((\alpha_{j'}^{(e)}),(\beta_j^{(e)}))\in (\mathbb{Z}_{\geq 0})^{J_{>i}}\times (\mathbb{Z}_{\geq 0})^{J}$ indexed by the elements $e$ of a finite  set $E$,  the following equality of ideal sheaves holds in $\mathcal{O}_X$: 
     {\small 
     $$ \textstyle \left(\sum_{e\in E} \left(\prod_{j'\in J_{>i}}  \mathcal{M}_{j'}^{\alpha_{j'}^{(e)}}\right)\left(\prod_{j\in J} \mathcal{I}_j^{\beta_j^{(e)}}\right)\right)\cap \mathcal{M}_i=\sum_{e\in E}  \left(\left(\prod_{j'\in J_{>i}}  \mathcal{M}_{j'}^{\alpha_{j'}^{(e)}}\right)\cap \mathcal{M}_i\right)\left(\prod_{j\in J} \mathcal{I}_j^{\beta_j^{(e)}}\right).$$   
     }
\end{itemize}
 \xdefi

Any deformation datum $\left({}^{D_I}_{X_I}\right)$ satisfying Assumption \ref{assumslid} is automatically dilatation regular with respect to every subset $J\subset I$ such that $(I\setminus J)>J$ (i.e. $i>j$ in $I$ for all $(i,j) \in (I\setminus J) \times J$).  Example   \ref{rem:not-iso} provides on the other hand an illustration of a deformation datum which is not dilatation regular with respect to all subsets of the index set $I$, namely: 

 \exam Consider the deformation datum  
 $X_2= V (x)\subset X_1 = V (x^2)$, $D_1= V(t_1)$ and $D_2= V (t_2)$ on $X= \Spec \left(\bbZ [t_1,t_2,x] \right)$ of Example \ref{rem:not-iso}. By the remark above, this deformation datum is automatically dilatation regular with respect to the subset $J=\{1\}$ of $I=\{1,2\}$. On the other hand, since the ideal $M_2^2\cap M_1=M_1=M_2^2$ properly contains the ideal $M_2M_1=M_2^3$, condition (R1) in Definition \ref{assutheopan3} fails for $J=\{2\}$, $I\setminus J=\{1\}$ and $\alpha=(2)$, implying that this deformation datum is not dilatation-regular with respect to the subset $J=\{2\}$ of $I=\{1,2\}$.  
 \xexam
 
 \smallskip

 \theo \label{theo-iso-panelization} 
 Let $\left({}^{D_I}_{X_I}\right)$ be a deformation datum on a scheme $X$ satisfying Assumption \ref{assumslid} and  dilatation-regular with respect to a subset $J\subset I$. Then the morphism $\Theta_J:\bbD\bbD_J\to \bbD_I$ of Proposition \ref{theopanelization-mor} is an isomorphism.
\xtheo

\pf 
We will show that under our hypotheses, the morphism $\upsilon_{I,J}:\bbD_I\to \bbD_J$ factors  through a $\bbD_J$-morphism $\Upsilon_I:\bbD_I\to \bbD\bbD_J$ which is an inverse for $\Theta_J:\bbD\bbD_J\to \bbD_I$. 
The scheme $\upsilon_{I,J}:\bbD_I\to \bbD_J$ belongs to the category $\Sch_{\bbD_J}^{\sigma_J^{-1}D_{I\setminus J} \text{-reg}}$ by Lemma  \ref{lem:nu_IJ}. In view of the definition of the $\bbD_J$-scheme $$\tau_J:\bbD\bbD_J=\bbD \left( \left({}^{\sigma_J^{-1}D_{i}}_{(\bbD_J)_{i}}\right)_{i\in I\setminus J} {~}_{\bbD_J}\right) \to \bbD_J, $$ 
Proposition \ref{Propunivdef} will provide the existence of a unique $\bbD_J$-morphism $\Upsilon_I:\bbD_I\to \bbD\bbD_J$ factoring $\upsilon_{I,J}$ provided that we show that the following inclusion 
\begin{equation}    
\label{eq:univ-condition-upsilon-IJ}
 \textstyle \sigma_I^{-1} \left( \sum _{s \in (I\setminus J)_{\geq i}}  D_s \right)=\upsilon_{I,J}^{-1}\left( \sum _{s \in (I\setminus J)_{\geq i}}  \sigma_J^{-1}(D_s) \right)\subset  \upsilon_{I,J} ^{-1} 
 \left( (\bbD_J)_i\right),
 \end{equation}
holds in $Clo(\bbD_I)$ for every $i\in I\setminus J$. The verification of these inclusions being of local nature with respect to the Zariski topology on $X$, we can proceed as in the proof of Proposition \ref{theopanelization-mor} by reducing to the local setting and notation of subsection \ref{subsec:affine}. In this setting, the assertion to verify becomes that for every $i\in I\setminus J$, the image by the $A$-algebra homomorphism $\upsilon_{I,J}^*:R_J\to R_I$ of the kernel of the homomorphism $F_J(i)^*$ is contained in the ideal $\left(\sigma_I^*(\prod_{s\in (I\setminus J)_{\geq i}}d_s)\right)$ of $R_I$. Assumption \ref{assumslid} together with properties $(R_1)$ and $(R_2)$ in Definition \ref{assutheopan3} imply that the conditions 
of Proposition \ref{prop:Kernel-formula} are satisfied, so that we have, for every $i\in I \setminus J$, 
\begin{equation}
\ker(F_J(i)^*)= \left(\frac{M_i}{\prod_{s\in J_{\geq i}}d_s}+\sum_{j\in J} \frac{M_j\cap M_i}{\prod_{s\in J_{\geq j}} d_s}\right).
\end{equation}
By Proposition \ref{Propunivdef}, we have $\sigma_I^*(M_i)\subset \left(\sigma_I^*(\prod_{s\in I_{\geq i}} d_s)\right)$ for every $i\in I$. This immediately implies that 
\begin{equation} \label{eq:term-1}
\upsilon_{I,J}^*\left(\frac{M_i}{\prod_{s\in J_{\geq i}}d_s}\right)\subset \textstyle \left(\sigma_I^*(\prod_{s\in (I\setminus J)_{\geq i}} d_s)\right)   
\end{equation}
Now consider a term of the form $M_j\cap M_i$. 

If $j\geq i$, then $M_j \cap M_i=M_i$ and $J_{\geq j}\subset J_{\geq i}$ so that we have 
\begin{equation}\label{eq:term-2-case1}
\upsilon_{I,J}^*\left(\frac{M_j\cap M_i}{\prod_{s\in J_{\geq j}}d_s}\right)=\upsilon_{I,J}^*\left(\frac{M_i}{\prod_{s\in J_{\geq j}}d_s}\right)\subset \textstyle \left(\sigma_I^*(\prod_{s\in (I\setminus J)_{\geq i} \cup J_{\geq i}\setminus J_{\geq j}} d_s) \right)\subset \left(\sigma_I^*(\prod_{s\in (I\setminus J)_{\geq i}} d_s)\right).
\end{equation}

Otherwise, if $j\leq i$, then $M_j\cap M_i=M_j$ and $(I\setminus J)_{\geq i} \subset (I\setminus J)_{\geq j}$ so that we have 
\begin{equation} \label{eq:term-2-case2}
\upsilon_{I,J}^*\left(\frac{M_j\cap M_i}{\prod_{s\in J_{\geq j}}d_s}\right)=\upsilon_{I,J}^*\left(\frac{M_j}{\prod_{s\in J_{\geq j}}d_s}\right)\subset \textstyle \left(\sigma_I^*(\prod_{s\in (I\setminus J)_{\geq j}} d_s) \right)\subset \left(\sigma_I^*(\prod_{s\in (I\setminus J)_{\geq i}} d_s)\right).
\end{equation}
Equations \eqref{eq:term-1}, \eqref{eq:term-2-case1} and \eqref{eq:term-2-case2} imply that the inclusion $\upsilon_{I,J}^*(\ker(F_J(i)^*))\subset \left(\sigma_I^*(\prod_{s\in (I\setminus J)_{\geq i}}d_s)\right)$ holds in $R_I$ for every $i\in I\setminus J$ and hence, the existence of a unique $\bbD_J$-morphism $\Upsilon_I:\bbD_I\to \bbD\bbD_J$ factoring $\upsilon_{I,J}$. To complete the proof, it remains to observe that  $\Theta_J\circ\Upsilon_I$ is then an $X$-endomorphism of $\bbD_I$ and $\Upsilon_I\circ\Theta_J$ a $\bbD_J$-endomorphism of $\bbD\bbD_J$, which, by Corollary \ref{coronumberofmor} applied over $X$ and over $\bbD_J$ respectively, must be equal to the identity.
\xpf
 
\coro\label{KsupS} Let $\left({}^{D_I}_{X_I}\right)$ be a deformation datum on a scheme $X$ satisfying Assumption \ref{assumslid} and let $J\subset I$ be a subset such that $(I\setminus J)> J$. Then the panelization morphism  
$\Theta_J:\bbD\bbD_J \to \bbD_I$ of Proposition \ref{theopanelization-mor} is an isomorphism which can be re-written in the form  
\begin{equation}
\label{eq:panel-reformulation}
 \Theta_J:\bbD \left( \left( {}^{\sigma_J^{-1}D_i}_{X_i}  \right)_{i \in I\setminus J} ~ {~}_{\bbD  \left(\left({}^{D_j}_{X_j}\right)_{j\in J} {~}_X\right)} \right)\isomto  \bbD \left( \left({}^{D_i}_{X_i} \right)_{i \in I} {~~}_{X} \right).
\end{equation}
\xcoro
\pf
Since $(I\setminus J)>J$ we have $J_{\geq i}=\emptyset$ for all $i\in I\setminus J$.  The deformation datum $\left({}^{D_I}_{X_I}\right)$ is thus dilatation regular with respect to $J$, and the conclusion that $\Theta_J$ is an isomorphism then follows from Theorem \ref{theo-iso-panelization}.
The reformulation \eqref{eq:panel-reformulation} follows from Remark \ref{reformulcoro}.
\xpf 

\exam \label{ex:iterated} Let $I=(\{1,\ldots, n\},\leq)$ and let 
$$\bbD \left( \left({}^{D_i}_{X_i} \right)_{i \in I} {~~}_{X} \right)=\bbD \big (^{D_n, \ldots, D_1} _{X_n,  \ldots, X_1  ~ X } \big)$$
(see \eqref{eq:ordered-notation} for the notation) be the multi-centered deformation space associated to a deformation datum $\left({}^{D_I}_{X_I}\right)$ on $X$ which satisfies Assumption \ref{assumslid}. Consider the exhaustion of $I$ by the subsets $I_s=\{1,\ldots, s\}$, $s=1,\ldots,n$. For every $s\geq 2$, applying Corollary \ref{KsupS} to the pair of subsets $I=I_s$ and $J=I_{s-1}$ gives an isomorphism of $X$-schemes
 $$\Theta_{I_{s-1}}: \bbD\bbD_{I_{s-1}}= \bbD \big (^{D_s} _{X_s, ~  \bbD_{I_{s-1}}} \big) \to \bbD_{I_s}$$
where we wrote $D_s$ instead of $\sigma_{I_{s-1}}^{-1} D_s$ to lighten the notation. Combining all these isomorphisms together, we obtain a unique isomorphism of $X$-schemes
\begin{equation}
\Theta :    \mathbb{D}\left(\begin{array}{cc}
D_{n}\\
X_{n}, & \mathbb{D}\left(\begin{array}{cc}
D_{n-1}\\
X_{n-1}, & \mathbb{D}\left(\begin{array}{cc}
D_{n-2}\\
X_{n-2}, & \mathbb{D}\cdots
\end{array}\right)
\end{array}\right)
\end{array}\right)
\to \bbD \big (^{D_n, \ldots, D_1} _{X_n,  \ldots, X_1  ~ X } \big)
\end{equation}
which expresses $\bbD \big (^{D_n, \ldots, D_1} _{X_n,  \ldots, X_1  ~ X } \big)$ as an iterated mono-centered deformation space.
\xexam
 
 \subsubsection{A setting for unconditional panelization} \label{secset}
The next theorem provides an algebro-geometric setting in which the dilatation-regularity conditions in Definition \ref{assutheopan3} are satisfied for all subsets $J$ of $I$, so that the panelization morphisms $\Theta_J:\bbD\bbD_J \to \bbD_I$ are guaranteed to be all isomorphisms, independently of the choice of the subset $J$.

 \theo \label{coroset} 
Let $I $ be the totally ordered set $\{ 1 , \ldots , n\}$ and let $\left({}^{D_I}_{X_I}\right)$ be a deformation datum on a locally Noetherian scheme $X$. Denote by $\mathcal{M}_i\subset \mathcal{O}_X$ and $\mathcal{I}_i\subset \mathcal{O}_X$ the ideal sheaves of the subschemes $X_i$ and $D_i$, $i\in I$. Assume that for every point $y\in X$ there exists a regular sequence, possibly empty, contained in the maximal ideal of $\mathcal{O}_{X,y}$, consisting of a family $x$ and of elements $d_{i'}$ indexed by a subset $I'\subset I$ which satisfy the following properties:
\begin{enumerate}[label=(\arabic*)]
 \item Either $\mathcal{M}_{i,y}=\mathcal{O}_{X,y}$ for all $i\in I$ or there exists $i_0\in I$ such that the ideals $\mathcal{M}_{i,y}\subset \mathcal{O}_{X,y}$, $1\leq i\leq  i_0$, are generated by nested, possibly empty, subfamilies of $x$, and $\mathcal{M}_{i,y}=\mathcal{O}_{X,y}$ for all $i>i_0$, 
 \item The ideal $\mathcal{I}_{i',y}\subset \mathcal{O}_{X,y}$ equals $(d_{i'})$ for $i'\in I'$ and $(1)$ otherwise. 
\end{enumerate} 
Then Assumption \ref{assumslid} is satisfied and the panelization morphisms $\Theta_J:\bbD\bbD_J\to \bbD_I$ of  Proposition  \ref{theopanelization-mor} are isomorphisms for all subsets $J\subset I$.
 \xtheo 
  \pf Both assertions can be checked on stalks at points $y\in X$. If $\mathcal{M}_{i,y}=(1)$ or $\mathcal{I}_{i',y}=(1)$ then $y$ does not belong to  $X_i\cap D_{i'}$. Otherwise, the assumption implies that $d_{i'}$ is a non-zero-divisor modulo $\mathcal{M}_{i,y}$ since regular sequences in a Noetherian local ring can be permuted. So Assumption \ref{assumslid} is satisfied at $y$. In view of Theorem \ref{theo-iso-panelization}, it is enough to prove that the conditions of Definition \ref{assutheopan3} are satisfied for all $J\subset I$. If $\mathcal{M}_{i,y}=(1)$ then $\mathcal{M}_{j,y}=(1)$ for all $j>i$, and both conditions are immediate. Assume that $\mathcal{M}_{i,y}\neq (1)$. In $(R_1)$, factors of index $j>i_0$ equal $(1)$. If $\alpha_j=0$ for all $j\leq i_0$ then both sides equal $\mathcal{M}_{i,y}$. Otherwise, the equality follows from Proposition \ref{prop:intersection_inclusions}. In $(R_2)$, all factors equal to $(1)$ disappear and the equality follows from Corollary \ref{coroap}. 
 \xpf 
 
 \exam \label{ex:vector-bundles} Let $C$ be an effective SNC Cartier divisor with irreducible components $C_i$, $i=1,\ldots , n$ on a locally Noetherian scheme $S$ and let $E_n\to E_{n-1}\to \cdots \to E_1\to E_0$ be closed immersions of vector bundles $\mathrm{p}_i:E_i=\mathbb{V}(\mathcal{E}_i)\to S$ on $S$ determined by surjections $\mathcal{E}_0\to \mathcal{E}_1\to \cdots \to \mathcal{E}_{n-1}\to \mathcal{E}_n$ between coherent locally free sheaves on $S$. Put $D_i=p_0^{-1}{C_i}$.  Then $\left({}^{D_I}_{E_I}\right)$ is a deformation datum on $E_0$ which satisfies the assumptions of Theorem \ref{coroset}.  
 \xexam
 
 \exam \label{ex:Def-space} Let $Z_n\to Z_{n-1}\to \cdots \to Z_1\to Z_0$ be a sequence of regular closed immersions between locally Noetherian schemes over a base scheme $B$. Let $\mathbb{A}^n_B=\mathrm{Spec}_B(\mathcal{O}_{B}[t_1,\ldots, t_n])$, $X_\ell=Z_\ell\times _B \mathbb{A}^n_B$, $\ell=0,1\ldots, n$, and for every $i=1,\ldots,n$, let $D_i=Z_0\times_B V(t_i)$. Then $\left({}^{D_I}_{X_I}\right)$ is a deformation datum on $X_0$ which satisfies the assumptions of  Theorem \ref{coroset}. 
\xexam
 
\section{Strata of multi-centered deformation spaces} \label{secstrata}

Given a deformation datum $\left({}^{D_I}_{X_I}\right)$ on a scheme $X$ indexed by a finite totally ordered set $(I,\leq)$, our goal in this section is to describe under suitable hypotheses the structure of the restrictions 
$\bbD_I\times_X \Sigma_S$ of the multi-centered deformation space $$\sigma_I:\bbD_I=\bbD \left( \left({}^{D_I}_{X_I}\right){~}_{X}\right)\to X$$ to the subschemes $\Sigma_S:=\bigcap_{s\in S} D_s$, where $S$ ranges through the subsets of $I$. We henceforth call the $\Sigma_S$-scheme $\bbD_I\times_X \Sigma_S$ the \emph{$S$-stratum} of $\bbD_I$.

By analogy  with the case $\sharp I=1$ in which $\bbD_I\times_X \Sigma_I$ is the exceptional divisor of the dilatation morphism   $\sigma:\bbD \big (^{D_1} _{X_1 ~ X } \big)\to X$ over the subscheme $X_1\cap D_1$, we call the $\Sigma_I$-scheme 
\begin{equation}\label{eq:exc-stratum-def}
\overline{\sigma}_I:\bbV \left( \left({}^{D_I}_{X_I}\right){~}_{X}\right):= \bbD_I\times_X \Sigma_I\to \Sigma_I 
\end{equation} 
the \emph{exceptional stratum} of the multi-centered deformation space $\bbD_I$. By \cite[Proposition 3.16]{Ma24d}, the above morphism $\overline{\sigma}_I$ factors through the  closed subscheme $O:=(\bigcap_{i\in I}X_i) \cap \Sigma_I$ of $X$ and induces a canonical isomorphism  of $O$-schemes
\begin{equation} \label{eq:bbV_def}
\bbV \left( \left({}^{D_I}_{X_I}\right){~}_{X}\right)= \bbD_I\times_X \Sigma_I \cong \bbD_I\times_X O.
\end{equation}
In this section, we will apply the panelization isomorphisms established in the previous  section to provide descriptions of the strata $\bbD_I\times_X \Sigma_S$ of $\bbD_I$, where $S$ ranges through subsets of $I$. 

\subsection{Panel strata}\label{panel-strat} 

Recall Definition \ref{def:panel-datum} that for a subset $J$ of $I$, the $J$-th panel of $\bbD_I$ is the $\bbD_J$-scheme 
 $$\tau_{J}: \bbD\bbD_J=
\bbD \left( \left({}^{\sigma_J^{-1}D_{I\setminus J}}_{(\bbD_J)_{I\setminus J}}\right) {~}_{\bbD_J}\right)= \bbD \left( \left( {}^{\sigma_J^{-1}D_i}_{ \bbD \left( \left( ^{D_j}_{X_j} \right)_{j \in J_{>i }}  {}_{X_i} \right)}  \right)_{i \in I\setminus J}  {~}_{\bbD_J}\right) \to \bbD_J. $$

\lemm \label{lem:stras} 
Let $K\subset J$ be a subset such that the scheme $\bbD\bbD_J  \times _{X} \Sigma_K  \to  \Sigma_K$ belongs to the category $\Sch_{\Sigma_K}^{\{D_i\cap \Sigma_K\}_{i\in I \setminus J}\text{-reg}}$. Then there exists a unique isomorphism of  $\bbD_J \times_X \Sigma_K $-schemes 
\begin{equation} \label{hfraks}
\eta_{J,K} : \bbD\bbD_J \times_X \Sigma_K \isomto \bbD \left( \left({}^{\sigma_J^{-1}D_{I\setminus J}\times_X \Sigma_K}_{(\bbD_J)_{I\setminus J}\times_X \Sigma_K}\right) {~}_{\bbD_J\times_X \Sigma_K}\right).
   \end{equation}
\xlemm
\pf
The assumption implies that  $\bbD\bbD _J  \times _{\bbD_J} (\bbD_J\times_X \Sigma_K)\to \bbD_J\times_X \Sigma_K$ belongs to the category  $\Sch_{\bbD_J\times_X \Sigma_K}^{\{\sigma_J^{-1}D_i\times_X \Sigma_K\}_{i \in I \setminus J}\text{-reg}}$. The existence and uniqueness of $\eta_{J,K}$  then follow from Lemma \ref{lem:base-change} and Corollary \ref{coronumberofmor} respectively.
\xpf

The next two propositions build on different choices of panelizations to provide,  under appropriate regularity conditions, a description of the strata $\bbD_I\times_X \Sigma_S$ for proper subsets $S$ of $I$ in two complementary ways: either in the form  of exceptional strata of suitable restricted multi-centered deformation spaces or in the form of multi-centered deformation spaces of suitable exceptional strata.  
\prop
\label{th:vb-strata}
Let $\left({}^{D_I}_{X_I}\right)$ be a deformation datum on a  scheme $X$ indexed by a finite totally ordered set $(I,\leq)$  which satisfies Assumption \ref{assumslid}. Let $S\subset I$ be a subset such that $\left({}^{D_I}_{X_I}\right)$ is dilatation-regular with respect to $I\setminus S$  (see Definition \ref{assutheopan3}). Then the inverse of the panelization morphism $\Theta_{I\setminus S}$ induces an isomorphism of $\bbD \left(\left({}^{D_j}_{X_j} \right)_{j \in (I\setminus S)_{>\max(S)}}  {}_{X_{\max(S)}} \right)\times_X \Sigma_S$-schemes
\begin{equation}
\label{fraksVB}    
\hat{\mathfrak{S}}_{I\setminus S}: \bbD_I \times_X \Sigma_S 
    \isomto \bbV \left( \left( {}^{\sigma_{I\setminus S}^{-1}D_s}_{\bbD \left( \left({}^{D_j}_{X_j} \right)_{j \in (I\setminus S)_{>s}}  {}_{X_s} \right)}  \right)_{s \in S}   {~}_{\bbD_{I\setminus S}} \right)
\end{equation}
\xprop
\pf
Since $\left({}^{D_I}_{X_I}\right)$ is dilatation-regular with respect to $I\setminus S$, Theorem \ref{theo-iso-panelization} implies that the panelization morphism $$\Theta_{I\setminus S}:\bbD\bbD_{I\setminus S}=\bbD \left(\left( {}^{\sigma_{I\setminus S}^{-1}D_s}_{ \bbD \left( \left( ^{D_j}_{X_j} \right)_{j \in (I\setminus S)_{>s }}  {}_{X_s} \right)}  \right)_{s \in S}   {~}_{\bbD_{I\setminus S}} \right)\to \bbD_I$$ of Proposition \ref{theopanelization-mor} is an isomorphism of $\bbD_{I\setminus S}$-schemes. 
Let $\sigma_{I\setminus S}: \bbD_{I\setminus S}\to X$ and $\tau_{I\setminus S}:\bbD\bbD_{I\setminus S}\to \bbD_{I\setminus S}$ be the dilatation morphisms. Then, by definition of the exceptional stratum of $\tau_{I\setminus S}$, we have 
$$ \bbD\bbD_{I\setminus S}\times_X \Sigma_S= \bbD\bbD_{I\setminus S}\times_{\bbD_{I\setminus S}} \sigma_{I\setminus S}^{-1}(\Sigma_S)=\bbV \left( \left({}^{\sigma_{I\setminus S}^{-1}D_s}_{\bbD \left( \left({}^{D_j}_{X_j} \right)_{j \in (I\setminus S)_{>s }}  {}_{X_s} \right)}  \right)_{s \in S}   {~}_{\bbD_{I\setminus S}} \right),$$
where the scheme on the right hand side is viewed as a scheme over 
$$\left(\bigcap_{s\in S} \bbD \left( \left( {}^{D_j}_{X_j} \right)_{j \in (I\setminus S)_{>s }} {}_{X_s} \right)\right) \cap \sigma_{I\setminus S}^{-1}(\Sigma_S)\cong\bbD \left(\left({}^{D_j}_{X_j} \right)_{j \in (I\setminus S)_{>\max(S)}}  {}_{X_{\max(S)}} \right)\times_X \Sigma_S,$$
and the assertion follows. 
\xpf

\prop  \label{higher_strata} \label{cor:panel-strata-1}  
Let $\left({}^{D_I}_{X_I}\right)$ be a deformation datum on a  scheme $X$ indexed by a finite totally ordered set $(I,\leq)$  which satisfies Assumption \ref{assumslid}. Let $S\subset I$ be a subset such that $\left({}^{D_I}_{X_I}\right)$ is dilatation-regular with respect to $S$  and such that the scheme $\bbD\bbD_S \times _{X} \Sigma_S  \to  \Sigma_S$ belongs to the category $\Sch_{\Sigma_S}^{\{D_i\cap \Sigma_S\}_{i\in I \setminus S}\text{-reg}}$.

Then there exists a unique isomorphism of 
$\bbD_S\times_X \Sigma_S$-schemes
\begin{equation} \label{frakss}
\mathfrak{S}_S : \bbD_I\times_X \Sigma_S \isomto \bbD \left( \left( {}^{\sigma_S^{-1}D_i\times_X \Sigma_S}_{ \bbD \left( \left( ^{D_s}_{X_s} \right)_{s \in S_{>i }}
  {}_{X_i} \right) \times_X \Sigma_S }  \right)_{i \in I\setminus S} {}_{\bbD_S\times_X \Sigma_S } \right),
   \end{equation}
and a unique induced isomorphism 
of $\bbV \left( \left({}^{D_S}_{X_S}\right){~}_{X}\right)$-schemes 
\begin{equation} \label{frak2s}
    \hat{\mathfrak{S}}_S:\bbD_I \times_X \Sigma_S \stackrel{\sim}{\rightarrow}  \bbD \left( \left( {}^{{\bar{\sigma}}_S^{-1}D_i}_{\bbV \left(\left({} ^{D_j}_{X_j} \right)_{j \in S_{>i}} {}_{X_i} \right) \times_{X_i}  \bigcap_{s \in S_{\leq i}} D_s|_{X_i}}  \right)_{i \in I\setminus S}  {~}_{\bbV  \left(\left({}^{D_S}_{X_S}\right) {~}_{X}\right)} \right).
\end{equation}
\xprop 
\pf
Since $\left({}^{D_I}_{X_I}\right)$ is dilatation-regular, Theorem \ref{theo-iso-panelization} implies that the panelization morphism $\Theta_S:\bbD\bbD_S\to \bbD_I$ of Proposition \ref{theopanelization-mor} is an isomorphism of $\bbD_S$-schemes, with inverse $\Upsilon_I$. 
The assumption
that the scheme $\bbD\bbD_S \times _{X} \Sigma_S  \to  \Sigma_S$ belongs to the category $\Sch_{\Sigma_S}^{\{D_i\cap \Sigma_S\}_{i\in I \setminus S}\text{-reg}}$ implies that   $\bbD\bbD _S  \times _{\bbD_S} \bbD_S\times_X \Sigma_S\to \bbD_S\times_X \Sigma_S$ belongs to the category  $\Sch_{\bbD_S\times_X \Sigma_S}^{\{\sigma_S^{-1}D_i\times_X \Sigma_S\}_{i \in I \setminus S}\text{-reg}}$ which in turn implies, by Lemma  \ref{lem:stras}, the existence 
of a unique isomorphism
\begin{equation} \label{hfraks-1}
\eta_{S,S} : \bbD\bbD_S \times_X \Sigma_S \isomto \bbD \left( \left({}^{\sigma_S^{-1}D_{I\setminus S}\times_X \Sigma_S}_{(\bbD_S)_{I\setminus S}\times_X \Sigma_S}\right) {~}_{\bbD_S\times_X \Sigma_S}\right)
   \end{equation}
of $\bbD_S\times_X \Sigma_S$-schemes.  The first isomorphism $\mathfrak{S}_S$ is then obtained as the composition $\eta_{S,S}\circ(\Upsilon_I\times \mathrm{id})$, combined with the identification of Remark \ref{reformulcoro}. The second isomorphism follows from the first,  together with the isomorphisms 
\begin{equation*}
\begin{split}
   \bbD \left( \left( ^{D_s}_{X_s} \right)_{s \in S_{>i }}
  {}_{X_i} \right) \times_X \Sigma_S & \cong  \bbD \left( \left( ^{D_s}_{X_s} \right)_{s \in S_{>i }}
  {}_{X_i} \right)\times_X \bigcap_{s\in S_{>i}} D_s\times _X \bigcap_{s \in S_{\leq i}} D_s 
    \\ 
    & \cong  \bbV \left( \left( ^{D_j}_{X_j} \right)_{j \in S_{>i }}  {}_{X_i} \right)\times_{X_i} \bigcap_{s \in S_{\leq i}} D_s|_{X_i}. 
\end{split}
\end{equation*}
\xpf 

\exam \label{cod1strata} Let $s\in I=(\{1,\ldots, n\},\leq )$ be such that hypotheses of Proposition \ref{th:vb-strata} and  Proposition \ref{higher_strata} are satisfied. Then there exist:
\begin{enumerate}[label=(\roman*)] 
    \item A canonical isomorphism of $\bbD \Big({}^{D_n, \ldots, D_{s+1}}_{X_n, \ldots, X_{s+1}, X_s}\Big)\times_X D_s$-schemes
    \begin{equation}
 \label{codim1-formula0}   
\hat{\mathfrak{S}}_{I\setminus \{s\}}: \bbD \Big({}^{D_n , \ldots, D_1}_{X_n , \ldots, X_1, X}\Big) \times_{X} D_s 
    \isomto  \bbV \left({}^{\sigma_{I\setminus \{s\}}^{-1}D_s}_{\bbD \Big({}^{D_n, \ldots, D_{s+1}}_{X_n, \ldots, X_{s+1}, X_s}\Big)}    {~}_{\bbD \Big({}^{D_n, \ldots, D_{s+1}, D_{s-1}, \ldots, D_1}_{X_n, \ldots , X_{s+1} ,X_{s-1}, \ldots, X_1, X}  \Big)}  \right).
\end{equation}
    \item A unique isomorphism of $\bbV \left( ^{D_s}_{X_s, X} \right)$-schemes
\begin{equation} \label{codim1-formula1}
\hat{\mathfrak{S}}_s:\bbD \Big( {}^{D_n , \ldots, D_1}_{X_n , \ldots, X_1, X}\Big) \times _{X} D_s \isomto \bbD \Big({}^{D_n\cap D_s, \ldots, D_{s+1}\cap D_s, ~D_{s-1}\cap D_s,~~~ \ldots, D_1\cap D_s}_{X_n \cap D_s , \ldots , X_{s+1} \cap D_s ,~ \bbV \big(^{D_s}_{X_s, X_{s-1}} \big), \ldots, \bbV \big(^{D_s}_{X_s, X_{1}} \big), \bbV \big(^{D_s}_{X_s, X} \big)} \Big).
\end{equation}
\end{enumerate}
\xexam

\section{Deformation space of a chain of closed immersions between smooth schemes}\label{sec:Verdier-Rost}
In this section, we describe the specialization of some of the results obtained in the previous sections in the situation of the deformation  space ``à la Verdier-Rost'' of a collection  of closed immersions  between schemes smooth over a locally Noetherian base scheme. 

\begin{assum} \label{VR-Def}
In what follows, we consider a collection of closed immersions $$((Z_i)_{i\in I},Z_0): \qquad Z_n\hookrightarrow Z_{n-1}\hookrightarrow \cdots \hookrightarrow Z_1 \hookrightarrow Z_0,$$ 
$I=(\{1,\ldots, n\},\leq)$,
of schemes smooth over a locally Noetherian base scheme $B$. We denote by $\mathcal{I}_{Z_1}\subset \cdots \subset \mathcal{I}_{Z_n}\subset \mathcal{O}_{Z_0}$ the ideal sheaves of the corresponding closed subschemes of $Z_0$.  

We put $\mathbb{A}^n_B=\mathbb{A}^n_{B,t_I}=B\times_{\mathbb{Z}}\mathrm{Spec}(\mathbb{Z}[t_1,\ldots,t_n])$, $X_0=X_{I,0}=Z_0\times_B \mathbb{A}^n_{B}$ and for every $i\in\{1,\ldots,n\}$, we put $X_i=X_{I,i}=Z_i\times_B \mathbb{A}^n_{B}=\mathbb{A}^n_{Z_i}$  and  $D_i=D_{I,i}=Z_0\times_B V(t_i)\subset Z_0\times_B \mathbb{A}^n_{B}$. 

For every subset $S\subset I$ of cardinal $r$, we denote by $H_S=H_{I,S}\cong \mathbb{A}^{n-r}_{B,t_{I\setminus S}}$ the closed subscheme $V((t_s)_{s\in S})$ of $\mathbb{A}^n_B$ and we put $\Sigma_S=\Sigma_{I,S}= Z_0\times_B H_S$.
\end{assum}

The collection  $\left({}^{D_i}_{X_i}\right)_{i\in I}$
is a deformation datum
on $X_0$ in the sense of Definition  \ref{defdatum}, which satisfies Assumption \ref{assumslid}.  Moreover,  since closed immersions between smooth schemes over a locally Noetherian base scheme are regular \cite[\href{https://stacks.math.columbia.edu/tag/067U}{Tag 067U}]{stacks-project}, this deformation datum actually satisfies all the regularity hypotheses in Theorem \ref{coroset}.

\defi \label{def:n-fold-def} With the notation and assumption above, the \emph  {deformation space} of the sequence of closed immersions $((Z_i)_{i\in I},Z_0): \quad Z_n\hookrightarrow Z_{n-1}\hookrightarrow \cdots \hookrightarrow Z_1 \hookrightarrow Z_0$
is the $X_0$-scheme

$$\sigma:\mathrm{Def}((Z_i)_{i\in I},Z_0):=\bbD \big (^{D_n, \ldots, D_1} _{X_n,  \ldots, X_1  ~ X_0 } \big)= \Bl \left\{ _{X_i}^{\sum _{j =i}^n D_j }\right\}_{n \geq i \geq 1} X_0 \longrightarrow X_0.$$
\xdefi

By definition of the multi-centered deformation space (see  Definition \ref{def:mc-def}),  $\mathrm{Def}((Z_i)_{i\in I},Z_0)$ is isomorphic as a scheme over $X_0$ to the relative spectrum of the quasi-coherent $\mathcal{O}_{X_0}$-subalgebra 
\begin{equation} \label{eq:multiple-def-algebra}
\mathcal{R}=\mathcal{O}_{X_0}\left[\frac{(\mathcal{I}_{Z_n})}{t_n},\ldots, \frac{(\mathcal{I}_{Z_k})}{t_k\cdots t_n},\ldots, \frac{(\mathcal{I}_{Z_1})} {t_1\cdots t_n}\right]    
\end{equation}
of
$\mathcal{O}_{Z_0}[t_1^{\pm 1},\ldots, t_n^{\pm 1}]$, see Subsection \ref{subsecdil}. Noting that $\mathcal{R}$ is a sub-$\mathcal{O}_{Z_0}$-algebra of $\mathcal{O}_{Z_0}[t_1^{\pm 1},\ldots, t_n^{\pm 1}]$ generated by homogeneous elements with respect to the natural $\mathbb{Z}^n$-grading on the variables $t=(t_1,\ldots, t_n)$, it follows that $\mathcal{R}$ is a $\mathbb{Z}^n$-graded sub-$\mathcal{O}_{Z_0}[t_1,\ldots,t_n]$-algebra $$\mathcal{R}=\sum_{\nu=(\nu_1,\ldots, \nu_n)} \mathcal{R}_{\nu}t^{\nu} \subset \mathcal{O}_{Z_0}[t_1^{\pm 1},\ldots, t_n^{\pm 1}],$$
where for every $\nu=(\nu_1,\ldots, \nu_n)\in \mathbb{Z}^n$,  $\mathcal{R}_\nu t^{\nu}=\mathcal{R}\cap \mathcal{O}_{Z_0}t^{\nu}$.     

\exam \label{ex:Def-Space-1}For $I=\{1\}$, the scheme $$\mathrm{Def}(Z_1,Z_0)=\mathrm{Bl}_{Z_1\times_B \mathbb{A}^1_{B,t_1}}^{Z_0\times_B V(t_1)}Z_0\times_B \mathbb{A}^1_{B,t_1}=\mathrm{Spec}_{Z_0\times_{B} \mathbb{A}^1_{B,t_1}} \big( \bigoplus_{i_1\in \mathbb{Z}} \mathcal{I}_{Z_1}^{i_1}t_1^{-i_1} \big)\to Z_0\times_B \mathbb{A}^1_{B,t_1},$$
where, by convention, $\mathcal{I}_{Z_1}^{i_1}=\mathcal{O}_{Z_0}$ for every $i_1\leq 0$, is the deformation to the normal cone \cite[$\S$2]{Ver76} \cite[(10.4)]{Ro96}  of the closed immersion $Z_1\subset Z_0$. 
\xexam

\rema 
In \cite[(10.5)]{Ro96}, for a pair of regular closed immersions $Z\to Y\to X=\mathrm{Spec}(A)$ with respective defining ideals $I=I_Y$ and $J=I_Z$ of $A$, the double deformation is defined as follows: $\bar{D}=\bar{D}(X,Y,Z)$ 
 is the spectrum of the subring 
\[
\mathcal O_{\bar D}
=
\sum_{n,m}
I^n J^{\,m-n}\cdot t^{-n} s^{-m}
\;\subset\;
A[t,s,t^{-1},s^{-1}],
\]
with the convention  \cite[Remark (10.2)]{Ro96}
that \(I^n=A\) for \(n\le0\) and \(J^k=A\) for \(k\le0\). 
 The above subring is in general not an $A[s,t]$-algebra (for instance, $t\not \in \mathcal{O}_{\bar{D}}$ unless $1\in J$) but it appears nevertheless to be treated in \cite{Ro96} as if it were.  With the re-interpretation that $\mathcal{O}_{\bar{D}}$ is the sub-$A[s,t]$-algebra of $A[s^{\pm 1}, t^{\pm 1}]$ generated by the terms in the above sum 
 rather than the sum itself, the definition enjoys the expected behavior and corresponds to the description  given in \cite[Proposition 9.1]{Ma24d}, see also Example \ref{ex:double-space} below. 
\xrema

We view $\mathrm{Def}((Z_i)_{i\in I},Z_0)$ as a scheme over $\mathbb{A}^n_{B}=\mathbb{A}^n_{B,t_I}$  via the composition of the dilatation morphism $\sigma:\mathrm{Def}((Z_i)_{i\in I},Z_0)\to X_0=Z_0\times_B \mathbb{A}^n_B$ with the projection $\mathrm{pr}_2:X_0\to \mathbb{A}^n_B$. 

\prop \label{smoothsmoo} With the assumption and notation above, 
$\mathrm{Def}((Z_i)_{i\in I},Z_0)$ is a smooth $\mathbb{A}^n_B$-scheme. Moreover, for every subset $S\subset I$,  the $H_S$-scheme $\mathrm{Def}((Z_i)_{i\in I},Z_0)\times_{\mathbb{A}^n_B} H_S \to H_S$
belongs to the category $\Sch_{H_S}^{\{V(t_i) \cap H_S\}_{i \in I \setminus S}\text{-reg}}$.
\xprop 
\pf 
Since $Z_0$ is a smooth $B$-scheme, $\mathrm{pr}_2:X_0=Z_0\times_B \mathbb{A}^n_B\to \mathbb{A}^n_B$ is a smooth morphism. Similarly, since $Z_1$ is a smooth $B$-scheme, $\mathrm{pr}_2:X_1\cap D_1=Z_1\times_B V(t_1)\to V(t_1)$ is a smooth morphism. Since by definition $\bbD\left( {}^{D_1}_{X_1}{~}_{X_0} \right)=\mathrm{Bl}_{X_1}^{D_1} X_0=\mathrm{Bl}_{X_1\cap D_1}^{D_1} X_0$, it follows from \cite[Proposition 2.16 (4)]{MRR20} that the composition $\mathrm{pr}_2\circ \sigma_1:\bbD\left( {}^{D_1}_{X_1}{~}_{X_0}\right)\to X_0=Z_0\times_B \mathbb{A}^n_B\to \mathbb{A}^n_B$ is a smooth morphism. The smoothness of 
$\mathrm{pr}_2\circ \sigma:\mathrm{Def}((Z_i)_{i\in I},Z_0)=\bbD \big (^{D_n, \ldots, D_1} _{X_n,  \ldots, X_1  ~ X_0 } \big) \to \mathbb{A}^n_B$ then follows from the isomorphism of $X_0$-schemes 
$$
\Theta :    \mathbb{D}\left(\begin{array}{cc}
D_{n}\\
X_{n}, & \mathbb{D}\left(\begin{array}{cc}
D_{n-1}\\
X_{n-1}, & \mathbb{D}\left(\begin{array}{cc}
D_{n-2}\\
X_{n-2}, & \mathbb{D}\cdots
\end{array}\right)
\end{array}\right)
\end{array}\right)
\to \bbD \big (^{D_n, \ldots, D_1} _{X_n,  \ldots, X_1  ~ X_0 } \big).
$$
of Example \ref{ex:iterated}, and repeated use of \cite[Proposition 2.16 (4)]{MRR20} as above.  For the second assertion, observe that since $\mathrm{pr}_2\circ \sigma:\mathrm{Def}((Z_i)_{i\in I},Z_0)\to \mathbb{A}^n_B$ is smooth, the induced morphism $\mathrm{Def}((Z_i)_{i\in I},Z_0)\times_{\mathbb{A}^n_B} H_S\to H_S$ is smooth  whence in particular flat. Since for every $i\in I\setminus S$,  $V(t_i)\cap H_S$ is a Cartier divisor on $H_S$, it follows  by \cite[\href{https://stacks.math.columbia.edu/tag/02OO}{Tag 02OO}]{stacks-project} that 
 $\mathrm{Def}((Z_i)_{i\in I},Z_0)\times_{\mathbb{A}^n_B} H_S\to H_S$ 
 belongs to the category $\Sch_{H_S}^{\{V(t_i)\cap H_S\}_{i \in I \setminus S}\text{-reg}}$.
\xpf

\subsection{Panelization of deformation spaces}

For every subset $J=\{j_1,\ldots, j_m\}$ of $I$ endowed with the induced ordering, we let $$((Z_j)_{j\in J},Z_0): \; Z_{j_m}\hookrightarrow Z_{j_{m-1}}\hookrightarrow \cdots \hookrightarrow Z_{j_1}\hookrightarrow Z_0$$ be the corresponding sequence of closed immersions
and for every $i\in I\setminus J$
we let 
 $$((Z_j|_{Z_i})_{j\in J}, Z_i): \quad Z_{j_m}\times_{Z_0} Z_i\hookrightarrow Z_{j_{m-1}}\times_{Z_0} Z_i\hookrightarrow \cdots \hookrightarrow Z_{j_1}\times_{Z_0} Z_i\hookrightarrow Z_0\times_{Z_0} Z_i=Z_i$$
be the induced sequence of closed immersions of subschemes of $Z_i$. By definition, the restricted deformation space $\mathrm{Def}((Z_j)_{j\in J},Z_0)$ is then a scheme over $Z_0\times_B \mathbb{A}^{m}_{B,t_J}$, which, by Proposition \ref{smoothsmoo}, is smooth over $\mathbb{A}^{m}_{B,t_J}$. Each deformation space 
 $$\mathrm{Def}((Z_j|_{Z_i})_{j\in J}, Z_i)\cong \mathrm{Def}((Z_j|_{Z_i})_{j\in J_{>i}}, Z_i)\times_B \mathbb{A}^r_{B,t_{J_{\leq i}}}, \quad  i\in I\setminus J,$$ where $r=\mathrm{card}\{J_{\leq i}\}$, is a scheme over $Z_i\times_B \mathbb{A}^m_{B,t_J}$ which, by Proposition \ref{smoothsmoo} again is smooth over $\mathbb{A}^m_{B,t_J}$. By Proposition \ref{factfks}, 
 each deformation space 
 $\mathrm{Def}((Z_j|_{Z_i})_{j\in J}, Z_i)$ is a closed subscheme of $\mathrm{Def}((Z_j)_{j\in J}, Z_0)$ and we have an induced sequence of nested closed immersions 
\begin{equation}\label{eq:def-immersions}
\mathrm{Def}((Z_j|_{Z_i})_{j\in J}, Z_i)\hookrightarrow \mathrm{Def}((Z_j|_{Z_{i'}})_{j\in J}, Z_{i'}) \quad \forall i\geq i'\in I\setminus J
\end{equation}
of schemes smooth over $\mathbb{A}^{m}_{B,t_J}$ for which the following diagram is commutative

$$
\xymatrix{
\mathrm{Def}((Z_j|_{Z_i})_{j\in J}, Z_i) \ar[d] \ar[r] &  \mathrm{Def}((Z_j|_{Z_{i'}})_{j\in J}, Z_{i'})  \ar[r] \ar[d] & \mathrm{Def}((Z_j)_{j\in J}, Z_0) \ar[d]
\\ Z_i\times_B \mathbb{A}^m_{B,t_J} \ar[r] & Z_{i'}\times_B \mathbb{A}^m_{B,t_J} \ar[r] & Z_0\times_B \mathbb{A}^m_{B,t_J}.
}
$$

\medskip 

\lemm \label{lem:def-space-translate} There exists a canonical isomorphism of $\bbD  \left(\left({}^{D_j}_{X_j}\right)_{j\in J} {~}_{X_0}\right)$-schemes 
\begin{equation}
\label{eq:def-panel}    
\bbD \left( \left( {}^{\sigma_J^{-1}D_i}_{ \bbD \left( \left( ^{D_j}_{X_j} \right)_{j \in J}  {}_{X_i} \right)}  \right)_{i \in I\setminus J}   {~}_{\bbD  \left(\left({}^{D_j}_{X_j}\right)_{j\in J} {~}_{X_0}\right)} \right)  \stackrel{\cong} {\to} \mathrm{Def}\left(
\mathrm{Def}((Z_j|_{Z_i})_{j\in J}, Z_i)_{i\in I\setminus J}, \mathrm{Def}((Z_j)_{j\in J}, Z_0)\right).
\end{equation}
\xlemm
\pf By definition, we have
$$\bbD  \left(\left({}^{D_j}_{X_j}\right)_{j\in J} {~}_{X_0}\right)=\bbD  \left(\left({}^{D_{I,j}}_{X_{I,j}}\right)_{j\in J} {~}_{X_{I,0}}\right)=\bbD \left(\left({}^{D_{J,j}\times_B \mathbb{A}^{n-m}_{B,t_{I\setminus J}}  }_{ X_{J,j}\times_B \mathbb{A}^{n-m}_{B,t_{I\setminus J}} }\right)_{j\in J} {~}_{X_{J,0}\times_B \mathbb{A}^{n-m}_{B,t_{I\setminus J}}}\right).$$
Since $\mathbb{A}^{n-m}_{B,t_{I\setminus J}}\to B$ is a flat morphism and multi-centered dilatations commute with flat base change \cite[Corollary 2.42, §3.7]{Ma24d}, it follows that 
$$
\bbD  \left(\left({}^{D_j}_{X_j}\right)_{j\in J} {~}_{X_0}\right) \cong 
\bbD  \left(\left({}^{D_{J,j}}_{X_{J,j}}\right)_{j\in J} {~}_{X_{J,0}}\right)\times_B \mathbb{A}^{n-m}_{B,t_{I\setminus J}}=\mathrm{Def}((Z_j)_{j\in J},Z_0)\times_{B} \mathbb{A}^{n-m}_{B,t_{I\setminus J}}$$
as schemes over $\mathbb{A}^n_{B,t_I}=\mathbb{A}^m_{B,t_J}\times_B \mathbb{A}^{n-m}_{B,t_{I\setminus J}}$. Under the isomorphism above, for every $i\in I\setminus J$, $\sigma_J^{-1}{D_i}$ is the divisor $\mathrm{Def}((Z_j)_{j\in J},Z_0)\times_{B} V(t_i)$ of $\mathrm{Def}((Z_j)_{j\in J},Z_0)\times_{B} \mathbb{A}^{n-m}_{B,t_{I\setminus J}}$. 
Furthermore, the same flat base change argument implies that for every $i\in I\setminus J$, $$\bbD  \left(\left({}^{D_j}_{X_j}\right)_{j\in J} {~}_{X_i}\right)\cong \mathrm{Def}((Z_j|_{Z_i})_{j\in J},Z_i)\times_{B} \mathbb{A}^{n-m}_{B,t_{I\setminus J}},$$ viewed as a closed subscheme of $\mathrm{Def}((Z_j)_{j\in J},Z_0)\times_{B} \mathbb{A}^{n-m}_{B,t_{I\setminus J}}$ by the base change of the closed immersion \eqref{eq:def-immersions}. Altogether, this identifies the left hand term in equation \eqref{eq:def-panel} with the scheme $\mathrm{Def}\left(
\mathrm{Def}((Z_j|_{Z_i})_{j\in J}, Z_i)_{i\in I\setminus J}, \mathrm{Def}((Z_j)_{j\in J}, Z_0)\right)$ and completes the proof.
\xpf

\theo \label{th:DefSpace-panels}With the notation above, for every subset $J\subset I$, there exists a unique isomorphism of schemes over $(Z_0\times_B \mathbb{A}^{m}_{B,t_J})\times_B \mathbb{A}^{n-m}_{B,t_{I\setminus J}} $
$$\Theta_J: \mathrm{Def}\left(
\mathrm{Def}((Z_j|_{Z_i})_{j\in J}, Z_i)_{i\in I\setminus J}, \mathrm{Def}((Z_j)_{j\in J}, Z_0)\right) \longrightarrow \mathrm{Def}((Z_i)_{i\in I},Z_0).$$
\xtheo
\pf 
As observed above, the deformation datum $\left({}^{D_i}_{X_i}\right)_{i\in I}$ on $X_0$
satisfies the hypothesis of Theorem  \ref{coroset}. It follows from Theorem \ref{theo-iso-panelization} 
that the panelization morphism 
\begin{equation*}
\Theta_J: 
\bbD \left( \left( {}^{\sigma_J^{-1}D_i}_{ \bbD \left( \left( ^{D_j}_{X_j} \right)_{j \in J}  {}_{X_i} \right)}  \right)_{i \in I\setminus J}   {~}_{\bbD  \left(\left({}^{D_j}_{X_j}\right)_{j\in J} {~}_{X_0}\right)} \right) \to \bbD \left( \left({}^{D_i}_{X_i} \right)_{i \in I} {~~}_{X_0} \right)
\end{equation*}
of Proposition  \ref{theopanelization-mor} 
is an isomorphism of $X_0$-schemes.
The left hand side of the above equation 
is isomorphic to $\mathrm{Def}\left(
\mathrm{Def}((Z_j|_{Z_i})_{j\in J}, Z_i)_{i\in I\setminus J}, \mathrm{Def}((Z_j)_{j\in J}, Z_0)\right)$ by Lemma \ref{lem:def-space-translate} whereas the right hand side equals  $\mathrm{Def}((Z_i)_{i\in I},Z_0)$ by definition. 
\xpf

\subsection{The polyptych of a multi-centered deformation space} 
Let $((Z_i)_{i\in I},Z_0)$ be as in Assumption \ref{VR-Def}, with $I=(\{1,\ldots , n\},\leq)$. Proposition  \ref{theopanelization-mor},  Remark \ref{reformulcoro} and Theorem \ref{th:DefSpace-panels} provide for every subset $J$ of $I$ a unique morphism 
$$\Theta_J: \bbD\bbD_J=\mathrm{Def}\left(
\mathrm{Def}((Z_j|_{Z_i})_{j\in J}, Z_i)_{i\in I\setminus J}, \mathrm{Def}((Z_j)_{j\in J}, Z_0)\right) \longrightarrow \mathrm{Def}((Z_i)_{i\in I},Z_0)=\bbD_I$$
of schemes over $Z_0\times_B \mathbb{A}^n_{B,t_I}$.
By \eqref{eq:def-immersions} and Proposition \ref{smoothsmoo}, the entries of $\bbD\bbD_J$ are deformation spaces of chains of closed immersions between schemes smooth over $\mathbb{A}^{m}_{B,t_J}$, so that Assumption \ref{VR-Def} holds again with $\mathbb{A}^{m}_{B,t_J}$ in place of $B$. By applying Theorem \ref{th:DefSpace-panels} iteratively to these entries, for all subsets of $J$ and $I\setminus J$, one gets a collection of multi-centered deformation spaces related to each other by a unique collection of isomorphisms of schemes over $Z_0\times_B \mathbb{A}^n_{B,t_I}$ induced by successive panelization isomorphisms, which take the form of a directed graph uniquely determined by the combinatorics of suitable partitions of the index set $(I,\leq)$. This graph is in particular independent of $Z_0$ and the sequence of closed immersions $Z_n\hookrightarrow Z_{n-1} \hookrightarrow \cdots \hookrightarrow Z_1\hookrightarrow Z_0$.  We call this directed graph the \emph{polyptych} of the multi-centered deformation space of length $n$ and denote it by $\mathscr{P}(n)$.

\exam The polyptych $\mathscr{P}(0)$ has a unique panel $Z_0$. The polyptych  $\mathscr{P}(1)$ has a unique panel as well, equal to the deformation space $\bbD \big( ^{D_1}_{X_1,X_0 } \big)=\mathrm{Def}(Z_1,Z_0)$.
\xexam

\exam{$($The double deformation space$)$.} \label{ex:double-space1} \label{ex:double-space}
For $I=(\{1,2\},\leq)$, the $\mathbb{A}^2_B$-scheme 
\begin{equation*}
\mathrm{Def}(Z_2,Z_1,Z_0) = \mathrm{Spec}_{{Z_0}}(\mathcal{O}_{Z_0}[t_1,t_2][\frac{\mathcal{I}_{Z_2}}{t_2},\frac{\mathcal{I}_{Z_1}}{t_1t_2}])  
\end{equation*}
is the double deformation space, cf.  \cite[(10.5) p. 376]{Ro96} and  \cite[Proposition 9.1]{Ma24d}. The polyptych $\mathscr{P}(2)$ is a triptych represented in Figure \ref{fig:figp1}, in which the two panelization isomorphisms express the double deformation space as an iterated deformation space.
\medskip

\begin{figure}[!ht]
 \centering
  \scalebox{0.70}
  {
\begin{tikzcd}[ampersand replacement=\&, row sep=0.6ex, column sep=5ex]
   \mathrm{Def}\left( \mathrm{Def}(Z_2,Z_2),\mathrm{Def}(Z_1,Z_0)\right)=\mathrm{Def}\left( Z_2\times_B\mathbb{A}^1_{B,t_1},\mathrm{Def}(Z_1,Z_0)\right)=\bbD \Big(^{D_2}_{X_2, \bbD \left( ^{D_1}_{X_1,X_0}\right)} \Big)  
   \ar[drrrr, controls={+(9,0) and +(-5,0)},  "\Theta_{\{1\}}"]  \\
   \&  \& \&  \& \bbD \left( ^{D_2, D_1} _{X_2,X_1,X_0}\right)=\mathrm{Def}(Z_2,Z_1,Z_0)  \\
  \mathrm{Def}\left ( \mathrm{Def}(Z_2,Z_1),\mathrm{Def}(Z_2,Z_0) \right )=\bbD \Big ( ^{D_1}_{\bbD \big(^{D_2}_{X_2 ,X_1} \big), \bbD \big(^{D_2}_{X_2 , X_0} \big)} \Big) \ar[urrrr, controls={+(+5,0) and +(-5,0)}, swap, "\Theta_{\{2\}}"]
\end{tikzcd}    
}
\caption{The polyptych $\mathscr{P}(2)$}
\label{fig:figp1}
\end{figure}
\xexam

\rema In general, the double deformation $\mathrm{Def}(Z_2,Z_1,Z_0)$ is not isomorphic to $\mathrm{Def}(\mathrm{Def}(Z_2,Z_0)|_{Z_1},\mathrm{Def}(Z_2,Z_0))$,  compare with \cite[3.2.19]{DJK21}. Indeed,  already in the affine case  $Z_0=\mathrm{Spec}(A)$,  the ring 
$A[t_1,t_2][\frac{I_{Z_2}}{t_2},\frac{I_{Z_1}}{t_1t_2}]$ is in general not isomorphic to the coordinate ring 
$(A[t_2][\frac{I_{Z_2}}{t_2}])[t_1][\frac{I_{Z_1}}{t_1}]$  of $\mathrm{Def}(\mathrm{Def}(Z_2,Z_0)|_{Z_1},\mathrm{Def}(Z_2,Z_0))$. 
\xrema

\medskip

\exam{$($The triple deformation space$)$.} \label{ex:Triple_space} \label{ex:3-space-polyptych} For $I=(\{1,2,3\},\leq)$, one gets an  $\mathbb{A}^3_B$-scheme $\mathrm{Def}(Z_3,Z_2,Z_1,Z_0)$ which we call the triple deformation space. Its polyptych $\mathscr{P} (3)$ has $19$ panels, represented in Figure \ref{fig:figp3} using the notation of subsection \ref{subsec:panels}. From the panel $p_1$,  Proposition  \ref{theopanelization-mor} and Remark \ref{reformulcoro} yield $6$ new panels $p_2$-$p_7$ corresponding to the non-trivial proper subsets of $\{1,2,3\}$. Each panel $p_i $ with $2 \leq i \leq 6$ can be further panelized in two different ways where the panel $p_7$ gives rise to $4$ new panels which can each be panelized in two different ways.  Panels on the left of Figure \ref{fig:figp3} can not be panelized further. These panels connect in several different ways to $p_1$ through panelization morphisms and the whole diagram is commutative.
\medskip

 \vspace{0.6cm}
 \begin{figure}[!ht]
 \centering
  \scalebox{0.85}{
\rotatebox{90}{
\begin{tikzcd}[ampersand replacement=\&, row sep=1.6ex, column sep=12ex]
p_{12}=\bbD \Big(^{D_3}_{X_3 , \bbD \big(^{D_2}_{X_2 , \bbD \big(^{D_1 } _{X_1, X_0}\big)} \big)}\Big) \ar[r, controls={+(+5,0) and +(-5,0)}] \ar[dr, controls={+(+5,0) and +(-5,0)}] \&  p_2=\bbD \Big( {}^{D_3}_{X_3 , } {}^{D_2}_{X_2,}{}_{ \bbD \big( {}^{D_1}_{X_1, X_0}\big) } \Big) \ar[rdddd, controls={+(+5,0) and +(-5,0)}]  \\
  p_{13}=\bbD \Big(^{D_2}_{\bbD\big( ^{D_3}_{X_3,X_2}\big), \bbD \big(^{D_3}_{ X_3, \bbD \big(^{D_1}_{X_1 , X_0} \big)} \big) } \Big) \ar[ru, controls={+(+5,0) and +(-5,0)}] \ar[rd, controls={+(+5,0) and +(-5,0)}] \&  p_3=\bbD \Big( {}^{D_3}_{X_3,\bbD \big({}^{D_2}_{X_2,}{}^{D_1}_{X_1, X_0} \big)} \Big) \ar[rddd, controls={+(+5,0) and +(-5,0)}]\\
   p_{14}=\bbD \Big(^{D_3}_{X_3,\bbD\big(^{D_1}_{\bbD\big(^{D_2}_{X_2, X_1}\big), \bbD\big(^{D_2}_{X_2,X_0} \big) } \big) } \Big) \ar[ru, controls={+(+5,0) and +(-5,0)}] \ar[rd, controls={+(+5,0) and +(-5,0)}]\& p_4=\bbD \Big({}^{D_2}_{\bbD\big(^{D_3}_{X_3,X_2}\big), \bbD \big(^{D_3,D_1}_{X_3,X_1,X_0}\big)}\Big)\ar[rdd, controls={+(+5,0) and +(-5,0)}]\\
  p_{15}=\bbD \Big({}^{D_2}_{{\bbD \big({}^{D_3}_{X_3,X_2} \big) },{\bbD \Big({}^{D_1}_{{{\bbD \big({}^{D_3}_{X_3,X_1} \big) }},{{\bbD \big({}^{D_3}_{X_3,X_0} \big) }}} \Big)}} \Big) \ar[ru, controls={+(+5,0) and +(-5,0)}] \ar[rd, controls={+(+5,0) and +(-5,0)}]  \&  p_5=\bbD \Big({}^{D_3, D_1}_{X_3, \bbD\big(^{D_2}_{X_2,X_1}\big),  \bbD\big(^{D_2}_{X_2,X_0}\big)} \Big) \ar[rd, , controls={+(+5,0) and +(-5,0)}] \\
   p_{16}=\bbD \Big( {}^{D_1}_{\bbD \Big({}^{D_2}_{\bbD\big({}^{D_3}_{X_3,X_2} \big ), \bbD\big({}^{D_3}_{X_3,X_1} \big )} \Big) , \bbD \Big({}^{D_2}_{\bbD\big({}^{D_3}_{X_3,X_2} \big ), \bbD\big({}^{D_3}_{X_3,X_0} \big )} \Big) } \Big) \ar[r, controls={+(+5,0) and +(-5,0)}] \ar[rdd, controls={+(+5,0) and +(-5,0)}]\ar[rddd,  controls={+(+5,0) and +(-5,0)}] \& p_6= \bbD \Big( {}^{D_2}_{\bbD \big( ^{D_3}_{X_3,X_2}\big) , } {}^{D_1}_{\bbD \big( {}^{D _3}_{X_3, X_1}\big) , \bbD \big( {}^{D_3}_{X_3, X_0}\big) } \Big) \ar[r,, controls={+(+5,0) and +(-5,0)}]\& p_1=\bbD \left(^{D_3,D_2,D_1}_{X_3,X_2,X_1,X_0} \right) \\
  p_{17}=\bbD \Big(^{D_1}_{{\bbD \big( ^{D_3}_{X_3, \bbD\big(^{D_2}_{X_2, X_1}\big)}\big),{\bbD \big(^{D_2}_{\bbD\big(^{D_3}_{X_3,X_2}\big) , \bbD\big(^{D_3}_{X_3,X_0}\big)} \big) }}} \Big) \ar[rd, controls={+(+5,0) and +(-5,0)}] \ar[r, controls={+(+5,0) and +(-5,0)}]  \& p_8=\bbD \Big( {}^{D_1}_{\bbD \big( ^{D_3}_{X_3, \bbD\big(^{D_2}_{X_2, X_1}\big)}\big),\bbD \big({}^{D_3}_{X_3,}{}^{D_2}_{X_2, X_0} \big)} \Big) \ar[rd, controls={+(+5,0) and +(-5,0)}]\\
   p_{18}=\bbD \Big(  _{\bbD \big( _{X_3, \bbD\big(_{X_2, X_1}^{D_2}\big) }^{D_3} \big) , {\bbD \big( _{X_3, \bbD\big(_{X_2, X_0}^{D_2}\big) }^{D_3}\big)}}^{D_1}  \Big) \ar[ru, controls={+(+5,0) and +(-5,0)}] \ar[ruuu, controls={+(+5,0) and +(-5,0)}]\ar[rdd, controls={+(+5,0) and +(-5,0)}] \& p_9=\bbD \Big(^{D_1}_{\bbD\big(^{D_3,D_2}_{X_3 ,X_2,X_1}\big), \bbD \big(^{D_2}_{\bbD\big(_{X_3,X_2}^{D_3}\big),\bbD \big( ^{D_3}_{X_3,X_0}\big)}\big)}\Big) \ar[r, controls={+(+5,0) and +(-5,0)}]\&  p_{7}=\bbD \Big( {}^{D_1}_{\bbD \big({}^{D_3}_{X_3,}{}^{D_2}_{X_2, X_1} \big),\bbD \big({}^{D_3}_{X_3,}{}^{D_2}_{X_2, X_0} \big)} \Big) \ar[uu, controls={+(+0,+2) and +(-5,0)}]\\
    p_{19}=\bbD \Big(^{D_1}_{\bbD \big(^{D_2} _{\bbD \big(^{D_3}_{X_3,X_2}\big), \bbD \big(^{D_3}_{X_3,X_1}\big) } \big) , \bbD\big(^{D_3}_{X_3, \bbD \big(^{D_2}_{X_2 , X_0}\big)} \big)} \Big) \ar[rd, controls={+(+5,0) and +(-5,0)}] \ar[r, controls={+(+5,0) and +(-5,0)}] \& p_{10}=\bbD \Big( {}^{D_1}_{\bbD \big({}^{D_2}_{\bbD \big(^{D_3}_{X_3,X_2}\big),\bbD \big(^{D_3}_{X_3,X_1}\big) } \big),\bbD \big({}^{D_3}_{X_3,}{}^{D_2}_{X_2, X_0} \big)} \Big)\ar[ru, controls={+(+5,0) and +(-5,0)}] \\
     \& p_{11}=\bbD \Big( {}^{D_1}_{\bbD \big({}^{D_3}_{X_3,}{}^{D_2}_{X_2, X_1} \big), \bbD\big(^{D_3}_{X_3, \bbD \big(^{D_2}_{X_2 , X_0}\big)} \big)} \Big)\ar[ruu, controls={+(+5,0) and +(-5,0)}]\\
  \end{tikzcd} 
  }
  }
\caption{The polyptych $\mathscr{P}(3)$}
\label{fig:figp3}
\end{figure} 

For instance, the panel $p_7$ in Figure \ref{fig:figp3} provides an isomorphism 
{\small 
$$\Theta_{\{2,3\}}:\mathrm{Def}(\mathrm{Def}(Z_{3},Z_{2},Z_{1}),\mathrm{Def}(Z_{3},Z_{2},Z_{0}))=\bbD \Big( {}^{D_1}_{\bbD \big({}^{D_3}_{X_3,}{}^{D_2}_{X_2, X_1} \big),\bbD \big({}^{D_3}_{X_3,}{}^{D_2}_{X_2, X_0} \big)} \Big) \to \bbD \left(^{D_3,D_2,D_1}_{X_3,X_2,X_1,X_0} \right) =\mathrm{Def}(Z_3,Z_2,Z_1,Z_0).$$
}
In a similar vein, the following three panelization isomorphisms, corresponding to the panels $p_2$, $p_5$ and $p_6$, express the triple  deformation space as an iterated double deformation space:  

{\small 
\begin{tikzcd}[ampersand replacement=\&, row sep=1.6ex, column sep=6ex]
\mathrm{Def}(Z_{3}\times_{B}\mathbb{A}_{B,t_{1}}^{1},Z_{2}\times_{B}\mathbb{A}_{B,t_{1}}^{1},\mathrm{Def}(Z_{1},Z_{0}))=\bbD \Big( {}^{D_3}_{X_3 , } {}^{D_2}_{X_2,}{}_{ \bbD \big( {}^{D_1}_{X_1, X_0}\big) } \Big)  \ar[drrrr, controls={+(+6,0) and +(-3,0)}, pos=0.1, "\Theta_{\{1\}}"] \\
\mathrm{Def}(Z_{3}\times_{B}\mathbb{A}_{B,t_{2}}^{1},\mathrm{Def}(Z_{2},Z_{1}),\mathrm{Def}(Z_{2},Z_{0}))=\bbD \Big({}^{D_3, D_1}_{X_3, \bbD\big(^{D_2}_{X_2,X_1}\big),  \bbD\big(^{D_2}_{X_2,X_0}\big)} \Big) \ar[rrrr, controls={+(+6,0) and +(-2,0)}, pos=0.17, "\Theta_{\{2\}}"]
 \& \& \& \& \mathrm{Def}(Z_3,Z_2,Z_1,Z_0).  \\
\mathrm{Def}(\mathrm{Def}(Z_{3},Z_{2}),\mathrm{Def}(Z_{3},Z_{1}),\mathrm{Def}(Z_{3},Z_{0}))=\bbD \Big( {}^{D_2}_{\bbD \big( ^{D_3}_{X_3,X_2}\big) , } {}^{D_1}_{\bbD \big( {}^{D _3}_{X_3, X_1}\big) , \bbD \big( {}^{D_3}_{X_3, X_0}\big) } \Big)
\ar[urrrr, controls={+(+6,0) and +(-3,0)}, swap, pos=0.1, "\Theta_{\{3\}}"]
\end{tikzcd}  
}
\xexam

\subsection{Strata of deformation spaces} \label{subsec:Strata-RV-Def}
For a subset $S$ of $I$ with $m$ elements, the $S$-stratum 
$\mathrm{Def}((Z_i)_{i\in I},Z_0)\times_{X_0} \Sigma_S$
of $\mathrm{Def}((Z_i)_{i\in I},Z_0)$ as introduced in section \ref{secstrata} equals the restriction of the structure morphism $\mathrm{Def}((Z_i)_{i\in I},Z_0)\to \mathbb{A}^n_B$ over the linear subspace $H_S=V((t_s)_{s\in S})\cong \mathbb{A}^{n-m}_{B,t_{I\setminus S}}$  of $\mathbb{A}^n_B$.

\theo \label{th:strata-def} With the notation above, the following hold:
\begin{enumerate}[label=(\roman*)]
  \item The restriction of  $\mathrm{Def}((Z_i)_{i\in I},Z_0)\to \mathbb{A}^n_{B}$ over $\mathbb{A}^n_{B}\setminus\bigcup_{i\in I} H_i\cong \mathbb{G}_{m,B}^n$ is canonically isomorphic to $Z_0\times_B\mathbb{G}_{m,B}^n$.
  \item The exceptional stratum $\mathrm{Def}((Z_i)_{i\in I},Z_0)\times_{X_0} \Sigma_I\cong  \mathrm{Def}((Z_i)_{i\in I},Z_0)\times_{\mathbb{A}^n_B} H_I$ is canonically isomorphic as a scheme over $Z_n$ to the total space of the vector bundle $$\mathbb{V}(\bigoplus_{i=1}^{n} \mathcal{C}_{Z_i/Z_{i-1}}|_{Z_n})=\bigoplus_{i=1}^n N_{Z_i/Z_{i-1}}|_{Z_n} \to Z_n,$$ where $\mathcal{C}_{Z_i/Z_{i-1}}$ denotes the conormal sheaf of the regular closed immersion $Z_i\subset Z_{i-1}$.
  \item More generally, for every nonempty subset $S=(\{s_1,\ldots, s_r\},\leq)$ of $I$, the $S$-stratum 
  $$\mathrm{Def}((Z_i)_{i\in I},Z_0)\times_{X_0} \Sigma_S\cong \mathrm{Def}((Z_i)_{i\in I},Z_0)\times_ {\mathbb{A}^n_B} H_S$$
  is canonically isomorphic to the total space of a vector bundle over  $\mathrm{Def}((Z_i|_{Z_{s_r}})_{i\in I\setminus S}, Z_{s_r})$. 
\end{enumerate}
\xtheo
\pf The first assertion follows from Corollary \ref{cor:isocomplement}. 

We now proceed to show (ii), by induction on the number $n$ of elements of $I$. For $n=1$, the isomorphism 
$\mathrm{Def}(Z_1,Z_0)\times_{\mathbb{A}^1_{B,t_1}} V(t_1)\cong N_{Z_1/Z_0}$ of schemes over $Z_1$ is a central property of the usual affine deformation space \cite[$\S$2]{Ver76} and \cite[(10.4)]{Ro96}, see also \cite[Proposition 2.9 (3)]{MRR20}. We now assume that the assertion holds for every index set with strictly less than $n$ elements and let $S=\{1,\ldots, n-1\}\subset I$.  Under our regularity assumptions, it follows from Theorem \ref{theo-iso-panelization} that the panelization morphism
\begin{equation}\label{eq:def-panel-strata}
\Theta_S: 
\bbD \left( {}^{\sigma_S^{-1}D_n}_{X_n ~~~ \bbD  \left(\left({}^{D_s}_{X_s}\right)_{s\in S} {~}_{X_0}\right)} \right) \to \bbD \left( \left({}^{D_i}_{X_i} \right)_{i \in I} {~~}_{X_0} \right)=\mathrm{Def}((Z_i)_{i\in I},Z_0)
\end{equation}
of Proposition  \ref{theopanelization-mor} 
is an isomorphism of $X_0$-schemes. Lemma \ref{lem:def-space-translate}  identifies  the left hand side of equation \eqref{eq:def-panel-strata} with $\mathrm{Def}(Z_n\times_B \mathbb{A}^{n-1}_{B,t_S},\mathrm{Def}((Z_s)_{s\in S},Z_0))$. By Proposition \ref{smoothsmoo},  $\mathrm{Def}((Z_s)_{s\in S},Z_0)$ is a smooth $\mathbb{A}^{n-1}_{B,t_S}$-scheme and since $Z_n\times_B \mathbb{A}^{n-1}_{B,t_S}$ is smooth over $\mathbb{A}^{n-1}_{B,t_S}$ as well, the same proposition implies that $\mathrm{Def}(Z_n\times_B \mathbb{A}^{n-1}_{B,t_S},\mathrm{Def}((Z_s)_{s\in S},Z_0))$ is a smooth $\mathbb{A}^n_{B,t_I}$-scheme, which therefore belongs to the category $\mathrm{Sch}_{H_S}^{\{V(t_n)\cap H_S\}\text{-reg}}$. This implies that the left hand side of equation \eqref{eq:def-panel-strata} belongs to the category  $\Sch_{\Sigma_S}^{\{D_n\cap \Sigma_S\}\text{-reg}}$. This allows us to apply 
Proposition \ref{higher_strata} to conclude that $\Theta_S$ induces an isomorphism between 
$$\bbD \left( {}^{D_n\times_{X_0}  \Sigma_S}_{X_n\times_{X_0} \Sigma_S ~~~ \bbD  \left(\left({}^{D_s}_{X_s}\right)_{s\in S} {~}_{X_0}\right)\times_{X_0} \Sigma_S} \right)\cong
\mathrm{Def}(Z_n, \mathrm{Def}((Z_s)_{s\in S},Z_0)\times_{\mathbb{A}^{n-1}_{B,t_S}} H_S)$$ and $\mathrm{Def}((Z_i)_{i\in I},Z_0)\times_{\mathbb{A}^n_{B,t_I}} H_S$ as schemes over $\mathbb{A}^1_{B,t_n}$. 
By induction hypothesis, the $Z_{n-1}$-scheme  
$\mathrm{Def}((Z_s)_{s\in S},Z_0)\times_{\mathbb{A}^{n-1}_{B,t_S}} H_S$ is isomorphic to the total space of the vector bundle $$\mathrm{p}_{n-1}:V_{n-1}:=\bigoplus_{i=1}^{n-1} N_{Z_i/Z_{i-1}}|_{Z_{n-1}}\to Z_{n-1}.$$ 
Moreover, under this isomorphism, the closed immersion of $Z_n$ into $V_{n-1}=\mathrm{Def}((Z_s)_{s\in S},Z_0)\times_{\mathbb{A}^{n-1}_{B,t_S}} H_S$ 
is the composition of the inclusion $Z_n\to Z_{n-1}$ with the zero section of $\mathrm{p}_{n-1}:V_{n-1}\to Z_{n-1}$. Appealing again to the case $n=1$ combined with Proposition \ref{prop:dila-zero-section} (ii), gives the conclusion that 
$$\mathrm{Def}((Z_i)_{i\in I},Z_0)\times_{\mathbb{A}^n_{B,t_I}} H_I \cong \mathrm{Def}(Z_n, \mathrm{Def}((Z_s)_{s\in S},Z_0)\times_{\mathbb{A}^{n-1}_{B,t_S}} H_S)\times_{\mathbb{A}^1_{B,t_n}} H_n$$ 
is a vector bundle over $Z_n$ isomorphic to $$N_{Z_n/V_{n-1}}\cong N_{Z_n/Z_{n-1}}\oplus V_{n-1}|_{Z_n}=\bigoplus_{i=1}^n N_{Z_i/Z_{i-1}}|_{Z_n}.$$

 To prove (iii),  we appeal to 
 Proposition \ref{th:vb-strata} combined with Lemma \ref{lem:def-space-translate} and Theorem \ref{th:DefSpace-panels} to conclude that the panelization isomorphism $\Theta_{I\setminus S}$ induces an isomorphism 
 \begin{equation}\label{eq:def-strata-lower-codim}
  \mathrm{Def}((Z_i)_{i\in I},Z_0)\times_{\mathbb{A}^n_B} H_S \cong \mathrm{Def}(\mathrm{Def}((Z_i|_{Z_s})_{i\in I\setminus S}, Z_s)_{s \in S}, \mathrm{Def}((Z_i)_{i\in I\setminus S}, Z_0)))\times_{\mathbb{A}^n_B} H_S
 \end{equation} 
 of schemes over   
 $\mathrm{Def}((Z_i|_{Z_{s_r}})_{i\in I\setminus S},Z_{s_r})$. The conclusion then follows from (ii) which asserts that the right hand side of equation \eqref{eq:def-strata-lower-codim} is canonically a vector bundle over $\mathrm{Def}((Z_i|_{Z_{s_r}})_{i\in I\setminus S},Z_{s_r})$.
 \xpf

\exam If $n\geq 2$, then reformulating Example \ref{cod1strata} gives that for every $s\in I$, 
the $s$-stratum $\mathrm{Def}((Z_i)_{i\in I},Z_0)\times_{\mathbb{A}^n_B} V(t_s)$ is a vector bundle over $\mathrm{Def}(Z_n,\ldots, Z_{s+1},Z_s)\times_B \mathbb{A}^{s-1}_{B, (t_{1},\ldots t_{s-1})}$ isomorphic to  
\begin{equation} \label{eq:codim1-strat_Def}
\mathrm{Def}(Z_n,\ldots, Z_{s+1}, N_{Z_s/Z_{s-1}},\ldots N_{Z_s/Z_1}, N_{Z_s/Z_0}),
\end{equation}
where 
$N_{Z_s/Z_{i}}\to N_{Z_s/Z_{i-1}}$ is the natural closed immersion and where $Z_{s+1}\to N_{Z_s/Z_{s-1}}$ is the composition of the closed immersion $Z_{s+1}\to Z_s$ with the zero section $Z_s\to N_{Z_s/Z_{s-1}}$ of the normal bundle $N_{Z_s/Z_{s-1}}\to Z_s$. 
\xexam

\exam{$($Strata of the double deformation space$)$.} \label{ex:Rost_double}The following table summarizes the different descriptions of the strata of the double deformation space computed using the panelization isomorphisms of Example \ref{ex:double-space1} combined with Proposition  \ref{th:vb-strata} and Proposition \ref{higher_strata}. 

\begin{center}
{\small 
\begin{tabular}{|l|c|c|c|}
\hline 
 & Panelization $\Theta_{\{1\}}$  & Panelization $\Theta_{\{2\}}$ & Rost's description \tabularnewline
\hline 
\hline 
$-\times_{\mathbb{A}^2_B} V(t_1)$  & $\mathrm{Def}(Z_2,N_{Z_1/Z_0})$ & $N_{\mathrm{Def}(Z_2,Z_1)/\mathrm{Def}(Z_2,Z_0)}$ & $\mathrm{Def}(N_{Z_1/Z_0}|_{Z_2},N_{Z_1/Z_0})$\tabularnewline
\hline 
$-\times_{\mathbb{A}^2_B} V(t_2)$
& $N_{Z_2\times\mathbb{A}^{1}_{B,t_1}/\mathrm{Def}(Z_1,Z_0)}$ & $\mathrm{Def}(N_{Z_2/Z_1},N_{Z_2/Z_0})$ & $\mathrm{Def}(N_{Z_2/Z_1},N_{Z_2/Z_0})$\tabularnewline
\hline 
$-\times_{\mathbb{A}^2_B} V(t_1,t_2)$
 & $N_{Z_2/N_{Z_1/Z_0}}$ & $N_{N_{Z_2/Z_1}/N_{Z_2/Z_0}}$ & $N_{(N_{Z_1/Z_0}|_{Z_2})/N_{Z_1/Z_0}}=N_{N_{Z_2/Z_1}/N_{Z_2/Z_0}}$\tabularnewline
\hline 
\end{tabular}
}
\par\end{center}

\noindent In the first line of the table, the schemes $\mathrm{Def}(Z_2,N_{Z_1/Z_0})$ and 
$\mathrm{Def}(N_{Z_1/Z_0}|_{Z_2},N_{Z_1/Z_0})$ are both isomorphic as schemes over  $\mathrm{Def}(Z_2,Z_1)$ to the pullback  of the vector bundle $$N_{Z_1/Z_0}\times_B \mathbb{A}^1_B\to Z_1\times_B \mathbb{A}^1_B$$
by the dilatation morphism  $\mathrm{Def}(Z_2,Z_1)\to Z_1\times_B \mathbb{A}^1_B$, see Proposition \ref{prop:dila-zero-section} in Appendix B.

In the last line of the table, the schemes $N_{Z_2/N_{Z_1/Z_0}}$ and $N_{N_{Z_2/Z_1}/N_{Z_2/Z_0}}$ are both isomorphic as schemes over $Z_2$ to the vector bundle $N_{Z_2/Z_1}\times_{Z_2} (N_{Z_1/Z_0}|_{Z_2})\to Z_2$. 
\xexam

\rema 
In \cite[$\S$ 3.5 p. 231]{Hoy17}, Hoyois considers a notion of deformation space $D_{W\times \mathbb{A}^1}(D_Z X)$ for a smooth closed triple $(X,Z,W)$, i.e. a triple of nested closed immersions $W\subset Z\subset X$ of schemes smooth over a base scheme $B$. When translated in our notation $Z_2\subset Z_1\subset Z_0$, it corresponds to the deformation space $\mathrm{Def}(Z_2\times_B \mathbb{A}^1_B, \mathrm{Def}(Z_1,Z_0))$, that is, the first panel $\Theta_1:\mathrm{Def}(Z_2\times_B \mathbb{A}^1_B, \mathrm{Def}(Z_1,Z_0)) \to \mathrm{Def}(Z_2,Z_1,Z_0)$ of the double deformation space in Example \ref{ex:double-space}. The diagram (3.24) in {\it loc. cit.} provides a description of the strata of this deformation space $D_{W\times \mathbb{A}^1}(D_Z X)$ which coincides with that in the table in Example \ref{ex:Rost_double} deduced through the panelization isomorphism $\Theta_1$. Later on in Remark 3.28 p. 233, Hoyois briefly mentions that some additional compatibility relations between certain purity isomorphisms for smooth closed tuple $(X_0, \ldots ,  X_n)$ can be established by considering an "$n$-fold deformation to the normal bundle". It seems likely that our proposed Definition \ref{def:n-fold-def} of $n$-fold deformation space $\mathrm{Def}(Z_n,\ldots, Z_0)$ for a chain of closed immersions $Z_n\subset \cdots \subset Z_0$ between smooth schemes together with the panelization techniques provide the appropriate framework in which to express and work out these compatibility relations in more detail.  
\xrema 

\exam{$($Strata of the triple deformation space$)$.} The following table collects a selection of expressions for the strata of the triple deformation space $\mathrm{Def}(Z_3,Z_2,Z_1,Z_0)$ of Example \ref{ex:Triple_space}. The third column of the table indicates the panelization isomorphisms of the polyptych $\mathscr{P}(3)$ in Figure \ref{fig:figp3} from which these expressions are deduced.

\begin{center}
{\small 
\begin{tabular}{|l|c|l|}
\hline 
$\times_{\mathbb{A}_{B}^{3}}V(t_{1})$ & $\mathrm{Def}(Z_{3},Z_{2},N_{Z_{1}/Z_{0}})$
& $(p_{2})$\tabularnewline
\hline 
$\times_{\mathbb{A}_{B}^{3}}V(t_{2})$ & $\mathrm{Def}(Z_{3},N_{Z_{2}/Z_{1}},N_{Z_{2}/Z_{0}})$ & $(p_{5})$\tabularnewline
\hline 
$\times_{\mathbb{A}_{B}^{3}}V(t_{3})$ & $\mathrm{Def}(N_{Z_{3}/Z_{2}},N_{Z_{3}/Z_{1}},N_{Z_{3}/Z_{0}})$ & $(p_{6})$\tabularnewline
\hline 
$\times_{\mathbb{A}_{B}^{3}}V(t_{1},t_{2})$ & $\mathrm{Def}(Z_{3},N_{Z_{2}/N_{Z_{1}/Z_{0}}})\cong\mathrm{Def}(N_{Z_{2}/N_{Z_{1}/Z_{0}}}|_{Z_{3}},N_{Z_{2}/N_{Z_{1}/Z_{0}}})$ & $(p_{12})$ and Proposition \ref{prop:dila-zero-section}  \tabularnewline
\hline 
$\times_{\mathbb{A}_{B}^{3}}V(t_{1},t_{3})$ &$\mathrm{Def}(N_{Z_{3}/Z_{2}},N_{Z_{3}/N_{Z_{1}/Z_{0}}})$ & $(p_{13})$\tabularnewline
\hline 
$\times_{\mathbb{A}_{B}^{3}}V(t_{2},t_{3})$ & $\mathrm{Def}(N_{Z_{3}/N_{Z_{2}/Z_{1}}},N_{Z_{3}/N_{Z_{2}/Z_{0}}})$ & $(p_{18})$\tabularnewline
\hline 
$\times_{\mathbb{A}_{B}^{3}}V(t_{1},t_{2},t_{3})$ & $N_{Z_{3}/Z_{2}}\times_{Z_{3}}(N_{Z_{2}/Z_{1}}|_{Z_{3}})\times_{Z_{3}}(N_{Z_{1}/Z_{0}})|_{Z_{3}}$ & Theorem \ref{th:strata-def} (ii) \tabularnewline
\hline 
\end{tabular}
}
\end{center}
\xexam

\appendix
\pagestyle{fancy}
 \fancyhf{}
\fancyfoot[C]{\thepage}
 \fancyhead[CE]{{\small APPENDIX}} 
\fancyhead[CO]{{\small SYMMETRIC DEFORMATION SPACES}}
 \renewcommand{\headrulewidth}{0pt}

\section{Symmetric multi-centered deformation spaces}\label{appendix:vbsym}
In this appendix we consider from the viewpoint of multi-centered dilatations another related notion of multi-centered deformation space, introduced in \cite[10.6]{Ro96}, \cite[3.1.3]{Ivorra14} and \cite[§2]{Le26}. 

\subsection{ Definition and universal property} \label{ssapsym}
\defi \label{defdatumsym} A \emph{symmetric deformation datum} on a scheme $X$ indexed by a finite set $I$ is a pair $\left({}^{D_I}_{Y_I}\right)=\left({}^{D_i}_{Y_i}\right)_{i\in I}$
consisting of a collection of closed subschemes $Y_I=\{Y_i\}_{i\in I}$ of $X$ and a collection $D_I=\{D_i\}_{i\in I}$ of locally principal closed subschemes of $X$. 
\xdefi

\defi 
\label{def:mc-desymf}The \emph{multi-centered symmetric deformation space} of a symmetric deformation datum $\left({}^{D_I}_{Y_I}\right)$ on a scheme $X$ is the $X$-scheme
\begin{equation} 
\label{sigma_Isym} 
\sigma_I^{sym}:\bbD^{sym} \left(\left({}^{D_I}_{Y_I}\right) {}_{X} \right)=\bbD^{sym} \left(\left({}^{D_i}_{Y_i}\right)_{i\in I} {}_{X} \right) := \Bl \left\{{}^{ D_i}_{Y_i}  \right\}_{i \in I} X\to X,
\end{equation}
that is, the dilatation of $X$ with multi-center $\{[Y_i, D_i ]\}_{i \in I}$.
\xdefi

 \prop \label{Propunivdefsym}\label{coronumberofmorsym}
 The multi-centered symmetric deformation space $\sigma^{sym}_I:\bbD^{sym} \left(\left({}^{D_I}_{Y_I}\right) {}_{X} \right)\to X$ represents the contravariant functor 
 $ \Sch _{X}^{\{D_i\}_{i\in I} \text{-reg}} \to Sets$ defined by
\begin{equation}\label{blow.up.iso.eqsym}
(f: T\to X) \;\longmapsto\; \begin{cases}\left\{*\right\}, \; \text{if $f^{-1} \left(D_i \right) \subset f^{-1} \left( Y_i \right) $ in $Clo \left( T \right) $ for every $i \in I  $};\\ \varnothing,\;\text{else.}\end{cases}
\end{equation}
 In particular, if $T \to X $ is an object in the category $\Sch _{X}^{\{D_i\}_{i\in I} \text{-reg}}$, we have \[\#\Hom _X \left( T, \bbD^{sym} \left(\left({}^{D_I}_{Y_I}\right) {}_{X} \right)\right) \in \left\{0,1\right\}.\]
 \xprop
 \pf
 This is a consequence of Definition \ref{def:mc-desymf} and \cite[Proposition 3.17]{Ma24d}.
 \xpf 

\lemm \label{lem:baseclosym}
Let $\left({}^{D_I}_{Y_I}\right)=\left({}^{D_i}_{Y_i}\right)_{i\in I}$ be a symmetric deformation datum on a scheme $X$ and let $X'\subset X$ be a closed subscheme of $X$. Assume that $\bbD^{sym} \left(\left({}^{D_I}_{Y_I}\right) {}_{X} \right)\times_X X' \to X'$ 
belongs to the category $\Sch_{X'}^{\{X' \cap D_i\}_{i \in I}\text{-reg}}$. Then there exists a unique isomorphism of  $X' $-schemes 
$$\bbD ^{sym}\left( \left({}^{X' \cap D_I}_{X' \cap Y_I}\right){~}_{X'}\right) \to \bbD^{sym} \left(\left({}^{D_I}_{Y_I}\right) {}_{X} \right)\times_X X'.$$
\xlemm 
\pf
The existence follows from \cite[Lemma 3.36]{Ma24d} and the uniqueness from the universal property of dilatations.
\xpf

\subsection{Panelization of symmetric deformation spaces } \label{ssapjhsym}
For a symmetric deformation datum 
$\left({}^{D_I}_{Y_I}\right)$ on a scheme $X$ and a subset $J$ of $I$, Proposition \ref{Propunivdefsym}  implies the existence of a unique  $X$-morphism $\delta_{I,J} : \bbD _I ^{sym} \to \bbD _J ^{sym}$. On the other hand, the base change  $ \left({}^{D_i \times _X \bbD^{sym}_J}_{Y_i \times_X {\bbD ^{sym}_J}}\right)_{i \in I \setminus J}$
is a symmetric deformation datum on the scheme $\bbD_J ^{sym}$. 

\defi \label{def:panel-datumsym} 
For every subset $J\subset I$, the multi-centered symmetric deformation space
\begin{equation} \label{inducedsym}
\xi_{J}:\bbD\bbD_{J}^{sym}:=
 \bbD ^{sym}\left( \left({}^{D_i \times _X \bbD^{sym}_J}_{Y_i \times_X {\bbD ^{sym}_J}}\right)_{i \in I \setminus J}{~}_{\bbD_J ^{sym}}\right)
 \to \bbD_J ^{sym}.
 \end{equation}
is called the \emph{$J$-th panel} of the multi-centered symmetric deformation space $ \bbD_I^{sym}$.
\xdefi

The next proposition asserts the existence of a panelization morphism $\Phi_J: \bbD\bbD_J^{sym}\to \bbD_I ^{sym}$ which, in contrast with the case of multi-centered deformation spaces (see Proposition \ref{theopanelization-mor} and Theorem  \ref{theo-iso-panelization}) is always an isomorphism.  

\prop \label{theopanelization-morsym} Let $\left({}^{D_I}_{Y_I}\right)$ be a symmetric  deformation datum on a scheme $X$ and let $J$ be a subset of $I$. Then there exists a unique $X$-morphism  
 $\Phi_J: \bbD\bbD_J^{sym}\to \bbD_I ^{sym}$.  Moreover, $\Phi_J$ is an isomorphism and $\xi _J = \delta_{I,J} \circ \Phi_J$. 
\xprop

\pf The existence of $\Phi_J$ and the fact that it is an isomorphism follows from \cite[Proposition 3.28 and its proof]{Ma24d}. The second assertion follows from Proposition \ref{coronumberofmorsym}.
\xpf 

\coro \label{strassym} 
Let $S\subset J $ be a subset and let $\Sigma_S=\bigcap_{s\in S} D_s$. Assume that the scheme $\bbD^{sym}_I \times _{X} \Sigma_S  \to  \Sigma_S$ belongs to the category $\Sch_{\Sigma_S}^{\{D_i\cap \Sigma_S\}_{i\in I \setminus J}\text{-reg}}$. Then there exists a unique isomorphism of  $\bbD_J^{sym} \times_X \Sigma_S$-schemes 
\begin{equation} \label{hfrakssym}
\epsilon_{J,S} : \bbD_I^{sym} \times_X \Sigma_S \isomto \bbD ^{sym}\left( \left({}^{D_i \times _X \bbD^{sym}_J \times _{X} \Sigma_S}_{Y_i \times_X {\bbD ^{sym}_J} \times _{X} \Sigma_S}\right)_{i \in I \setminus J}{~}_{\bbD_J ^{sym}\times _{X} \Sigma_S}\right).
   \end{equation}
\xcoro
\pf
The existence follows from \cite[Lemma 3.36]{Ma24d} and the uniqueness from the universal property of dilatations.
\xpf 

\subsection{Symmetric deformation space of a collection of transverse closed subschemes}\label{sec:Verdier-Rostsym}

In this section, we establish basic properties of symmetric deformation spaces specialized to the situation of the deformation space of a collection  of closed immersions  between schemes smooth over a locally Noetherian base scheme which satisfy suitable transversality conditions. 

\begin{assum} \label{VR-Defsym}
We consider a collection of closed immersions $$((V_i)_{i\in I},V_0): \qquad V_i \hookrightarrow V_0, \qquad I=\{1,\ldots, n\}$$ 
of schemes smooth over a locally Noetherian base scheme $B$. We denote by $\mathcal{I}_{V_i}\subset \mathcal{O}_{V_0}$ the ideal sheaf of $V_i$ and by $\mathcal{C}_{V_i/V_0}$ its conormal sheaf.   

We let $\mathbb{A}^n_{B}=\mathbb{A}^{n}_{B,t_I}=B\times_{\mathbb{Z}}\mathrm{Spec}(\mathbb{Z}[t_1,\ldots,t_n])$, $X_0=X_{I,0}=V_0\times_B \mathbb{A}^n_B$ and for every $i\in\{1,\ldots,n\}$, we put $Y_i=Y_{I,i}=V_i\times_B \mathbb{A}^n_{B}=\mathbb{A}^n_{V_i}$  and  $D_i=D_{I,i}=V_0\times_B V(t_i)\subset X_0$. For every subset $S\subset I$ of cardinal $r$, we denote by $H_S=H_{I,S}\cong \mathbb{A}^{n-r}_{B,t_{I\setminus S}}$ the closed subscheme $V((t_s)_{s\in S})$ of $\mathbb{A}^n_B$ and we put $\Sigma_S=\Sigma_{I,S}= V_0\times_B H_S$.
\end{assum}

\defi With the notation and assumption above, the \emph{symmetric deformation space} of the collection of closed immersions $((V_i)_{i\in I},V_0)$ is the $X_0$-scheme  $$\Def^{sym}((V_i)_{i\in I},V_0) := \bbD^{sym} \left(\left( {}^{D_i }_{Y_i}\right) _{i\in I } {}_{X_0}  \right)=\mathrm{Bl}\left\{{}^{D_i}_{Y_i}\right\}_{i\in \{1,\ldots,n\}}X_0\to X_0.$$
\xdefi 

By definition of the multi-centered symmetric deformation space as a multi-centered dilatation, $\Def^{sym}((V_i)_{i\in I},V_0)$ is isomorphic as a scheme over $X_0$ 
to the relative spectrum over $X_0$ of the $\mathbb{Z}^n$-graded quasi-coherent $\mathcal{O}_{X_0}$-subalgebra 
\begin{equation} \label{eq:multiple-def-algebrasym}
\mathcal{O}_{X_0}\left[\frac{(\mathcal{I}_{V_n})}{t_n},\ldots, \frac{(\mathcal{I}_{V_k})}{t_k},\ldots, \frac{(\mathcal{I}_{V_1})} {t_1}\right] \cong 
\bigoplus_{(\nu_1,\ldots,\nu_n)\in \mathbb{Z}^n} \mathcal{I}_{V_1}^{\nu_1}\cdots \mathcal{I}_{V_k}^{\nu_k}\cdots  \mathcal{I}_{V_n}^{\nu_n}t_1^{-\nu_1}\cdots t_k^{-\nu_k }\cdots t_n^{-\nu_n}
\end{equation}
of $\mathcal{O}_{V_0}[t_1^{\pm 1},\ldots, t_n^{\pm 1}]$, where, by convention,  $\mathcal{I}_{V_j}^{\nu_j}=\mathcal{O}_{V_0}$ for every $\nu_j<0$.

\medskip 

\exam \label{ex:DefSym=Def}For $I=\{1\}$, the scheme $$\Def^{sym}(V_1,V_0)=\mathrm{Bl}_{V_1\times_B \mathbb{A}^1_{B,t_1}}^{V_0\times_B V(t_1)} (V_0\times_B \mathbb{A}^1_{B,t_1})=\mathrm{Spec}_{V_0\times_B \mathbb {A}^1_{B,t_1}}(\bigoplus_{n\in \mathbb{Z}} \mathcal{I}_{V_1}^n t_1^{-n})$$ is the deformation space $\mathrm{Def}(V_1,V_0)$ of the closed immersion $V_1\subset V_0$, see Example \ref{ex:Def-Space-1}.
\xexam

We view $\Def^{sym}((V_i)_{i\in I},V_0)$ as a scheme over $\mathbb{A}^n_B=\mathbb{A}^n_{B,t_I}$ via the composition of the dilatation morphism $\sigma:\Def^{sym}((V_i)_{i\in I},V_0)\to X_0=V_0\times_B \mathbb{A}^n_B$ with the second projection.
For every subset $J$ of $I$ of cardinal $m$, the collection $\left({}^{D_j}_{Y_j}\right)_{j\in J}$ is a symmetric deformation datum on $X_0$. Since multi-centered dilatations commute with flat base change \cite[Corollary 2.42, §3.7]{Ma24d}, it follows that there is a canonical isomorphism of $X_0$-schemes
\begin{equation*}
\label{eq:flat-bc-sym}
    \begin{array}{rcl} 
\bbD^{sym}\left(\left({}^{D_{I,j}}_{Y_{I,j}}\right)_{j\in J} ~ {}_{X_{I,0}}\right) & = &    \bbD^{sym}\left(\left({}^{D_{J,j}\times_B \mathbb{A}^{n-m}_{B,t_{I\setminus J}}}_{Y_{J,j}\times_B \mathbb{A}^{n-m}_{B,t_{I\setminus J}}} \right)_{j\in J} ~ {}_{X_{J,0}\times_B \mathbb{A}^{n-m}_{B,t_{I\setminus J}}}\right) \\ & \cong & 
\bbD^{sym}\left(\left({}^{D_{J,j}}_{Y_{J,j}} \right)_{j\in J} ~ {}_{X_{J,0}}\right) \times_B \mathbb{A}^{n-m}_{B,t_{I\setminus J}} \\
& \cong &
\Def^{sym}((V_j)_{j\in J},V_0)\times_{B} \mathbb{A}^{n-m}_{B,t_{I\setminus J}}. 
\end{array}
\end{equation*}

We henceforth assume further that the collection of closed subschemes $(V_i)_{i\in I}$ satisfies the following condition: 

\begin{assum} \label{asum:transverse}
For all $(m_1,\ldots,m_n)\in \mathbb{Z}^n$, the canonical homomorphism of $\mathcal{O}_{V_0}$-modules 
\begin{equation}
\label{eq:strong-transversality}
\mathcal{I}_{V_1}^{m_1}\otimes_{\mathcal{O}_{V_0}} \cdots \otimes_{\mathcal{O}_{V_0}} \mathcal{I}_{V_n}^{m_n}\to \mathcal{I}_{V_1}^{m_1}\cdots \mathcal{I}_{V_n}^{m_n}
\end{equation}
is an isomorphism, where, by convention, $\mathcal{I}_{V_i}^{m_i}=\mathcal{O}_{V_0}$ for $m_i\leq 0$.
\end{assum}

The following proposition  shows that under the  assumption above, the double deformation space considered by Levine \cite[§2]{Le26} via fiber products of deformation spaces  coincides with the symmetric deformation space of Rost \cite[(10.6)]{Ro96} and  Ivorra \cite[3.1.3]{Ivorra14}.
\prop \label{tensorsym}
With the notation and assumption above, for  every partition $I= I_1 \sqcup \ldots \sqcup I_k$ of $I$, there exist unique isomorphisms of $X_0$-schemes
$$\Def^{sym}((V_i)_{i\in I},V_0) \cong \Def^{sym}((V_i)_{i\in I_1},V_0) \times _{V_0} \ldots \times _{V_0}  \Def^{sym}((V_i)_{i\in I_k},V_0).$$
\xprop 
\pf 
The uniqueness is guaranteed by
Proposition \ref{Propunivdefsym}. To prove the existence part, it suffices to consider the case of the partition $I=\{1\}\sqcup \cdots \sqcup \{n\}$. Putting $\mathcal{A}_i=\bigoplus_{\nu_i \in \mathbb{Z}} \mathcal{I}_{V_i}^{\nu_i} t_i^{-\nu_i}$, $i\in I$, and
$\mathcal{A}=\bigoplus_{\nu=(\nu_1,\ldots, \nu_n)\in \mathbb{Z}^n} \mathcal{I}_{V_1}^{\nu_1}\cdots  \mathcal{I}_{V_n}^{\nu_n} t_1^{-\nu_1}\cdots t_n^{-\nu_n}$ and using equation \eqref{eq:multiple-def-algebrasym}, it suffices to prove that the natural homomorphism of $\mathbb{Z}^n$-graded  ${\calO _{X_0}}$-algebras $ \mathcal{A}_1\otimes_{\mathcal{O}_{V_0}} \cdots \otimes_{\mathcal{O}_{V_0}} \mathcal{A}_n \to \mathcal{A}$ is an isomorphism. 
This holds exactly when the canonical homomorphism of  equation \eqref{eq:strong-transversality} 
is an isomorphism for every $\nu=(\nu_1,\cdots, \nu_n)\in \mathbb{Z}^n$, which is precisely guaranteed by Assumption \ref{asum:transverse}. 
\xpf 

\exam Let $((V_i)_{i\in I},V_0)$ be a collection of closed immersions as in Assumption \ref{VR-Defsym} which are globally transversal, in the sense that for every subset $J\subset I$, the closed subscheme $\bigcap_{j\in J} V_j$ of $V_0$ is smooth over $B$ with conormal sheaf equal to the direct sum of the restrictions of the conormal sheaves of the $V_j$, $j\in J$. Then the collection $((V_i)_{i\in I},V_0)$ satisfies Assumption  \ref{asum:transverse} and hence satisfies the conclusion of Proposition \ref{tensorsym}. Indeed, the transversality condition implies that at every point $v\in \bigcap_{j\in J} V_j$, the ideals $\mathcal{I}_{V_j,v}\subset \mathcal{O}_{V_0,v}$ are generated by disjoint regular subsequences $s_j$, $j\in J$, of a regular sequence in the maximal ideal of $\mathcal{O}_{V_0,v}$, a property which implies, by Proposition \ref{AppPro:Tensor-equal-product}, that the homomorphism in equation  \eqref{eq:strong-transversality} is an isomorphism. 

Note however that Assumption \ref{asum:transverse} is strictly weaker than transversality: for instance, for a pair $V_1$, $V_2$ of closed subschemes of $V_0$ smooth over $B$, it follows from the proof of Proposition  \ref{AppPro:Tensor-equal-product}, using the same d\'evissage in both arguments, that Assumption \ref{asum:transverse} is equivalent to the vanishing of $\mathrm{Tor}^{\mathcal{O}_{V_0}}_2(\mathcal{O}_{V_0}/\mathcal{I}_{V_1}, \mathcal{O}_{V_0}/\mathcal{I}_{V_2})$. So for instance, Assumption \ref{asum:transverse} is satisfied for $V_1=V_2=\{0\}_B$ in $\mathbb{A}^1_B$ or for the pair of non-transverse coordinate lines $V_1=\{x_1=x_2=0\}$ and $V_2=\{x_2=x_3=0\}$ in $\mathbb{A}^3_B$.
\xexam

Finally, we consider for a subset $S$ of $I$ with $m$ elements  the structure of the stratum  $$\Def^{sym}((V_i)_{i\in I},V_0)\times_{X_0} \Sigma_S\cong \Def^{sym}((V_i)_{i\in I},V_0)\times_{\mathbb{A}^n_B} H_S$$
over the linear subspace $H_S=V((t_s)_{s\in S})\cong \mathbb{A}^{n-m}_{B,t_{I\setminus S}}$ of $\mathbb{A}^n_B$. 
By \cite[Proposition 3.16]{Ma24d}, the morphism  $\Def^{sym}((V_i)_{i\in I},V_0)\times_{X_0} \Sigma_S \to \Sigma_S=V_0\times_B H_S$ factors through $\bigcap_{s\in S} Y_s\cap D_s$ so that  we have a canonical isomorphism
$$\Def^{sym}((V_i)_{i\in I},V_0)\times_{X_0} \Sigma_S \cong \Def^{sym}((V_i)_{i\in I},V_0)\times_{X_0} ((\bigcap_{s\in S} V_s) \times_B H_S)$$   
of schemes over $Q_S=Q_{I,S}:=(\bigcap_{s\in S} V_s )\times_B H_S$.

\theo \label{th:strata-sym} 
 With the assumption and notation above, the following hold: 
\begin{enumerate}[label=(\roman*)]
  \item The restriction of  $\Def^{sym}((V_i)_{i\in I},V_0)\to \mathbb{A}^n_{B}$ over $\mathbb{A}^n_{B}\setminus\bigcup_{i\in I} H_i\cong \mathbb{G}_{m,B}^n$ is canonically isomorphic to $V_0\times_B\mathbb{G}_{m,B}^n$.
  \item The stratum $\Def^{sym}((V_i)_{i\in I},V_0)\times_{X_0} \Sigma_I\cong  \Def^{sym}((V_i)_{i\in I},V_0)\times_{\mathbb{A}^n_{B,t_I}} H_I$ is canonically isomorphic as a scheme over $Q_I\cong \bigcap_{i\in I} V_i$ to the total space of the vector bundle $$\mathbb{V}(\bigoplus_{i=1}^{n} \mathcal{C}_{V_i/V_{0}}|_{Q_I})=\bigoplus_{i=1}^n N_{V_i/V_{0}}|_{Q_I} \to Q_I.$$ 
  \item More generally, for every nonempty subset $S=(\{s_1,\ldots, s_m\},\leq)$ of $I$, the $S$-stratum 
  $$\Def^{sym}((V_i)_{i\in I},V_0)\times_{X_0} \Sigma_S\cong \Def^{sym}((V_i)_{i\in I},V_0)\times_{\mathbb{A}^n_B} H_S$$
  is canonically isomorphic as a scheme over $(\bigcap_{s\in S} V_s)\times_{V_0}\Def^{sym}((V_i)_{i\in I\setminus S},V_0)$ to the total space of the vector bundle 
 $$ (\bigoplus_{s\in S} N_{V_s/V_0}|_{\bigcap_{s\in S} V_s})\times_{V_0}\Def^{sym}((V_i)_{i\in I\setminus S},V_0)\to (\bigcap_{s\in S} V_s)\times_{V_0}\Def^{sym}((V_i)_{i\in I\setminus S},V_0).$$
\end{enumerate}
\xtheo
 \pf The first assertion follows from the definition of $\Def^{sym}((V_i)_{i\in I},V_0)$ as a multi-centered dilatation. For assertion (ii),  combining Example \ref{ex:DefSym=Def} and Proposition  \ref{tensorsym} with \cite[Proposition 2.9]{MRR20} gives
 \begin{equation*}
 \begin{array}{rcl}
 \Def^{sym}((V_i)_{i\in I},V_0)\times_{\mathbb{A}^n_{B,t_I}} H_I &\cong & (\Def(V_n,V_0)\times_{V_0}\cdots \times_{V_0} \Def(V_1,V_0))\times_{\mathbb{A}^n_{B,t_I}} H_I \\
 & \cong & (\Def(V_n,V_0)\times_{\mathbb{A}^1_{B,t_n}} V(t_n))\times_{V_0} \cdots \times_{V_0} (\Def(V_1,V_0)\times_{\mathbb{A}^1_{B,t_1}} V(t_1)) \\ 
 & \cong & N_{V_n/V_0}\times_{V_0}\cdots \times_{V_0} N_{V_1/V_0}  \cong  \bigoplus_{i=1}^n N_{V_i/V_0}|_{Q_I}.
 \end{array} 
 \end{equation*}
For assertion (iii), by applying Proposition \ref{tensorsym} to the partition $I=S\sqcup I\setminus S$ of $I$, we obtain the isomorphism 
\begin{equation*}
 \begin{array}{rcl}    
\Def^{sym}((V_i)_{i\in I},V_0)\times_{\mathbb{A}^n_B} H_{I,S} & \cong & (\Def^{sym}((V_i)_{i\in I\setminus S},V_0)\times_{V_0} \Def^{sym}((V_s)_{s\in S},V_0))\times_{\mathbb{A}^n_B} H_{I,S} \\
& \cong & \Def^{sym}((V_i)_{i\in I\setminus S},V_0)\times_{V_0} (\Def^{sym}((V_s)_{s\in S},V_0))\times_{\mathbb{A}^m_{B,t_S}} H_{S,S}, 
\end{array}
\end{equation*}
 where $\Def^{sym}((V_i)_{i\in I\setminus S},V_0)$ is viewed as a scheme over $V_0\times_B \mathbb{A}^{n-m}_{B,t_{I\setminus S}}$ and $\Def^{sym}((V_s)_{s\in S},V_0)$ as a scheme over $V_0\times_B\mathbb{A}^m_{B,t_S}$. By assertion (ii), $\Def^{sym}((V_s)_{s\in S},V_0)\times_{\mathbb{A}^m_{B,t_S}} H_{S,S}$ 
is isomorphic to the vector bundle $\bigoplus_{s\in S} N_{V_s/V_0}|_{Q_{S,S}}$ over $Q_{S,S}=\bigcap_{s\in S} V_s$ and the conclusion follows.  
\xpf

\pagestyle{fancy}
 \fancyhf{}
\fancyfoot[C]{\thepage}
 \fancyhead[CE]{{\small APPENDIX}} 
\fancyhead[CO]{{\small ELEMENTARY TRANSFORMATIONS OF VECTOR BUNDLES}}
 \renewcommand{\headrulewidth}{0pt}

\section{On some mono-centered dilatations of vector bundles}\label{appendix:vb}
Our aim is to describe two natural types of mono-centered dilatations of vector bundles. 
 \subsection{Elementary dilatations of vector bundles}
  We consider the following setup:  $X$ is a locally Noetherian scheme, 
  $$\pi_{E}:E=\mathbb{V}(\mathcal{E})=\mathrm{Spec}_X(\mathrm{Sym}^{\bullet} \mathcal{E}) \to X$$ is a vector bundle associated to a coherent locally free $\mathcal{O}_{X}$-module $\mathcal{E}$, $i:D\hookrightarrow X$ is an effective Cartier divisor on $X$ and  $\pi_{F}:F=\mathbb{V}(\mathcal{F})\to D$ is a vector bundle on $D$, associated to a coherent locally free $\mathcal{O}_{D}$-module $\mathcal{F}$, which is a sub-vector bundle of the restriction $E|_{D}=E\times_{X}D\cong\mathbb{V}(i^{*}\mathcal{E})\to D$ of $E$ to $D$. Then $[F,E|_D]$ is a mono-center in $E$ and we let $\sigma :\Bl_{F}^{E|_{D}}E\to E$ be the corresponding dilatation morphism.
  
  The following proposition is a reinterpretation  in the language of dilatation of the notion and construction of elementary transformations between vector bundles first introduced and developed by Maruyama, see \cite[$\S$ 1]{Mar73} and \cite[$\S$1 p. 242]{Mar82}, for their applications to the study of vector bundles on higher dimensional projective varieties:
 
  \prop \label{prop:Maruyama-dilatations}
 The composition $\pi_E\circ \sigma:\Bl_{F}^{E|_{D}}E\to X$ is a vector bundle on $X$. If moreover $F=E'|_D$ for some sub-vector bundle $\pi_{E'}:E'\to X$ of $\pi_E:E\to X$ then there exists a unique morphism of $E$-scheme $E'\to \Bl_{F}^{E|_{D}}E$ and the latter is a closed immersion of $E'$ as the total space of  a sub-vector bundle of $\pi_E\circ \sigma:\Bl_{F}^{E|_{D}}E\to X$.
 \xprop
 
 \pf
Let  $\mathcal{I}\subset \mathcal{O}_X$ be the invertible ideal sheaf of $D$, let $p:i^*\mathcal{E}\to \mathcal{F}$ be the surjective homomorphism of $\mathcal{O}_D$-modules corresponding to the closed immersion of $F$ as a sub-vector bundle of $E|_D$ and let  $q:\mathcal{E}\to i_*\mathcal{F}$ be the surjective homomorphism of $\mathcal{O}_X$-modules adjoint to $p$. 

Letting $\mathcal{F}'$ be the kernel of $p$, we claim that the kernel $\mathcal{G}$ of $q$ is a locally free $\mathcal{O}_X$-module which fits into the following commutative diagram 
{\small 
\begin{equation*}
\begin{tikzcd}
   & & 0 \ar[d] & 0 \ar[d]   \\ 
   0 \ar[r]  & \mathcal{E}\otimes\mathcal{I} \ar[d, equal]\ar[r] & \mathcal{G} \ar[r] \ar[d] &  i_{*}\mathcal{F}' \ar[r] \ar[d] & 0 \\
   0 \ar[r]  & \mathcal{E}\otimes\mathcal{I} \ar[r] & \mathcal{E} \ar[r] \ar[d, swap, "q"] &  i_{*}i^*\mathcal{E} \ar[d] \ar[r] & 0 \\
   & & i_*\mathcal{F} \ar[r,equal] \ar[d] & i_*\mathcal{F} \ar[d]\\
   & & 0  & 0 
\end{tikzcd}
\end{equation*}
}
whose lines and columns are exact sequences of $\mathcal{O}_{X}$-modules. The existence and properties of the above diagram are immediate from the definition of $\mathcal{G}$. The fact that $\mathcal{G}$ is locally free can be verified on the stalks at closed points $x$ of $X$ as follows. First observe that if  $x$ is not a point of $D$, then $(i_*\mathcal{F})_x$ is the zero sheaf so that  $\mathcal{G}_x=\mathcal{E}_x$ which is by assumption a free $\mathcal{O}_{X,x}$-module. Now assume that $x$ is a point of $D$, let $A=\mathcal{O}_{X,x}$ and  let $t\in A$ be a generator of the ideal $\mathcal{I}_x\subset A$. Without loss of generality, we can assume that $\mathcal{E}_x= A^{\oplus n}$ for some $n\geq 0$ and that $(i_*\mathcal{F})_{x}=(A/tA)^{\oplus m}$ for some $m\leq n$. Then there exists an automorphism $\varphi$ of the $A/tA$-module $(A/tA)^{\oplus n}$ such that the surjection $p_x:(i^*\mathcal{E})_x\cong (A/tA)^{\oplus n}\to \mathcal{F}_x\cong (A/tA)^{\oplus m}$ is the composition of $\varphi$ with the projection onto the first $m$ factors. Let $\phi$ be any lift of $\varphi$ to an $A$-module endomorphism of $A^{\oplus n}$. Then $\phi$ is an automorphism since its determinant is an element of $A$ invertible modulo $t$, whence invertible in $A$ because $t$ belongs to the maximal ideal of $A$ and it follows that $q_x:\mathcal{E}_x\to (i_*\mathcal{F})_x$ is the composition of $\phi:A^{\oplus n}\to A^{\oplus n}$, the projection $A^{\oplus n}\to A^{\oplus m}$ onto the first $m$ factors and the direct sum of $m$ copies of the canonical homomorphism $A\to A/tA$. This implies in turn that $\mathcal{G}_x$ is isomorphic to the free $A$-module $(tA)^{\oplus m}\oplus A^{\oplus n-m}$.  

Now let $\tilde{\mathcal{E}}=\mathcal{G}\otimes \mathcal{I}^{\vee}$. Then $\tilde{\mathcal{E}}$ is a locally free $\mathcal{O}_X$-module and the first line of the diagram above provides an injective homomorphism $\mathcal{E}\cong (\mathcal{E}\otimes \mathcal{I})\otimes \mathcal{I}^{\vee} \to \tilde{\mathcal{E}}$. The latter determines in turn an affine morphism of $X$-schemes $\tau:\tilde{E}:=\mathbb{V}(\tilde{\mathcal{E}})\to E$ which is an isomorphism over $X\setminus D$ and whose restriction over $D$ is the contraction of $\tilde{E}|_{D}$ onto the sub-vector bundle $F$ of $E|_{D}$. Since $\tilde{E}$ obviously belongs to the category $\Sch_E^{E|_D\text{-reg}}$ and since, by construction, we have $\tau^{-1}(E|_D)=\tau^{-1}(F)$ it follows from Theorem-Definition 
\ref{Def-Multi-centered} that $\tau$ factors through a unique morphism of $X$-schemes $\epsilon:\tilde{E}\to\mathrm{Bl}_{F}^{E|_{D}}E$. To complete the proof, it remains to verify that $\epsilon$ is an isomorphism. Since $\tilde{E}$ and $\mathrm{Bl}_{F}^{E|_{D}}E$ are affine over $X$, we can reduce to showing that for every  point $x\in X$ the base change of $\epsilon$ by the canonical morphism $\mathrm{Spec}(\mathcal{O}_{X,x})\to X$ is an isomorphism. With the notation of the previous paragraph, we can thus assume that $\mathcal{E}=A^{\oplus n}$, $\mathcal{F}=(A/tA)^{\oplus m}$,   $\mathcal{G}= (tA)^{\oplus m} \oplus A^{\oplus n-m}$ and hence, by construction, that $\tilde{\mathcal{E}}\cong A^{\oplus m} \oplus (t^{-1}A)^{\oplus n-m}$. The choice of a basis  $x_1,\ldots, x_n$ of $A^{\oplus n}$ determines an isomorphism of $A$-algebras  $\mathrm{Sym}^{\bullet}\mathcal{E}\cong A[x_1,\ldots, x_n]$ for which the ideal of $F$ as a subscheme of $E$ is generated by the elements $t,x_{m+1},\ldots, x_n$. Now one gets from the construction in the affine case reviewed just after Theorem-Definition 
\ref{Def-Multi-centered} that the dilatation $\Bl_F^{E|_D}E$ of $E$ is isomorphic to the spectrum of the $A$-algebra $A[x_1,\ldots,x_m, t^{-1}x_{m+1}, \ldots, t^{-1}x_n] \cong \mathrm{Sym}^{\bullet}\tilde{\mathcal{E}}$.

Now assume that  $F=E'|_D$ for some sub-vector bundle $j:E'\hookrightarrow E$.  Then  $j^{-1}(E|_D)=E'|_D=j^{-1}(E'|_D)=F$ which is a Cartier divisor on $E'$. Thus, $E'$ belongs to the category $\Sch_E^{{E|_D}\text{-reg}}$ and the universal property of $\sigma:\mathrm{Bl}_{F}^{E|_D} E\to E$ in Theorem-Definition \ref{Def-Multi-centered} implies that $j$ factors through a unique morphism of $E$-schemes $\tilde{j}:E'\to \mathrm{Bl}_{F}^{E|_D} E$. It is then straightforward to check on an open cover of $X$ by affine open subsets over which $\pi_{E}:E\to X$ and $\pi_{E'}:E'\to X$ are trivial vector bundles and $D$ is a principal Cartier divisor that $\tilde{j}:E'\to \mathrm{Bl}_{F}^{E|_D} E$ is a closed immersion of $E'$ as the total space of a sub-vector bundle of $\pi_E\circ \sigma:\Bl_{F}^{E|_{D}}E\to X$.
 \xpf
 
\exam\label{ex:dilatation_zero_section}
With the notation of Proposition \ref{prop:Maruyama-dilatations}, assume that $F$ is the restriction to $D$ of the image of the zero section $s_{0,E}:X\to E$ of $\pi_E:E\to X$, equivalently, $\mathcal{F}$ is the zero sheaf on $D$. Then the vector bundle $\pi_E\circ \sigma:\mathrm{Bl}_{F}^{E|_D} E=\mathrm{Bl}_{s_{0,E}(X)}^{E|_D} E\to X$ is isomorphic to
\begin{equation}
\pi_{E[D]}:E[D]:=\mathbb{V}(\mathcal{I}_D^\vee \otimes \mathcal{E})\to X.
\end{equation}
The dilatation morphism $\sigma:E[D]=\mathbb{V}(\mathcal{I}_D^\vee \otimes \mathcal{E})\to \mathbb{V}(\mathcal{E})=E$ is induced by the canonical homomorphism of $\mathcal{O}_X$-modules $\iota\otimes \mathrm{id}_{\mathcal{E}}: \mathcal{E}=\mathcal{O}_X\otimes \mathcal{E}\to \mathcal{I}_D^{\vee} \otimes \mathcal{E}$, where $\iota:\mathcal{O}_X\to \mathcal{I}^\vee_D$ is the homomorphism dual to the inclusion $\mathcal{I}_D\subset \mathcal{O}_X$. It contracts the restriction $E[D]|_D$ of $E[D]$ over $D$ to $s_{0,E}(D)\subset E|_D$ and restricts to an isomorphism over $X\setminus D$. 
\xexam

\subsection{Exceptional vector bundles} 
We consider the following setup over a locally Noetherian base scheme $S$: $j:Y\to X$ is a closed immersion of smooth $S$-schemes and   $i:D\to X$ is an effective Cartier divisor on $X$ with ideal sheaf $\mathcal{I}_D\subset \mathcal{O}_X$ such that $Y$ and $D$ intersect transversely in $X$, in the sense that for $Z=Y\times_X D$, the induced closed immersions $\iota_Y:Z\to Y$ and $\iota_D:Z\to D$ are both regular. We let $\mathcal{E}$ be a coherent locally free $\mathcal{O}_X$-module on $X$, we denote by $\pi_E:E=\mathbb{V}(\mathcal{E})\to X$ and $\pi_{E[D]}:E[D]=\mathbb{V}(\mathcal{I}_D^{\vee}\otimes \mathcal{E})\to X$ the associated vector bundles and by $\sigma:E[D]\to E$ the dilatation $X$-morphism as in Example \ref{ex:dilatation_zero_section}. We let  $s_{0,E}:X\to E$ be the zero section of $\pi_E$ and  $\pi_{E|_Y}:E|_Y=\mathbb{V}(j^*\mathcal{E})\to Y$ and  $\pi_{E|_D}:E|_D=\mathbb{V}(i^*\mathcal{E})\to D$ be the restrictions of $E$ to $Y$ and $D$ respectively.  We let  $\hat{j}=s_{0,E}\circ j:Y\to E$ so that we have a commutative diagram 
\[\xymatrix@=13ex{ E|_Y \ar@<1ex>[d]^{\pi_{E|_Y}} \ar[r]^{j_E} & E \ar@<1ex>[d]^{\pi} \\ Y \ar@<1ex>[u]^{s_{0,E|_Y}} \ar[r]^{j} \ar[ur]^{\hat{j}} & X. \ar@<1ex>[u]^{s_{0,E}} } \] 
This datum determines a mono-center $[Y,D]$ on $X$ and a pair of mono-centers $[\hat{j}(Y),E|_D]$ and  $[j_E(E|_Y),E|_D]$ on $E$. 
We let 
$$ \alpha:\mathrm{Bl}_Y^D X\to X, \quad 
\tau_0:\Bl_{\hat{j}(Y)}^{E|_D} E\to E \quad \textrm{and} \quad \tau:\Bl_{j_E(E|_Y)}^{E|_D} E\to E$$ be the corresponding dilatation morphisms and consider the last two schemes  
 as $X$-schemes via the  compositions $\pi_E\circ \tau_0$ and $\pi_E\circ \tau$ respectively. 
 
\prop \label{prop:dila-zero-section} With the notation above the following hold:
\begin{enumerate}[label=(\roman*)]
    \item There exists a commutative diagram 
\[\xymatrix{\mathrm{Bl}_{\hat{j}(Y)}^{E|_D} E \ar[d] \ar[r] \ar[dr]^{\tau_0} & \mathrm{Bl}_{E|_Y}^{E|_D} E \ar[d]^{\tau} \ar[r] & \mathrm{Bl}_{Y}^D X \ar[d]^{\alpha} \\
E[D] \ar[r]^{\sigma} & E \ar[r]^{\pi_{E}} & X}
\]    
in which the three rectangles are cartesian.    

In particular, the $\mathrm{Bl}_Y^D X$-schemes $\mathrm{Bl}_{E|_Y}^{E|_D} E$ and  $\mathrm{Bl}_{\hat{j}(Y)}^{E|_D} E$ are respectively isomorphic to the vector bundles 
 \begin{equation}
 \alpha^*E=\mathbb{V}(\alpha^*\mathcal{E})\to \mathrm{Bl}_Y^D X \quad \textrm{and} \quad  \alpha^*E[D]=\mathbb{V}(\alpha^*(\mathcal{I}_D^{\vee}\otimes \mathcal{E}))\to \mathrm{Bl}_Y^D X.
 \end{equation}
 \item The morphisms $\mathrm{p}_2:(\Bl_{\hat{j}(Y)}^{E|_D} E)\times_X D\to D$ and $\mathrm{p}_2:(\Bl_{E|_Y}^{E|_D} E)\times_X D\to D$ factor through the regular  closed immersion $\iota_D:Z\to D$ and the associated $Z$-schemes are canonically isomorphic to the vector bundles 
  \begin{equation}
 \mathbb{V}((\mathcal{C}_{Z/D}\oplus (j\circ \iota_Y)^*\mathcal{E})\otimes \iota_D^*\mathcal{C}_{D/X}^{\vee})\to Z \; \textrm{and} \; \mathbb{V}(\mathcal{C}_{Z/D}\otimes \iota_D^*\mathcal{C}_{D/X}^{\vee} \oplus (j\circ \iota_Y)^*\mathcal{E})\to Z,
  \end{equation}
 where $\mathcal{C}_{Z/D}$ and $\mathcal{C}_{D/X}$ are the conormal sheaves of the regular immersions $\iota_D:Z\to D$ and $i:D\to X$ respectively.
\end{enumerate}
\xprop

\pf 
It is straightforward to verify that the schemes $$ \mathrm{pr}_2:(\mathrm{Bl}_{Y}^{D} X)\times_X E\to E 
\textrm{ and }\mathrm{pr}_2:(\mathrm{Bl}_{Y}^{D} X)\times_X E[D]\to E[D]$$ belong to the categories $\Sch_{E}^{E|_D\text{-reg}}$ and $\Sch_{E[D]}^{E[D]|_D\text{-reg}}$ respectively. It then follows from flat base change \cite[Lemma 2.7]{MRR20} that the canonical morphisms  $$ \mathrm{Bl}_{E|_Y}^{E|_D} E \to (\mathrm{Bl}_{Y}^{D} X)\times_X E \textrm{ and } \mathrm{Bl}_{E[D]|_Y}^{E[D]|_D} E[D] \to (\mathrm{Bl}_{Y}^{D} X)\times_X E[D]$$ are both isomorphisms. This implies in particular that the right hand square of the diagram is cartesian. It remains to verify that the $E$-schemes $\mathrm{Bl}_{\hat{j}(Y)}^{E|_D} E$ and $\mathrm{Bl}_{E[D]|_Y}^{E[D]|_D} E[D]$ are isomorphic. By definition, $\mathrm{Bl}_{\hat{j}(Y)}^{E|_D} E$ belongs to the category $\Sch _{E}^{E|_D\text{-reg}}$ and since $\hat{j}(Y)=E|_Y\cap s_{0,E}(X)$, 
we have $$\tau_0^{-1}(E|_D)\subset \tau_0^{-1}(\hat{j}(Y))\subset \tau_0^{-1}(E|_Y)\cap \tau_0^{-1}(s_{0,E}(X)).$$ By the respective universal properties of $\tau:\mathrm{Bl}_{E|_Y}^{E|_D} E\to E$ and $\sigma:E[D]=\mathrm{Bl}_{s_{0,E}(X)}^{E|_D} E\to E$ in Theorem-Definition \ref{Def-Multi-centered}, this implies that  $\tau_0:\mathrm{Bl}_{\hat{j}(Y)}^{E|_D} E\to E$ factors through unique morphisms of $E$-schemes $\xi:\mathrm{Bl}_{\hat{j}(Y)}^{E|_D} E\to \mathrm{Bl}_{E|_Y}^{E|_D} E$ and $\psi:\mathrm{Bl}_{\hat{j}(Y)}^{E|_D} E\to E[D]$ respectively. The universal property of the fiber product then implies in turn that $\tau_0$ factors through a unique morphism of $E$-schemes  $$\tilde{\tau}_0:\mathrm{Bl}_{\hat{j}(Y)}^{E|_D} E \to E[D]\times_E \mathrm{Bl}_{E|_Y}^{E|_D} E \cong E[D]\times_X \mathrm{Bl}_{Y}^{D} X\cong\mathrm{Bl}_{E[D]|_Y}^{E[D]|_D} E[D].$$  
As a scheme over $E$ via the composition $\theta:=\sigma\circ \mathrm{pr}_1=\tau\circ \mathrm{pr}_2$, 
$E[D]\times_X \mathrm{Bl}_{Y}^{D} X$ belongs to the category $\Sch_{E}^{E|_D\text{-reg}}$. Moreover, by combining the properties of $\sigma$ and $\tau$, we deduce that $$\theta^{-1}(E|_D)\subset \theta^{-1}(s_{0,E}(X))\cap \theta^{-1}(E|_Y)=\theta^{-1}(\hat{j}(Y)),$$
which implies, by the universal property of $\tau_0:\mathrm{Bl}_{\hat{j}(Y)}^{E|_D} E\to E$ that $\theta$ factors through a unique morphism of $E$-schemes $\tilde{\theta}:E[D]\times_E \mathrm{Bl}_{E|_Y}^{E|_D} E \to \mathrm{Bl}_{\hat{j}(Y)}^{E|_D} E$. Since $\sigma$, $\tau$ and $\tau_0$ all restrict to isomorphisms over the Zariski dense open subset $E\setminus E|_D$ of $E$, the same holds for $\tilde{\tau}_0$ and $\tilde{\theta}$, from which it follows that they are isomorphisms of $E$-schemes, inverse to each other. This completes the proof of assertion (i).   

Assertion (ii) follows
from assertion (i) and \cite[Lemma 2.4 ; Proposition 2.9 (3)]{MRR20} which, under our assumption on $Z=Y\times_X D$, asserts 
that $\mathrm{p}_2:(\mathrm{Bl}_Y^D X)\times _X D \to D$ factors through $Z$ and  that the induced morphism $(\mathrm{Bl}_Y^D X)\times _X D\to Z$ is isomorphic to the vector bundle $\mathbb{V}(\mathcal{C}_{Z/D} \otimes \iota_D^*\mathcal{C}_{D/X}^{\vee})\to Z$. 
\xpf

\pagestyle{fancy}
 \fancyhf{}
\fancyfoot[C]{\thepage}
 \fancyhead[CE]{{\small APPENDIX BY SYLVAIN BROCHARD, ADRIEN DUBOULOZ AND ARNAUD MAYEUX}}
\fancyhead[CO]{{\small IDEALS GENERATED BY MONOMIALS IN A REGULAR SEQUENCE}}
 \renewcommand{\headrulewidth}{0pt}
\section{Intersections, products, and sums of ideals generated by monomials in a regular sequence} \label{secap}
 \begin{center} {\bf \small by Sylvain Brochard, Adrien Dubouloz  and Arnaud Mayeux} \end{center}

\medskip

In this appendix, we prove that ideals generated by monomials in a regular sequence behave like monomial ideals in a polynomial ring with respect to intersections, sums and products. This might be well-known to experts, but we were unable to find a convenient reference. Before proceeding, let us fix some notations.

Let $x=(x_1,\dots, x_n)$ be a sequence of elements in a ring $A$. For $\alpha\in \N^n$, say $\alpha=(\alpha_1,\dots, \alpha_n)$, we denote by $x^\alpha$ the element
$ x^\alpha=x_1^{\alpha_1}\dots x_n^{\alpha_n}$.
The elements $x^\alpha$ for $\alpha\in \N^n$ are called the $x$-monomials of $A$, or simply the monomials if the sequence $x$ is understood. We denote by $\Mon_x(A)$, or just $\Mon(A)$, the set of monomials of $A$. In the remainder of this section, we fix a sequence $x$, and all monomials are monomials relatively to this sequence. We denote by $(x)$ the ideal $(x_1,\dots, x_n)$. 

If $I$ is an ideal of $A$, we denote by $\Mon(I)$ the set of monomials that belong to~$I$. In other words,
$\Mon(I)=I\cap \Mon(A)$. We denote by $\Delta_I$ the set of elements $\alpha\in \N^n$ such that $x^\alpha\in I$.

\defi
  An ideal $I$ of $A$ is a \emph{monomial ideal} if it is generated by monomials. 
\xdefi

There is a (partial) order on $\N^n$ given by the relation $(\alpha_1,\dots, \alpha_n)\leq (\beta_1,\dots, \beta_n)$ if and only if $\alpha_i\leq \beta_i$ for all $i$. If $\Delta$ is a subset of $\N^n$, we denote by $\Delta_{\textrm{min}}$ the set of minimal elements of $\Delta$, and by $\langle \Delta\rangle$ the ideal of $\N^n$ generated by $\Delta$.
\[
  \langle \Delta\rangle=\{\alpha+\beta \ |\ \alpha\in \Delta \textrm{ and }\beta\in \N^n\}.
\]
It is well-known that $\Delta_{\textrm{min}}$ is a finite set, and that for any $\alpha\in \Delta$, there is an element~$\beta\in \Delta_{\textrm{min}}$ such that $\beta\leq \alpha$.
We have the obvious relations $\langle \Delta_{\textrm{min}}\rangle=\langle \Delta\rangle$ and $\langle \Delta\rangle_{\textrm{min}}= \Delta_{\textrm{min}}$.

\prop
  We have the following properties for monomial ideals $I, J$ of $A$.
  \begin{enumerate}[label=(\arabic*)]
    \item $I$ is generated by $\Mon(I)$.
    \item $I$ is generated by the elements $x^\alpha$ for $\alpha\in (\Delta_I)_{\textrm{\rm min}}$.
    \item $I=J$ if and only if $\Mon(I)=\Mon(J)$.
    \item The ideals $I+J$ and $IJ$ are monomial.
  \end{enumerate}
\xprop
\begin{proof}
  Let $S$ be a set of monomials that generate $I$. The inclusion $S\subset \Mon(I)$ implies (1). Assertions (2) and (3) follow immediately. (4) is obvious.
\end{proof}

\defi
  Let $I$ be a monomial ideal of $A$. The \emph{support} of $I$ is the set $\supp(I)$ of indices $i\in \{1,\dots, n\}$ such that $\alpha_i\neq 0$ for some $\alpha\in (\Delta_I)_{\textrm{\rm min}}$.
\xdefi

\begin{lem}
\label{lem:monomes_et_comb_lin}
  Assume that the sequence $x$ is quasi-regular. Let $f\in A$ be an element of the form
$f=\sum_{i=1}^r a_if_i$,  
  where $f_1,\dots, f_r$ are monomials of $A$ that are distinct to each other, and the $a_i$ are elements of $A\smallsetminus (x)$. Let $g_1,\dots, g_\ell$ be other monomials. Assume that $f$ belongs to the ideal  $(g_1,\dots, g_\ell)$. Then each element $f_i$ is a multiple of some $g_{j}$.
\end{lem}
\begin{proof}
  Let $d$ be the minimal degree of the monomials $f_i$. Then $f\in (x)^d$. We can write $f=\sum b_jg_j$ for some $b_j\in A$. Let $t$ be the minimal degree of the monomials $g_j$. If $t<d$, we have $f\in (x)^d\subset (x)^{t+1}$, hence the relation $f=\sum b_jg_j$ projects to $\sum \ov{b_j}\ov{g_j}=0$ in $(x)^t/(x)^{t+1}$. Since the monomials of degree $t$ form an $A/(x)$-basis, it follows that $b_j\in (x)$. Hence, replacing the $g_j$ by some multiples if necessary, we may assume that $t\geq d$. Now, in $(x)^d/(x)^{d+1}$, we have $\sum_i \ov{a_i}\ov{f_i}=\sum_j \ov{b_j}\ov{g_j}$.
  Any monomial $f_i$ of degree $d$ on the left-hand side has a nonzero coefficient $\ov{a_i}$ (because $a_i\notin (x)$), hence it must appear on the right-hand side too. This means that each monomial $f_i$ of degree $d$ on the left-hand side is one of the $g_j$'s. Arguing by induction on $d$, the statement follows.
\end{proof}

\begin{cor}
\label{cor:monomes_somme_et_produit}
  Assume that the sequence $x$ is quasi-regular with $(x)\neq A$. Let $I, J$ be monomial ideals of $A$. Then:
  \begin{enumerate}[label=(\arabic*)]
    \item $\Mon(I+J)=\Mon(I)\cup \Mon(J)$
    \item $\Mon(IJ)=\Mon(I).\Mon(J)$
  \end{enumerate}
\end{cor}
\begin{proof}
  The inclusion $\Mon(I)\cup \Mon(J)\subset \Mon(I+J)$ is obvious. Conversely, let $f\in \Mon(I+J)$. Since $f\in I+J$, it is a linear combination of monomials of $I$ and monomials of $J$. By~\ref{lem:monomes_et_comb_lin}, it is a multiple of one of them. Hence $f\in \Mon(I)\cup \Mon(J)$ and this proves (1). The proof of (2) is very similar: the inclusion $\Mon(I)\Mon(J)\subset \Mon(IJ)$ is obvious. Conversely, let $f\in \Mon(IJ)$. Since $f\in IJ$, it is a linear combination of monomials of the form $gh$, where $g\in I$ and $h\in J$. By~\ref{lem:monomes_et_comb_lin}, it is a multiple of one of these products. Hence $f\in \Mon(I)\Mon(J)$.
\end{proof}

\begin{lem}
\label{lem:ecriture_standard}
  Assume that $A$ is a Noetherian local ring and that $(x)\neq A$. Let $I$ be a monomial ideal. Then any element  $f\in I$ can be written as a finite sum $f=\sum_i a_if_i$,
where the $f_i$ are pairwise distinct monomials of $I$, and $a_i\notin (x)$ for all $i$.
\end{lem}
\begin{proof}
   If $\Delta$ is a subset of $\Delta_I$, we denote by $I(\Delta)\subset I$ the monomial ideal of $A$ generated by the $x^\alpha$ for $\alpha\in \Delta$. Note that $I(\Delta_{\textrm{min}})=I(\Delta)$. For $f\in I$, let $E_f$ denote the set of subsets $\Delta\subset \Delta_I$ such that $f\in I(\Delta)$ and $\Delta=\Delta_{\textrm{min}}$.
  \[
    E_f=\{ \Delta\subset \Delta_I\ |\ f\in I(\Delta) \textrm{ and } \Delta=\Delta_{\textrm{min}}\}
  \]
Since $f\in I$, it can be written as a linear combination of monomials of $I$, hence $E_f$ is nonempty.
  If $\Delta, \Delta'\in E_f$, we say that $\Delta\leq \Delta'$ if $\langle \Delta\rangle\subset \langle \Delta'\rangle$. This is a partial order on $E_f$. Using Zorn's Lemma, we will prove that $E_f$ has a minimal element. For this we need to prove that $E_f$ is inductive. Let $S$ be a nonempty totally ordered subset of $E_f$. Let
$$\Delta_0=\left(
  \bigcap_{\Delta\in S} \langle \Delta\rangle
\right)_{\textrm{min}}.$$
It is clear that $\Delta_0=(\Delta_0)_{\textrm{min}}$, that $\Delta_0\subset \Delta_I$, and that $\langle\Delta_0\rangle\subset \langle\Delta\rangle$ for all $\Delta\in S$. Hence, to conclude that $E_f$ is inductive, it suffices to prove that $f\in I(\Delta_0)$. For each $\ell\in \N$, let $\Delta^{\ell}_0=\langle \Delta_0\cup \{\alpha\in \N^n\ |\ |\alpha|=\ell\}\rangle$. Note that the set  $\N^n\smallsetminus \Delta_0^\ell$ is a finite set. If $\alpha \notin \Delta_0^\ell$, we have in particular $\alpha\notin \langle\Delta_0\rangle$, hence there exists $\Delta\in S$ such that $\alpha\notin \langle\Delta\rangle$. Since $\N^n\smallsetminus \Delta_0^\ell$ is a finite set, and since~$S$ is totally ordered, there exists a small enough $\Delta\in S$ such that $\langle\Delta\rangle$ avoids all the elements $\alpha\notin \Delta_0^\ell$, in other words such that $\langle\Delta\rangle\subset \Delta_0^\ell$. But $f\in I(\Delta)$, hence $f\in I(\Delta_0^\ell)=I(\Delta_0)+(x)^\ell$. This proves that
\(
  f\in \bigcap_{\ell\geq 0}(I(\Delta_0)+(x)^\ell).
\)
By the Krull intersection theorem, the latter intersection is equal to $I(\Delta_0)$. This concludes the proof that $E_f$ is inductive.

Now, let $\Delta$ be a minimal element of $E_f$. Since $f\in I(\Delta)$, we can write $f=\sum_{\alpha\in \Delta}a_{\alpha}x^\alpha$ with $a_{\alpha}\in A$. If there exists $\alpha\in \Delta$ such that $a_{\alpha}\in (x)$, we can write $a_{\alpha}=\sum_{i=1}^n b_ix_i$. But then we can expand the above sum, and this proves that there exists $\Delta'\subset \Delta_I$ such that $f\in I(\Delta')$ and $\langle\Delta'\rangle\subsetneq \langle\Delta\rangle$, contradicting the minimality of $\Delta$. Hence we have $a_{\alpha}\notin (x)$ for all $\alpha\in \Delta$ and this concludes the proof.
\end{proof}

\prop
\label{prop:intersection_monomiale}
  Assume that $A$ is a Noetherian local ring and that the sequence $x$ is regular with $(x)\neq A$. Then we have the following properties for monomial ideals $I, J, K$ of~$A$.
  \begin{enumerate}[label=(\arabic*)]
    \item The ideal $I\cap J$ is monomial, and $\Mon(I\cap J)=\Mon(I)\cap \Mon(J)$.
    \item $(I+J)\cap K=(I\cap K)+(J\cap K)$
    \item If $I$ and $J$ have disjoint supports, then $I\cap J=IJ$.
  \end{enumerate}
\xprop
\begin{proof} Let $f\in I\cap J$. By~\ref{lem:ecriture_standard}, since $f\in I$, we can write it as a finite sum
$f=\sum_i a_if_i$, where the $f_i$ are pairwise distinct monomials of $I$, and $a_i\notin (x)$ for all $i$. By~\ref{lem:monomes_et_comb_lin}, since $f\in J$, we have $f_i\in J$, hence $f_i\in \Mon(I\cap J)$, and this proves that $I\cap J$ is monomial. The relation $\Mon(I\cap J)=\Mon(I)\cap \Mon(J)$ is obvious on the definition. 

Let us prove (2). Since the ideals on both sides of the equation are monomial, it suffices to prove that they have the same monomials. By (1) and~\ref{cor:monomes_somme_et_produit} we have:
\begin{align*}
  \Mon((I+J)\cap K) &= \Mon(I+J)\cap \Mon(K) \\
  &= (\Mon(I)\cup \Mon(J))\cap \Mon(K) \\
  &= (\Mon(I)\cap \Mon(K))\cup (\Mon(J)\cap \Mon(K)) \\
  &= \Mon(I\cap K)\cup \Mon(J\cap K) \\
  &= \Mon((I\cap K)+(J\cap K)).
\end{align*}

To prove (3), let $x^\alpha\in \Mon(I\cap J)$. Since $x^\alpha\in \Mon(I)$, we have $\alpha\in \Delta_I$, hence there exist $\beta\in (\Delta_I)_{\textrm{\rm min}}$ and $\beta'\in \N^n$ such that $\alpha=\beta+\beta'$. Similarly, there are $\gamma\in (\Delta_J)_{\textrm{\rm min}}$ and $\gamma'\in \N^n$ such that $\alpha=\gamma+\gamma'$. Let $\mu=(\mu_1,\dots, \mu_n)$ where $\mu_i=\beta'_i$ if $i$ is in the support of $I$, and $\mu_i=\gamma_i'$ otherwise. Since the supports of $I$ and $J$ are disjoint, in both cases we have $\alpha_i=\beta_i+\gamma_i+\mu_i$. Hence $\alpha=\beta+\gamma+\mu$, which proves that $x^\alpha\in \Mon(I).\Mon(J)$, as desired.
\end{proof}

\begin{cor} \label{coroap}
  Assume that $A$ is a Noetherian local ring and that the sequence $x$ is regular. Let $N_1,\dots, N_r, Q_1,\dots, Q_r$ and $Q$ be monomial ideals such that $\supp(N_i)$ is disjoint from $\supp(Q_i)$ and $\supp(Q)$ for each $i$. Then:
  \[
    \left(\sum_{i=1}^r N_iQ_i\right)\cap Q=\sum_{i=1}^rN_i(Q_i\cap Q).
  \]
\end{cor}
\begin{proof}
  By~\ref{prop:intersection_monomiale} (2), we have $(\sum_{i=1}^rN_iQ_i)\cap Q=\sum_{i=1}^r(N_iQ_i)\cap Q$. Since the supports of $N_i$ and $Q_i$ are disjoint, by~\ref{prop:intersection_monomiale} (3) we have $(N_iQ_i)\cap Q=N_i\cap Q_i\cap Q$. Since $\supp(Q_i\cap Q)\subset \supp(Q_i)\cup \supp(Q)$, it is disjoint from $\supp(N_i)$, hence applying~\ref{prop:intersection_monomiale}~(3) again we get $(N_iQ_i)\cap Q=N_i(Q_i\cap Q)$.
\end{proof}

To deal with intersections of monomial ideals whose supports are not disjoint, the following Proposition will be useful. 

\prop
\label{prop:intersection_inclusions}
 Let $I_1\subset I_2\subset \dots \subset I_n$ be ideals of $A$. Assume that each ideal $I_j$ is generated by a subsequence $s_j$ of $x$, such that for all $j\in\{1,\dots, n-1\}$ the sequence $s_j$ is a subsequence of $s_{j+1}$, and for all $j\geq 2$ the sequence $s_j$ is quasi-regular. Then
  \begin{enumerate}[label=(\arabic*)]
    \item For any two sequences of non-negative integers $(a_1,\dots, a_n)$ and $(b_1,\dots, b_n)$, we have
    \[
      (I_1^{a_1}\dots I_n^{a_n})\cap (I_1^{b_1}\dots I_n^{b_n})=I_1^{m_1}\cap \dots \cap I_n^{m_n},
    \]
    where the sequence $(m_i)$ is defined by $m_i=\max(\sum_{\ell=1}^ia_\ell, \sum_{\ell=1}^ib_\ell)$.
  \item For any nondecreasing sequence of non-negative integers $(m_1,\dots, m_n)$, we have 
  \[
    I_1^{m_1}\cap \dots \cap I_n^{m_n}=I_1^{m_1}I_2^{m_2-m_1}\dots I_n^{m_n-m_{n-1}}.
  \]
  \end{enumerate}
\xprop
\begin{proof}
  (2) Let $y\in I_1^{m_1}\cap \dots \cap I_n^{m_n}$. We prove by induction on $p\in \{1,\dots, n\}$ that $y\in  I_1^{m_1}I_2^{m_2-m_1}\dots I_p^{m_p-m_{p-1}}$. This is obvious for $p=1$. Assume that this holds for some $p<n$. Now we prove by induction on $t\in\{0,\dots, m_{p+1}-m_p\}$ that $y\in I_1^{m_1}I_2^{m_2-m_1}\dots I_p^{m_p-m_{p-1}}I_{p+1}^t$. This is obvious for $t=0$, and the induction step follows from Lemma~\ref{lem:cas_I_inclus_dans_J} applied with $I=I_1^{m_1}\dots I_p^{m_p-m_{p-1}}I_{p+1}^t$ and $J=I_{p+1}$.
  
  (1) By definition of $m_i$, the inclusion $\subset$ is obvious. To prove the converse, we use (2). Since $m_i\geq a_1+\dots +a_i$ for each $i$, we have
  \begin{align*}
    I_1^{m_1}\cap \dots \cap I_n^{m_n}&=I_1^{m_1}I_2^{m_2-m_1}\dots I_n^{m_n-m_{n-1}} \\
    &=I_1^{a_1}I_1^{m_1-a_1}I_2^{m_2-m_1}\dots I_n^{m_n-m_{n-1}} \\
    &\subset I_1^{a_1}I_2^{m_1-a_1}I_2^{m_2-m_1}\dots I_n^{m_n-m_{n-1}} \\
    &=I_1^{a_1}I_2^{a_2}I_2^{m_2-(a_1+a_2)}\dots I_n^{m_n-m_{n-1}} \\
    &\subset \dots \\
    &\subset I_1^{a_1}\dots I_k^{a_k}I_k^{m_k-(a_1+\dots+a_k)}I_{k+1}^{m_{k+1}-m_k}\dots I_n^{m_n-m_{n-1}} \\
    &\subset I_1^{a_1}\dots I_k^{a_k}I_{k+1}^{m_k-(a_1+\dots+a_k)}I_{k+1}^{m_{k+1}-m_k}\dots I_n^{m_n-m_{n-1}} \\
    &= I_1^{a_1}\dots I_k^{a_k}I_{k+1}^{a_{k+1}}I_{k+1}^{m_{k+1}-(a_1+\dots+a_{k+1})}\dots I_n^{m_n-m_{n-1}} \\
    &\subset \dots \\
    &\subset I_1^{a_1}\dots I_n^{a_n}.
  \end{align*}
  This finishes the proof.
\end{proof}

The Lemma below was used in the proof. Note that its assumptions are satisfied (with the integer $i=a_1+\dots+a_n$) if $J$ is generated by a quasi-regular sequence $x_1,\dots, x_r$, and if $I$ is the product $I_1^{a_1}\dots I_n^{a_n}$ where each ideal $I_k$ is generated by a subset of $\{x_1,\dots, x_r\}$. Indeed, in this case, $J^i/J^{i+1}$ admits an $(A/J)$-basis consisting of all the monomials of degree $i$ in $x_1,\dots, x_r$.

\begin{lem}
\label{lem:cas_I_inclus_dans_J}
  Let $I\subset J$ be two ideals with $I\subset J^i$ for some integer $i$. Suppose that $I$ has a generating set consisting of elements whose images in $J^{i}/J^{i+1}$ form a free family over $A/J$. Then $I\cap J^{i+1}=IJ$. 
\end{lem}
\begin{proof}
  The inclusion $IJ\subset I\cap J^{i+1}$ is obvious. Conversely, let $x\in I\cap J^{i+1}$. Let $(x_\ell)_\ell$ be a generating set of $I$ as in the statement. Then there exist elements $\lambda_\ell\in A$ such that $x=\sum \lambda_\ell x_\ell$. Since $x\in J^{i+1}$, we have $\sum \lambda_\ell x_\ell=0$ in $J^{i}/J^{i+1}$, but since the family of $x_\ell$ is free there, we deduce that $\lambda_\ell\in J$, hence $x\in IJ$.
\end{proof}

\exam
  Note that Proposition~\ref{prop:intersection_inclusions} is not true for ideals generated by monomials, even if the sequence $x$ is regular and the ring $A$ is local and Noetherian. For example, consider the ring $A=k\llbracket x,y,z,t\rrbracket$ and the ideals $I_1=(yz)$ and $I_2=(xy,yz,zt)$. Then $xyzt$ belongs to $I_1\cap I_2^2$ but not to $I_1I_2$.
\xexam

\exam
  The conclusion of~\ref{prop:intersection_inclusions} is not true either if we replace the quasi-regular assumption on $s_j$ with the weaker assumption that $I_j$ is strongly Lech-independent.
  For example, let $A=k\llbracket x,y,z\rrbracket/(x^2-zy^3)$, and take the ideals $I_1=(x)$ and $I_2=(x,y)$. We have $A/I_2\simeq k\llbracket z\rrbracket$. For any $i\geq 1$, $I_2^i/I_2^{i+1}$ is free of rank 2 over $A/I_2$, with basis $(xy^{i-1}, y^i)$, so $I_2$ is strongly Lech-independent. But we have $x^2=zy^3\in (I_1^2\cap I_2^3)\smallsetminus I_1^2I_2$, hence $(I_1^2\cap I_2^3)\neq I_1^2I_2$.
\xexam

\prop \label{AppPro:Tensor-equal-product} Let $I_1,\ldots, I_n$ be ideals of a Noetherian local ring $A$. Assume that each ideal $I_j$ is generated by a subsequence $s_j$ of a regular sequence $x$ in the maximal ideal of $A$ such that for all $j\neq j'$ in $\{1,\ldots, n\}$ the sequences $s_j$ and $s_{j'}$ are disjoint. Then for any sequence of non-negative integers $(m_1,\ldots,m_n)$ the canonical homomorphism of $A$-modules
$$ I_1^{m_1}\otimes_A \cdots \otimes_A I_n^{m_n}\to I_1^{m_1}\cdots I_n^{m_n}$$ 
is an isomorphism.
\xprop
\pf We proceed by induction on the number $n$ of ideals. Putting $J=I_1^{m_1}\cdots I_{n-1}^{m_{n-1}}$, the canonical homomorphism $\bigotimes_{j=1}^{n-1}I_j^{m_j}\to \prod_{j=1}^{n-1} I_j^{m_j}$ is an isomorphism by the induction hypothesis. The canonical homomorphism $\mu_{m_n}:J\otimes_A I_n^{m_n}\to JI_n^{m_n}$ is clearly surjective, and an isomorphism when $J=A$. We henceforth assume that $J\neq A$. Tensoring the exact sequence $0\to J\to A\to A/J\to 0$ by $I_n^{m_n}$ and using the fact that $\mathrm{Tor}_1^{A}(A,I_n^{m_n})=0$ gives a canonical isomorphism $\mathrm{Ker}(\mu_{m_n})\cong \mathrm{Tor}_1^A(A/J,I_n^{m_n})$. On the other hand, tensoring the short exact sequence $0\to I_n^{m_n}\to A\to A/I_n^{m_n}\to 0$ by $A/J$ gives a long exact sequence $$\cdots \to 0=\mathrm{Tor}_2^{A}(A/J,A)\to\mathrm{Tor}_2^{A}(A/J,A/I_n^{m_n}) \to  \mathrm{Tor}_1^{A}(A/J,I_n^{m_n})\to \mathrm{Tor}_1^{A}(A/J,A)=0\to \cdots$$ which canonically identifies in turn   $\mathrm{Ker}(\mu_{m_n})$ with $\mathrm{Tor}_2^{A}(A/J,A/I_n^{m_n})$. Since $I_n$ is generated by the regular subsequence $s_n$ of $x$ and a regular sequence is Koszul regular  \cite[\href{https://stacks.math.columbia.edu/tag/062F}{Tag 062F}]{stacks-project}, the Koszul complex $K_{\bullet}(A,s_n)$ of $s_n$ is a free resolution of $A/I_n$ and hence, we have $\mathrm{Tor}_k^A(A/J,A/I_n)=H_k(K_{\bullet}(A,s_n)\otimes_A A/J)$ for all $k$. 

Now we claim that the image $\overline{s}_n$ of $s_n$ by the quotient morphism $A\to A/J$ is a regular sequence in $A/J$. Writing  $s_n=(s_{n,1},\ldots, s_{n,r_n})$, we show first that the image $\overline{s}_{n,1}$ in $A/J$ is not a zero divisor. Assume that there exists $a \in A$ such that $as_{n,1}$ belongs to $J$. Since  $J$ and $(s_{n,1})$ are monomial ideals with disjoint supports, Proposition \ref{prop:intersection_monomiale}(3) implies that  $(s_{n,1})\cap J = (s_{n,1})J$, so that  $as_{n,1} = s_{n,1}b$ for some $b\in J$. Since $s_{n,1}$ is an element of the regular sequence $x$ and $A$ is local, $s_{n,1}$ is a regular element of $A$. It follows that $a=b\in J$, and hence that $\overline{s}_{n,1}$ is a regular element of $A/J$.  We conclude by induction on the number of elements $r_n$ of the sequence $s_n$ by replacing $A$ by $A/(s_{n,1})$, which is again local and Noetherian, the regular sequence $x$ by the image of the sequence $x \setminus \{s_{n,1}\}$  in $A/s_{n,1}$, which is again a regular sequence contained in the maximal ideal of $A/(s_{n,1})$, the subsequences $s_i$ by their respective images, the sequence $s_n$ by the image of the sequence $s_n\setminus \{s_{n,1}\}$ and $J$ by its image.

The fact that the image $\overline{s}_n$ of $s_n$ is a regular sequence in $A/J$ implies that the complex  $K_{\bullet}(A,s_n)\otimes_A A/J$ is isomorphic to the Koszul complex $K_{\bullet}(A/J,\overline{s}_n)$ and that the latter is exact.  It follows that $\mathrm{Tor}_k^A(A/J,A/I_n)=0$ for every $k>0$, whence that $\mathrm{Ker}(\mu_1)=\mathrm{Tor}_2^A(A/J,A/I_n)=0$. Now assume by induction that $\mathrm{Ker}(\mu_m)=0$ for every $1\leq m < m_n$. Tensoring the exact sequence $0\to I_n^{m_n-1}/I_n^{m_n}\to A/I_n^{m_n}\to A/I_n^{m_n-1}\to 0$ by $A/J$ gives the long exact sequence 
$$ \cdots \to \mathrm{Tor}_2^A(A/J,I_n^{m_n-1}/I_n^{m_n})\to \mathrm{Tor}_2^A(A/J,A/I_n^{m_n}) \to \mathrm{Tor}_2^A(A/J,A/I_n^{m_n-1})\to \cdots $$ 
Since $I_n$ is generated by the regular sequence $s_n$, $I_n^{m_n-1}/I_n^{m_n}$ is a free $A/I_n$-module, isomorphic to the $(m_n-1)$-th symmetric module of $I_n/I_n^2$, and since $\mathrm{Tor}_2^A(A/J,A/I_n)=0$ by the previous observation, it follows that $\mathrm{Tor}_2^A(A/J, I_n^{m_n-1}/I_n^{m_n})=0$ as well. Since by the induction hypothesis $\mathrm{Ker}(\mu_{m_n-1})=\mathrm{Tor}_2^A(A/J,A/I_n^{m_n-1})=0$, this implies that $\mathrm{Tor}_2^A(A/J,A/I_n^{m_n})=\mathrm{Ker}(\mu_{m_n})=0$ and completes the proof.  
\xpf 

\medskip

We conclude this appendix with a result that can provide formulas like $I\cap J=IJ$ in other situations. This variant is not used in the main part of the article. In particular it contains a slight generalization of~\cite[(16.9.13.5)]{Gr67}: in \cite{Gr67} it is proved that $I^\ell\cap J=I^\ell J$ for any integer $\ell\geq 0$, under the assumption that $I$ is generated by a sequence that is regular in $A/J$, with $A$ Noetherian and $I$ contained in the Jacobson radical. Recall that, following Meng~\cite{Me21}, an ideal $I$ in a ring~$A$ is called strongly Lech-independent if $I^i/I^{i+1}$ is free over $A/I$ for any $i\geq 1$. It is clear on the definition that a quasi-regular sequence generates a strongly Lech-independent ideal. But the latter condition is actually significantly weaker. For example, any maximal ideal is clearly strongly Lech-independent, but is not always generated by a quasi-regular sequence.

\prop
\label{prop:intersection_cas_H1_regulier}
  Let $A$ be a ring. Let $p\leq q$ be two natural integers. Let $(f_1,\dots, f_q)$ be a sequence of elements of $A$. Consider the ideals $I=(f_1,\dots, f_p)$ and $J=(f_{p+1},\dots, f_q)$.
  \begin{enumerate}[label=(\arabic*)]
    \item If the sequence $(f_1,\dots, f_q)$ is $H_1$-regular, then the sequence $(\ov f_1, \dots, \ov f_p)$ in $A/J$ is $H_1$-regular.  \cite[\href{https://stacks.math.columbia.edu/tag/068L}{Tag 068L}]{stacks-project}
    \item If the sequence $(\ov f_1, \dots, \ov f_p)$ in $A/J$ is $H_1$-regular, then $I\cap J=IJ$. \cite[\href{https://stacks.math.columbia.edu/tag/0665}{Tag 0665}]{stacks-project}
    \item If $I$ is strongly Lech-independent, then for any ideal $K$ such that $I\cap K=IK$, we have $I^\ell\cap K=I^\ell K$ for any integer $\ell\geq 0$.
    \item If both $I$ and $J$ are strongly Lech-independent, and if $I\cap J=IJ$, then for all $\alpha, \beta \in \N$:
    \[
      I^\alpha\cap J^\beta =I^\alpha J^\beta.
    \]
  \end{enumerate}
\xprop
\begin{proof}
(1) Let $\sum_{i=1}^p \ov\lambda_i\ov f_i=0$ be a relation in $A/J$. It can be lifted to a relation $\sum_{i=1}^q \lambda_i f_i=0$ in $A$. Since the sequence $(f_1,\dots, f_q)$ is $H_1$-regular, there exists a family of elements $(\mu_{j,k})$ indexed by the pairs $(j,k)$ with $1\leq j<k \leq q$, such that for all $i$,
\[
  \lambda_i=\sum_{j=1}^{i-1}f_j\mu_{j,i}-\sum_{k=i+1}^q f_k\mu_{i,k}.
\]
By projecting this relation in $A/J$, we obtain the desired result.

(2) Let $x\in I\cap J$. Since $x\in I$, there exist elements $\lambda_1,\dots, \lambda_p$ such that 
$x=\sum_{i=1}^p f_i\lambda_i$. By projecting this relation in $A/J$, we get $\sum \ov f_i\ov\lambda_i=0$, in other words the family $(\ov \lambda_1,\dots, \ov \lambda_p)$ is in the kernel of the differential $d$ of the Koszul complex. Since the sequence $(\ov f_1, \dots, \ov f_p)$ is $H_1$-regular, $(\ov \lambda_1,\dots, \ov \lambda_p)$ is in the image of $d$, so there exists a family of elements $(\mu_{j,k})$ indexed by the pairs $(j,k)$ with $j<k$, and elements $y_i\in J$, such that for all $i$,
\[
  \lambda_i=\sum_{j=1}^{i-1}f_j\mu_{j,i}-\sum_{k=i+1}^p f_k\mu_{i,k}+y_i.
\]
We deduce that
\begin{align*}
  x=\sum_{i=1}^p f_i \lambda_i = \sum_{1\leq j<i\leq p}f_if_j\mu_{j,i}-\sum_{1\leq i<k\leq p}f_if_k\mu_{i,k}
  +\sum_{i=1}^p f_iy_i
  = \sum_{i=1}^p f_iy_i\in IJ.
\end{align*}

(3) The equality is proved by induction on $\ell$. It is obvious for $\ell=0$ and it is our assumption for $\ell=1$. Suppose $I^\ell\cap K=I^\ell K$ for some integer $\ell\geq 0$. Let $x$ be an element of $I^{\ell+1}\cap K$. By the inductive hypothesis, we know that $x\in I^{\ell}K$. Let $(g_i)$ be a family of elements that form an $A/I$-basis of $I^\ell/I^{\ell+1}$. Then there exist a finite number of elements $\lambda_i\in K$, and a $y\in I^{\ell+1}K$ such that
 \(
   x=y+\sum_i \lambda_ig_i.
 \)
Moreover, since $x\in I^{\ell+1}$, we have $\sum_i\lambda_ig_i=0$ in $I^\ell/I^{\ell+1}$. By the strong Lech-independence of $I$, this implies that $\lambda_i\in I$ for all $i$. We deduce that $\lambda_i\in I\cap K=IK$, which proves that $x\in I^{\ell+1}K$.

(4) From (3) applied with $K=J$, we deduce that $I^\alpha\cap J=I^\alpha J$ for all $\alpha\in \N$. Applying (3) again with the strongly independent ideal $J$ and $K=I^\alpha$, we deduce the desired result.
\end{proof}

 \end{document}